\documentclass[12pt]{amsart}
\usepackage{amsfonts, amsthm, amssymb, amsmath, stmaryrd}
\usepackage{mathrsfs,array}
\usepackage{eucal,color,times,enumerate,accents}
\usepackage[all]{xy}
\usepackage{url}
\usepackage{enumitem}
\usepackage[pdftex,bookmarksnumbered,bookmarksopen]{hyperref}
\hypersetup{
  colorlinks,
  citecolor=black,
  linkcolor=black,
  urlcolor=black}
  
\usepackage{tikz-cd}
\usetikzlibrary{cd}
\usepackage{leftidx}
\usetikzlibrary{positioning}
\usepackage{dsfont}
\usepackage{tikz}
\usetikzlibrary{matrix,arrows,decorations.pathmorphing,positioning}

\usepackage[utf8]{inputenc}
\usepackage[T1]{fontenc}

\newcommand\blfootnote[1]{%
\begingroup
\renewcommand\thefootnote{}\footnote{#1}%
\addtocounter{footnote}{-1}%
\endgroup
}
  
  \newcommand{\Addresses}{{
  \bigskip
  \footnotesize

\textsc{I.H.E.S., Université Paris-Saclay, CNRS, Laboratoire Alexandre Grothendieck. 35 Route de Chartres, 91440 Bures-sur-Yvette (France)}\par\nopagebreak
  \textit{E-mail address}, G.~Baldi: \texttt{baldi@ihes.fr} 
  
\textsc{I.H.E.S., Université Paris-Saclay, CNRS, Laboratoire Alexandre Grothendieck. 35 Route de Chartres, 91440 Bures-sur-Yvette (France)}\par\nopagebreak
  \textit{E-mail address}, E.~Ullmo: \texttt{ullmo@ihes.fr}
}}

\usepackage{bbm}

\usepackage{leftidx}

\usepackage[bottom=3cm, top=3cm, left=2.5cm, right=2.5cm]{geometry}

\input xy
\xyoption{all}

\addtocontents{toc}{\setcounter{tocdepth}{1}}

\theoremstyle{plain}
\newtheorem{thm}{Theorem}[subsection]

\newtheorem{conj}[thm]{Conjecture}
\newtheorem{lemma}[thm]{Lemma}

\newtheorem{prop}[thm]{Proposition}
\newtheorem{cor}[thm]{Corollary}
\theoremstyle{definition}
\newtheorem{defi}[thm]{Definition}

\newtheorem{rmk}[thm]{Remark}

\theoremstyle{remark}
\numberwithin{equation}{subsection}

\newcommand{\DT}{\mathbb{S}}
\newcommand{\Bb}{\mathbb{B}}

\newcommand{\an}{\operatorname{an}}
\newcommand{\Zar}{\operatorname{Zar}}

\newcommand{\PU}{\operatorname{PU}}
\newcommand{\SU}{\operatorname{SU}}
\newcommand{\SO}{\operatorname{SO}}

\newcommand{\ad}{\operatorname{ad}}
\newcommand{\der}{\operatorname{der}}
\newcommand{\pr}{\operatorname{pr}}

\newcommand{\MT}{\mathbf{MT}}

\newcommand{\Ad}{\operatorname{Ad}}

\DeclareMathOperator{\End}{End}
\DeclareMathOperator{\Hom}{Hom}

\DeclareMathOperator{\Res}{Res}

\DeclareMathOperator{\im}{Im}

\DeclareMathOperator{\Lie}{Lie}
\DeclareMathOperator{\tr}{tr}

\DeclareMathOperator{\Gl}{GL}
\DeclareMathOperator{\Sl}{SL}

\newcommand{\Gm}{\mathbb{G}_m}

\newcommand{\Aut}{\operatorname{Aut}}

\newcommand{\Pu}{\operatorname{PU}}
\newcommand{\codim}{\operatorname{codim}}
\newcommand{\Comm}{\operatorname{Comm}}

\newcommand{\id}{\operatorname{Id}}

\newcommand{\N}{\mathbb{N}}

\newcommand{\Z}{\mathbb{Z}}
\newcommand{\Q}{\mathbb{Q}}

\newcommand{\R}{\mathbb{R}}

\newcommand{\Oo}{\mathcal{O}}

\newcommand{\D}{\mathcal{D}}

\newcommand{\F}{\mathcal{F}}

\newcommand{\Hh}{\mathbb{H}}

\newcommand{\C}{\mathbb{C}}

\newcommand{\Qbar}{\overline{\mathbb{Q}}}

\usepackage{microtype}

\begin{document}

\newcommand{\adjunction}[4]{\xymatrix@1{#1{\ } \ar@<-0.3ex>[r]_{ {\scriptstyle #2}} & {\ } #3 \ar@<-0.3ex>[l]_{ {\scriptstyle #4}}}}

\title{Special subvarieties of non--arithmetic ball quotients and Hodge Theory}\date{\today}
\author{GREGORIO BALDI and EMMANUEL ULLMO}

\begin{abstract}
Let $\Gamma \subset \PU(1,n)$ be a lattice and $S_\Gamma$ be the associated ball quotient. We prove that, if $S_\Gamma$ contains infinitely many maximal complex totally geodesic subvarieties, then $\Gamma$ is arithmetic. We also prove an Ax-Schanuel Conjecture for $S_\Gamma$, similar to the one recently proven by Mok, Pila and Tsimerman. One of the main ingredients in the proofs is to realise $S_\Gamma$ inside a period domain for polarised integral variations of Hodge structure and interpret totally geodesic subvarieties as unlikely intersections.
\end{abstract}
\maketitle

\tableofcontents

\section{Introduction}
\subsection{Motivation}
The study of lattices of semisimple Lie groups $G$ is a field rich in open questions and conjectures. Complex hyperbolic lattices and their finite dimensional representations are certainly far from being understood. A lattice $\Gamma \subset G$ is \emph{archimedean superrigid} if for any simple noncompact Lie group $G'$ with trivial centre, every homomorphism $\Gamma \to G'$ with Zariski dense image extends to a homomorphism $G\to G'$. Thanks to the work of Margulis \cite{margulisbook}, Corlette \cite{MR1147961} and Gromov--Schoen \cite{MR1215595}, all lattices in simple Lie groups are archimedean superrigid unless $G$ is isogenous to either $\SO(1,n)$ or $\PU(1,n)$ for some $n\geq 1$. Real hyperbolic lattices are known to be softer and more flexible than their complex counterpart and non-arithmetic lattices in $\SO(1,n)$ can be constructed for every $n$ \cite{MR932135}. Non-arithmetic complex hyperbolic lattices have been found only in $\PU(1,n)$, for $n=1,2,3$ \cite{MR586876, MR849651}.

In both cases one can consider the quotient by $\Gamma$ of the symmetric space $X$ associated to $G$. In the complex hyperbolic case we obtain a ball quotient $S_\Gamma= \Gamma \backslash X$ which has a natural structure of a quasi-projective variety, as proven by Baily-Borel \cite{MR0216035} in the arithmetic case, and by Mok \cite{MR849651} in general. By the commensurability criterion for arithmeticity of Margulis \cite{margulisbook} we can decide whether $\Gamma$ is arithmetic or not by looking at \emph{modular correspondences} in the product $S_\Gamma \times S_\Gamma$. That is $\Gamma$ is arithmetic if and only if $ S_\Gamma$ admits infinitely many totally geodesic correspondences. In this paper, by \emph{totally geodesic subvarieties} we always mean complex subvarieties of $S_\Gamma$ of dimension $>0$, whose smooth locus is totally geodesic with respect to the canonical K\"{a}hler metric. Even if $\Gamma$ is arithmetic, $S_\Gamma$ may have no strict totally geodesic subvarieties. However, the existence of countably many Hecke correspondences implies that, if $S_\Gamma$ contains one totally geodesic subvariety, then it contains countably many of such. The aim of this paper is to study totally geodesic subvarieties of $S_\Gamma$ from multiple point of views, ultimately explaining how often and why they appear.
\subsection{Main results}
Let $G$ be $\PU(1,n)$, for some $n>1$, and $\Gamma \subset G$ be a lattice. Let $X$ be the Hermitian symmetric space associated to $G$ and $S_\Gamma$ be the quotient of $X$ by $\Gamma$. We first look at maximal totally geodesic subvarieties of $S_\Gamma$, that is totally geodesic subvarieties that are not strictly contained in any totally geodesic subvariety different from $S_\Gamma$. Our first main result is the following.
\begin{thm}\label{thmbader}
If $\Gamma \subset G$ is non-arithmetic, then $S_\Gamma$ contains only finitely many maximal totally geodesic subvarieties. 
\end{thm}
The problem was originally proposed informally by Reid and, independently, by McMullen for real hyperbolic lattices \cite[Question 7.6]{delp2015problems}, \cite[Question 8.2]{MR3247048}. For $\SO(1,n)$ this was recently proven by Bader, Fisher, Miller and Stover \cite[Theorem 1.1]{2019arXiv190308467B} and, for closed hyperbolic 3-manifolds, by Margulis and Mohammadi \cite[Theorem 1.1]{margulis2019arithmeticity}. Around the same time Bader, Fisher, Miller and Stover annouced also a proof of the finiteness of the maximal totally geodesic (both real and complex) subvarieties of a ball quotient. The proof has then appeared in \cite{bader2020arithmeticity}. Such approaches build on superrigidity theorems and use results on equidistribution from homogeneous dynamics. As we explain in Section \ref{strategy}, we are able to interpret the problem as a phenomenon of \emph{unlikely intersections} inside a period domain for polarised $\Z$-variations of Hodge structure (VHS from now on) and deduce the finiteness from a strong Ax-Schanuel theorem established by Bakker and Tsimerman \cite{MR3958791}. Indeed we will see that Theorem \ref{thmbader} is predicted by a conjecture of Klingler \cite[Conjecture 1.9]{klingler2017hodge}, which was our main motivation for studying such finiteness statements. The main novelty is the use of $\Z$-Hodge theory, rather than $\R$-Hodge theory. As a by-product of our strategy, in Section \ref{sectionmarguliscomm}, we give a new proof of the Margulis commensurability criterion for arithmeticity for complex hyperbolic lattices in $\PU(1,n)$, at least for $n>1$ (without using any superrigidity techniques). Moreover our approach gives an explicit description of the totally geodesic subvarieties in terms of the Hodge locus of a $\Z$-VHS, giving for example Corollary \ref{cor1} as a simple application.

Our second main result looks at totally geodesic subvarieties from the functional transcendence point of view, in the sense of \cite{MR3821177}. Whenever a complex algebraic variety $S$ has a semi-algebraic universal cover $\pi : X \to S$, one can formulate Ax--Schanuel and Zilber--Pink type conjectures. For more details see \cite[Section 2.2]{MR3728310} and \cite{MR3821177}. For example an \emph{abstract Ax--Lindemann--Weierstrass} would assert the following. Let $V$ be a semialgebraic subvariety of $X$, and let $S'$ be the Zariski closure of $\pi (V)$. Then $S'$ is \emph{bi-algebraic}, that is the irreducible components of $\pi^{-1}(S')$ are semialgebraic in $X$. 

As in the case of Shimura varieties, we prove that totally geodesic subvarieties are the same as bi-algebraic subvarieties and have a natural group theoretical description in terms of \emph{real sub-Shimura datum}, see Definition \ref{defsubshimura}. We prove the non-arithmetic Ax-Schanuel conjecture, generalising the Ax-Schanuel conjecture for Shimura varieties \cite{MR3502100, as} to non-arithmetic ball quotients.
\begin{thm}\label{nonarithasconjfinal}
Let $W \subset X \times  S_\Gamma$ be an algebraic subvariety and $\Pi \subset X \times S_\Gamma$ be the graph of $\pi : X \to S_\Gamma$. Let $U$ be an irreducible component of $W \cap \Pi$ such that
\begin{displaymath}
\codim U < \codim W + \codim \Pi,
\end{displaymath}
the codimension being in $X \times S_\Gamma$ or, equivalently,
\begin{displaymath}
\dim W < \dim U + \dim S_\Gamma.
\end{displaymath}
If the projection of $U$ to $S_\Gamma$ is not zero dimensional, then it is contained in a strict totally geodesic subvariety of $S_\Gamma$.
\end{thm}
Mok \cite{mokalw} has developed a different perspective on functional transcendence problems for ball quotients, and using methods of several complex variables, algebraic geometry and K\"{a}hler geometry recently managed to prove a form of Ax-Lindemann-Weierstrass, which follows from the above theorem. See also the earlier work of Chan and Mok \cite{chan2018asymptotic} regarding compact ball quotients.

Finally notice that both statements depend only on the commensurability class of $\Gamma$. Therefore, we may and do replace $\Gamma$ by a finite index subgroup to assume that all lattices we consider are torsion free.
\subsection{Strategy}\label{strategy}
For both Theorem \ref{thmbader} and \ref{nonarithasconjfinal}, the starting point is to embed $S_\Gamma$ in a period domain for polarised integral VHS. Let $n$ be an integer $>1$, and $G=\PU(1,n)$. We start with the key observation that lattices in $G$ satisfy a form of \emph{infinitesimal rigidity} (this is known to fail when $n=1$). Namely Garland and Raghunathan \cite{MR0267041}, building on the work initiated by Calabi-Vesentini and Weil \cite{MR0111058, MR137792, MR137793}, proved that the first Eilenberg--MacLane cohomology group of a lattice $\Gamma$ in $G$ with respect to the adjoint representation is trivial. Therefore, thanks to the work of Simpson \cite{MR1179076}, $\Gamma$ determines a totally real number field $K$ and a $K$-form of $G$, which we denote by $\mathbf{G}$. Since $X$ is a Hermitian symmetric space, $S_\Gamma$ supports a natural polarised variation of $\C$-Hodge structure which we denote by $\mathbb{V}$. Using a recent work of Esnault and Groechenig \cite{MR3874695}, and Simpson's theory \cite{MR1179076}, we prove.
\begin{thm}\label{strategy1}
For every $\gamma \in \Gamma$, the trace of $\Ad (\gamma)$ lies in the ring of integers of a totally real number field $K$. As a consequence, up to conjugation by $G$, $\Gamma$ lies in $\mathbf{G}(\Oo_K)$. Moreover $\mathbb{V}$ induces a $\Z$-variation of Hodge structure $\widehat{\mathbb{V}}$ on $S_\Gamma$.
\end{thm}
Let $\widehat{\mathbf{G}}/\Q$ be the generic Mumford--Tate group of $\widehat{\mathbb{V}}$. By the theory of period domains and period maps of Griffiths \cite{MR2918237}, $\widehat{\mathbb{V}}$ induces a commutative diagram in the complex analytic category
\begin{center}
\begin{tikzpicture}[scale=2]
\node (A) at (-1,1) {$X$};
\node (B) at (1,1) {$D=D_{\widehat{G}}$};
\node (C) at (-1,0) {${S_\Gamma}^{\an}$};
\node (D) at (1,0) {$\widehat{\mathbf{G}}(\Z)\backslash D$};
\path[right hook->,font=\scriptsize]
(A) edge node[above]{$\tilde{\psi}$} (B);
\path[->,font=\scriptsize,>=angle 90]
(C) edge node[above]{$\psi$} (D)
(A) edge node[right]{$\pi$} (C)
(B) edge node[right]{$\pi_\Z$} (D);
\end{tikzpicture},
\end{center}
where $\psi : {S_\Gamma}^{\an} \to \widehat{\mathbf{G}}(\Z)\backslash D$ is the period map associated to $\widehat{\mathbb{V}}$. Corollary \ref{gammadense} shows that $\widehat{\mathbf{G}}$ is the Weil restriction from $K$ to $\Q$ of $\mathbf{G}$. Such a map $\psi$ generalises the theory of modular embeddings of triangle groups \cite{MR931211, modularemb1} and Deligne--Mostow lattices \cite{modularemb2}.

From a functional transcendence point of view, there are now two \emph{bi-algebraic structures} on $S_\Gamma$ (in the sense of \cite[Definition 7.2]{klingler2017hodge}): on the one hand, the one coming from the fact that its universal covering $X$ is semialgebraic and on the other, the one coming from the fact that $S_\Gamma$ supports a polarised integral VHS. Depending on the special structure we choose, we have two, a priori different, definitions of \emph{bi-algebraic subvarieties}. For example a totally geodesic subvariety is bi-algebraic with respect to $\pi$, that is, it is the image of algebraic subvariety of $X$ along $\pi$. More precisely, we will see that totally geodesic subvarieties are $\Gamma$-\emph{special}, that is sub-locally hermitian spaces of $S_\Gamma$. From a Hodge theoretical point of view it is natural to consider the image in $S_\Gamma$ of some algebraic subvarieties of $D$. Namely let $D'$ be a Mumford--Tate sub-domain of $D$, then any analytic irreducible component of $\psi^{-1} (\pi_\Z(D'))$ is an algebraic subvariety of $S_\Gamma$ by a famous theorem of Cattani, Deligne and Kaplan \cite{MR1273413} (for the fact that $\psi^{-1} (\pi_\Z(D'))$ contains finitely many connected components, see also the work of Bakker, Klingler and Tsimerman \cite{bakker2018tame}). We refer to such irreducible algebraic subvarieties $\psi^{-1} (\pi_\Z(D'))^0$ as $\Z$-\emph{special subvarieties}. We remark here that the definition of $\Z$-special subvarieties is non-trivial even for zero dimensional subvarieties, whereas every point of $S_\Gamma$ could be regarded as both a totally geodesic and a $\Gamma$-special subvariety. We will come back to this in Section \ref{pointsection}. Until then it is understood that totally geodesic, $\Z$- and $\Gamma$-special subvarieties are of dimension $>0$.

What is the relation between the two special/bi-algebraic structures? Since $\tilde{\psi}$ is just a holomorphic map (rather than, a priori, a totally geodesic immersion), such a question is far from being trivial. To prove Theorems \ref{thmbader} and \ref{nonarithasconjfinal} we need the following.
\begin{thm}\label{mainthm}
The totally geodesic subvarieties of $S_\Gamma$ are precisely the $\Z$-special ones, equivalently the $\Gamma$-special ones.
\end{thm}
Theorem \ref{mainthm} reduces Theorem \ref{thmbader} to counting the $\Z$-special subvarieties in $S_\Gamma$ and we will see that this is actually an \emph{unlikely intersection problem} (see Section \ref{sketch} for a sketch of the proof of this fact). A similar strategy appeared also in the work of Wolfart \cite{MR931211}, where certain Riemann surfaces $C$, associated to non-arithmetic lattices, are embedded in Shimura varieties and the André--Oort conjecture is used to deduce that $C$ contains only finitely many \emph{CM-points}, see also \cite{MR1973057}, where Wolfart's programme is completed.

Theorem \ref{mainthm} is crucial to relate our Ax--Schanuel conjecture (Theorem \ref{nonarithasconjfinal}) to the Ax--Schanuel known in Hodge theory by the work of Bakker and Tsimerman \cite{MR3958791}, applied to the pair $(S_\Gamma, \widehat{\mathbb{V}})$. Finally, in Section \ref{sectioncor}, we prove some consequences of Theorem \ref{mainthm} which may be of independent interest. 
\begin{cor}\label{cor1}
Let $H$ be a subgroup of $G=\PU(1,n)$ of hermitian type. If $\Gamma \cap H$ is Zariski dense in $H$, then $\Gamma \cap H$ is a lattice in $H$.
\end{cor}
The second application is about \emph{non-arithmetic monodromy groups}, as first studied by Nori \cite{MR832040}. After Sarnak \cite{MR3220897}, a subgroup of a lattice $\Delta$ of a real algebraic group $R$ is called \emph{thin subgroup} if it has infinite-index in $\Delta$ and it is Zariski dense in $R$. A link between non-arithmetic lattices and thin subgroups was also noticed by Sarnak, see indeed \cite[Section 3.2]{MR3220897}. We present a geometric way of constructing many thin subgroups, as long as a non-arithmetic lattice $\Gamma$ is given.
\begin{cor}\label{cor2}
Let $\Gamma$ be a non-arithmetic lattice in $G=\PU(1,n)$, $n>1$. Let $W$ be an irreducible subvariety of $S_\Gamma$ which is not contained in any of the finitely many maximal totally geodesic subvarieties of $S_\Gamma$. The image of 
\begin{displaymath}
\pi_1(W)\to \pi_1(S_\Gamma) \cong \Gamma
\end{displaymath}
gives rise to a thin subgroup of $\widehat{\mathbf{G}}(\Z)$, where $\widehat{\mathbf{G}}$ denotes the Weil restriction from $K$, the (adjoint) trace field of $\Gamma$, to $\Q$ of $\mathbf{G}$ (the $K$-form of $G$ determined by $\Gamma$).
\end{cor}

\subsection{Sketch of the proof of Theorem \ref{thmbader} (in the simplest case)}\label{sketch}
We conclude the introduction presenting a sketch of the proof of Theorem \ref{thmbader} in the case where $\Gamma$ is a lattice in $\PU(1,2)$ and the (adjoint) trace field $K$ of $\Gamma$ has degree two over $\Q$. We hope this section can clarify the relationships between Hodge theory, functional transcendence and \emph{o-minimality} we employ in the sequel. The case of dimension two is particularly interesting since the maximality condition of Theorem \ref{thmbader} is automatically satisfied, and most of known non-arithmetic lattices were found in $\PU(1,2)$, see indeed Section \ref{known}. Our main theorem reads as follows.
\begin{thm}\label{thmfinal}
If $\Gamma \subset G= \PU(1,2)$ is non-arithmetic, then $S_\Gamma= \Gamma \backslash \mathbb{B}^2$ contains only finitely many one dimensional (complex) totally geodesic subvarieties. 
\end{thm}
\begin{proof}[Sketch of the proof]
If $S_\Gamma$ contains no totally geodesic subvarieties, there is nothing to prove. Assume there is a totally geodesic subvariety $W$ of $S_\Gamma$, say associated to a triplet $(H=\PU(1,1),X_H=\mathbb{B}^1, \Gamma_H=\Gamma \cap H)$, such that $\Gamma_H$ is a lattice in $H$. Let $K$ be the field generated by the traces of $\Gamma$, under the adjoint representation. Thanks to Theorem \ref{strategy1}, it is a totally real number field and we assume, for simplicity, that $[K:\Q]=2$. Let $\mathbf{G}$ be the $K$-form of $G$ determined by $\Gamma$ (resp. $\mathbf{H}$ for the $K$-form of $H$), and write $\widehat{\mathbf{G}}$ for the Weil restriction from $K$ to $\Q$ of $\mathbf{G}$ (resp. $\widehat{\mathbf{H}}$). Theorem \ref{mainthm} shows that the totally geodesic subvariety $W$ is an irreducible component of an intersection:
\begin{displaymath}
\psi(W)=  {\psi(S_\Gamma) \cap \widehat{\mathbf{H}}(\mathbb{Z}) \backslash D_{\widehat{H}}}^0,
\end{displaymath}
where $\widehat{\mathbf{H}}(\mathbb{Z}) \backslash D_{\widehat{H}}$ is a subperiod domain of $\widehat{\mathbf{G}}(\Z)\backslash D$. The only possibilities for the dimension of the latter space are 4 or 5, depending on whether it is a Shimura variety or not. Observe that, if $\Gamma$ is non-arithmetic, by the Mostow–Vinberg arithmeticity criterion (Theorem \ref{arithcrit}) $D$ is a homogeneous space under $\widehat{\mathbf{G}}(\R)=G \times G$ and 
\begin{displaymath}
\codim_{\widehat{\mathbf{G}}(\Z)\backslash D} \psi(S_\Gamma) = \dim \widehat{\mathbf{G}}(\Z)\backslash D - \dim S_\Gamma \geq 2.
\end{displaymath}
Moreover we observe that
\begin{displaymath}
\codim_{\widehat{\mathbf{G}}(\mathbb{Z}) \backslash D_{\widehat{G}}} \widehat{\mathbf{H}}(\mathbb{Z}) \backslash D_{\widehat{H}}\geq 2,
\end{displaymath}
\begin{displaymath}
\codim_{\widehat{\mathbf{G}}(\mathbb{Z}) \backslash D_{\widehat{G}}} \psi(W)\geq 3.
\end{displaymath}
The above computations show that, if $\Gamma$ is non-arithmetic, $W$ is an \emph{unlikely intersection}:
\begin{displaymath}
\codim_{\widehat{\mathbf{G}}(\mathbb{Z}) \backslash D_{\widehat{G}}} \psi(W) < \codim_{\widehat{\mathbf{G}}(\mathbb{Z}) \backslash D_{\widehat{G}}} \psi(S_\Gamma) + \codim_{\widehat{\mathbf{G}}(\mathbb{Z}) \backslash D_{\widehat{G}}} \widehat{\mathbf{H}}(\mathbb{Z}) \backslash D_{\widehat{H}}.
\end{displaymath}
To show that unlikely intersections arise in a finite number, we argue using o-minimality and functional transcendence (in the o-minimal structure $\R_{\an, \exp}$, given by expanding the semialgebraic sets with restricted analytic functions and the real exponential). Indeed here we sketch how to show that unlikely intersections can be parametrised by a countable and \emph{definable set}, see Proposition \ref{mainpropsec} for all details. Let $\mathcal{F}$ be a semialgebraic fundamental set for the action of $\Gamma$ on $X$ and consider
\begin{displaymath}
\Pi_0(\mathbf{H}):=\{(x,\hat{g}) \in \mathcal{F} \times \widehat{G}: \im (\tilde{\psi}(x): \DT=\Res^{\C}_{\R}\Gm \to \widehat{G})\subset \hat{g} \widehat{H} \hat{g}^{-1} \},
\end{displaymath}
where $\tilde{\psi}: X \to D_{\widehat{G}}$ is the lift of the period map constructed in Theorem \ref{strategy1}. We will show that $\Pi_0(\mathbf{H})$ is a definable subset of $\mathcal{F} \times \widehat{G}$, and we then study
\begin{displaymath}
\Sigma^1=\Sigma(\mathbf{H})^1 := \{\hat{g} \widehat{H}\hat{g}^{-1}: (x,\widehat{g})	\in \Pi_0(\mathbf{H}) \text{  for some  } x \in \mathcal{F} \text{  such that  }d_X(x,\hat{g})\geq 1 \},
\end{displaymath}
where
\begin{displaymath}
d_X(x,\hat{g}):= \dim_{\tilde{\psi}(x)} \left(\hat{g} \widehat{H} \hat{g}^{-1}. \tilde{\psi}(x)  \cap \tilde{\psi}(X)\right).
\end{displaymath}
The set $\Sigma^1$ parametrises all totally geodesic subvarieties (apart from points, of course) of $S_\Gamma$, and we want to show that it is finite. The fact that it is definable follows from easy observations, at least once the set $\Pi_0(\mathbf{H})$ is shown to be definable. To prove that $\Sigma^1$ is countable we use the codimension computations displayed above. They are indeed essential to invoke the Ax--Schanuel Theorem of Bakker and Tsimerman (Theorem \ref{zasthm}), which is used to show that each
\begin{displaymath}
\hat{g}\widehat{H}\hat{g}^{-1}\in \Sigma^1
\end{displaymath}
is naturally a $\Q$-subgroup of $\widehat{\mathbf{G}}$ (rather than just a real subgroup of $\widehat{G}$). Since $\widehat{\mathbf{G}}$ has only countably many $\Q$-subgroups, the standard fact that a definable set in an o-minimal structure has only finitely many connected components gives the finiteness of $\Sigma^1$ and therefore the finiteness of the totally geodesic subvarieties of $S_\Gamma$.
\end{proof}

\subsection{Outline of paper}
We first fix some notations and discuss preliminary results about lattices in semisimple Lie groups without compact factors. In Section \ref{sectiongeneral}, we extend the theory of Shimura varieties starting from a \emph{real Shimura datum} $(G,X,\Gamma)$, generalising many classical results and defining $\Gamma$-special (and weakly $\Gamma$-special) subvarieties. Section \ref{zhodgesection} is devoted to $\Z$-Hodge theory and period/Mumford-Tate domains. Here is explained how to realise $S_\Gamma$ inside a period domain for $\Z$-VHS. In Section \ref{comparisonsection}, we compare the Mumford--Tate and the monodromy groups associated to the natural $\Oo_K$-VHS on $S_\Gamma$ and the $\Z$-VHS we have constructed, proving a tight relation between the two worlds, namely Theorem \ref{mainthm}. We prove the non-arithmetic Ax-Schanuel conjecture in Section \ref{proofofas}. Building on the results of the previous sections, we prove Theorem \ref{thmbader} in Section \ref{sectionfinit}. Finally, we discuss arithmetic aspects of the theory, and, in Section \ref{pointsection}, we discuss various notions of special points and André--Oort type conjectures.

\subsection{Acknowledgements}
The second author thanks Bader for an inspiring lecture on his work with Fisher, Miller and Stover in Orsay, in October 2019. The authors are grateful to Klingler for several discussion on related topics and for his \emph{cours de l'IHES} on Tame Geometry and Hodge Theory. It is a pleasure to thank Bader, Fisher, Miller and Stover for having explained us their approach and various conversations around superrigidity. We are also grateful to Esnault, Klingler and Labourie for comments on a previous draft of the paper. Finally we thank the anonymous referees for their comments and suggestions which helped improving the exposition and streamlining some arguments.

The first author would like to thank the IHES for two research visits in March 2019 and 2020. His work was supported by the Engineering and Physical Sciences Research Council [EP/ L015234/1], the EPSRC Centre for Doctoral Training in Geometry and Number Theory (The London School of Geometry and Number Theory), University College London.
\section{Notation, terminology and recollections}
In this section we fix some notation, and discuss the first properties we need about lattices in semisimple groups without compact factors. In particular we explain the arithmeticity criteria of Mostow--Vinberg and of Margulis.
\subsection{Notations}\label{notations}
We make free use of the following standard notations. We denote by $G$ real algebraic groups, and by $\mathbf{G}$ algebraic groups defined over some number field, which usually is either $\Q$, or a totally real number field. Let $K$ be a real field. The $K$-forms of $\SU(1,n)$ are known, by the work of Weil \cite{MR136682}, to be obtained as $\SU(h)$ for some Hermitian form $h$ on $F^r$, where $F$ is a division algebra with involution over a quadratic imaginary extension $L$ of $K$ and $n+1=r \deg_L (F)$. 

Regarding algebraic groups, we have:
\begin{itemize}
\item Let $G$ be a connected real algebraic group. By \emph{rank of G} we always mean the real rank of the group $G$, i.e. the dimension of a maximal real split torus of $G$;
\item If $G$ is reductive, which for us requires also that $G$ is connected, we write $G^{\ad}$ for the \emph{adjoint} of $G$, i.e. the quotient of $G$ by its centre (it is a semisimple group with trivial center);
\item The Weil restriction of $\mathbf{G}/K$ from $K$ to $\Q$ is usually denoted by $\widehat{\mathbf{G}}$.
\end{itemize}
Regarding subgroups of real algebraic groups, we have:
\begin{itemize}
\item A discrete subgroup $\Gamma$ of a locally compact group $G$ is a \emph{lattice} if $\Gamma \backslash G$ has a finite invariant measure;
\item All lattices considered in this paper are also assumed to be torsion free. Selberg’s Lemma asserts, if $G$ is semisimple, that $\Gamma$ has a torsion-free subgroup of finite index, see for example \cite[Theorem 4.8.2]{MR3307755};
\item Given $H$ an algebraic subgroup of $G$, we write $\Gamma_H$ for $\Gamma \cap H(\R)$;
\item A lattice $\Gamma \subset G$ in a connected semisimple group without compact factors is \emph{reducible} if $G$ admits infinite connected normal subgroups $H,H'$ such that $HH'=G$, $H \cap H'$ is discrete and $\Gamma / (\Gamma_H \cdot \Gamma_{ H'}) $ is finite. A lattice is \emph{irreducible} if it is not reducible;
\item A subgroup $\Gamma \subset G$ is \emph{arithmetic} if there exists a semisimple linear algebraic group $\mathbf{G}/\Q$ and a surjective morphism with compact kernel $p:\mathbf{G} (\R)\to G$ such that $\Gamma$ lies in the commensurability class of $p(\mathbf{G}(\Z))$. Here we denote by $\mathbf{G}(\Z)$ the group $\mathbf{G}(\Q)\cap v^{-1}(\Gl(V_\Z))$ for some faithful representation $v: \mathbf{G}\to \Gl(V_\Q)$, where $V_\Q$ is a finite dimensional $\Q$-vector space and $V_\Z$ a lattice in $V_\Q$. Torsion free arithmetic subgroups are lattices.
\end{itemize}

\subsection{Local rigidity}\label{localrig}
Let $G$ be a real semisimple algebraic group without compact factors and $\Gamma$ be a subgroup of $G$. Denote by
\begin{displaymath}
\Ad: \Gamma \subset G \xrightarrow{\Ad} \Aut(\mathfrak{g})
\end{displaymath}
the adjoint representation in the automorphisms of the Lie algebra $\mathfrak{g}$ of $G$.
\begin{defi}
We define the (adjoint) \emph{trace field} of $\Gamma$ as the field generated over $\Q$ by the set
\begin{displaymath}
\{\tr \Ad (\gamma) : \gamma \in \Gamma\}.
\end{displaymath}
\end{defi}
If $\Gamma$ is a lattice, its trace field is a finitely generated field extension of $\Q$, which depends only on the commensurability class of $\Gamma$. Indeed it is a well known fact that lattices are finitely generated (or even finitely presented). See for example \cite[Theorem 4.7.10]{MR3307755} and references therein.

\begin{defi}
An irreducible lattice $  \Gamma  \subset G$ is \emph{locally rigid} if there exists a neighbourhood $U$ of the inclusion $i: \Gamma \hookrightarrow G$ in $\Hom (\Gamma , G)$, such that every element of $U$ is conjugated to $i$. 
\end{defi}
The trace field $K$ of a locally rigid lattice is a real number field (which is not, a priori, totally real). By Borel density theorem, lattices in $G$ are Zariski dense. For details see for example \cite[Corollary 4.5.6]{MR3307755}. In particular $\Gamma$ determines a $K$-form of $G$, which we denote by $\mathbf{G}$, such that, up to conjugation by an element in $G$, $\Gamma$ lies in $\mathbf{G}(K)$, and $K$ is minimal with this property. For proofs of such facts we refer to Vinberg's paper \cite{MR0279206}. See also \cite[Chapter VIII, Proposition 3.22]{margulisbook} and \cite[Proposition 12.2.1]{MR849651}.

If $n>1$, lattices in $G=\PU(1,n)$ are known to be locally rigid. For completeness we recall a more general\footnote{More general in the sense that it allows also to consider $G=\Sl_2(\C)$. Even if local rigidity for non-cocompact lattices in $\Sl_2(\C)$ can fail.} result, which builds on the study of lattices initiated by Selberg, Calabi and Weil \cite{MR0169956}.
\begin{thm}[{\cite[Theorem 0.11]{MR0267041}}]\label{coeffinnumbers}
If $G$ is not locally isomorphic to $\Sl_2(\R)$, then, for every irreducible lattice $\Gamma$ in $G$ there exists $g\in G$ and a subfield $K \subset \R$ of finite degree over $\Q$, such that $g\Gamma g^{-1}$ is contained in the set of $K$-rational points of $\mathbf{G}$.
\end{thm}

Since lattices are finitely generated, we know that there exist a finite set of finite places $\Sigma$ of $K$ such that $\Gamma$ lies in $\mathbf{G}(\Oo_{K,\Sigma})$ (once a representation is fixed, and up to conjugation by $G$), where $\Oo_{K,\Sigma}$ is the ring of $\Sigma$-integers of $K$. In the next subsection we discuss properties of lattices that are contained in $\mathbf{G}(\Oo_K)$.

\subsection{Lattices and integral points}\label{weilresbasic}
Let $\Gamma$ be a lattice in $G$. Assume that a \emph{totally} real number field\footnote{In Theorem \ref{totreal} we will see that the trace field of a complex hyperbolic lattice is indeed totally real.} $K'$ and a form $\mathbf{G}_{K'}$ of $G$ over $K'$ are given, such that a subgroup of finite index of $\Gamma$ is contained in $\mathbf{G}_{K'} (\Oo_{K'})$, where $\Oo_{K'}$ denotes the ring of integers of $K'$. Then the field $K=\Q\{\tr \Ad \gamma : \gamma \in \Gamma\}$ is contained in $K'$, as remarked above. The following is \cite[Corollary 12.2.8]{MR849651}, see also \cite[Lemma 4.1]{MR586876}. It will be the fundamental criterion to detect arithmeticity in Theorem \ref{thmbader}.
\begin{thm}[Mostow–Vinberg arithmeticity criterion]\label{arithcrit}
Let $\Gamma \subset \mathbf{G}(\Oo_{K'})$ be a lattice in $G$. The lattice $\Gamma$ is arithmetic if and only if for every embedding $\sigma : K' \to \R $, not inducing the identity embedding of $K$ into $\R$, the group $\mathbf{G}_{K'} \otimes_{K',\sigma} \R$ is compact.
\end{thm}

\begin{rmk}\label{infiniteindex}
If $\Gamma \subset \mathbf{G}(\Oo_K)$ is a non-arithmetic lattice, where $K$ is the trace field of $\Gamma$, then $\mathbf{G}(\Oo_K)$ is not discrete in $G$ and $\Gamma$ has infinite index in $\mathbf{G}(\Oo_K)$. 
\end{rmk}

Let $\mathbf{G}$ be a semisimple algebraic group defined over a totally real number field $K \subset \R$, and write $G$ for its real points.  Denote by $\Omega_\infty$ the set of all archimedean places of $K$ and write $\widehat{\mathbf{G}}$ for the Weil restriction from $K$ to $\Q$ of $\mathbf{G}$, as in \cite[Section 1.3]{MR670072}. It has a natural structure of $\Q$-algebraic group and 
\begin{displaymath}
\widehat{\mathbf{G}}(\R)= \prod _{\sigma \in \Omega_\infty} G_\sigma,
\end{displaymath}
where $G_\sigma$ is the real group $\mathbf{G}\times_{K,\sigma} \R$.

\begin{prop}[{\cite[Section 5.5]{MR3307755}}]
The subgroup $\mathbf{G}(\Oo_K)$ of $G$ embeds as an arithmetic subgroup of $\widehat{\mathbf{G}}$ via the natural embedding
\begin{displaymath}
 \mathbf{G}(\Oo_K) \hookrightarrow \widehat{\mathbf{G}},  \ \ g \mapsto (\sigma (g))_{\sigma \in \Omega_\infty}.
\end{displaymath}
\end{prop}
Indeed $\mathbf{G}(\Oo_K)$ becomes $\widehat{\mathbf{G}} (\Z)$. We also have that, if $\mathbf{G}$ is simple, then $ \mathbf{G}(\Oo_K)$ gives rise to an irreducible lattice. Moreover, if for some $\sigma \in \Omega_\infty$, $G_\sigma$ is compact, then $\mathbf{G}(\Oo_K)$ gives rise to a cocompact lattice. We recall here that a closed subgroup $\Gamma$ of $G$ is \emph{cocompact} if $\Gamma\backslash G$ is compact. Every cocompact, torsion-free, discrete subgroup of $G$ is a lattice.

\subsection{Arithmeticity criteria, after Margulis}
The results discussed in this section can be found in \cite{margulisbook} and are due to Margulis.

\begin{thm}[Arithmeticity]\label{margulisrank}
Let $G$ be a semisimple group without compact factors. If $G$ has rank at least $2$ and $\Gamma\subset G$ is an irreducible lattice, then $\Gamma$ is arithmetic.
\end{thm}
The proof of Theorem \ref{margulisrank} is based on applying the \emph{supperrigidity theorem}, which is recalled below, to representations obtained from different embeddings of the trace field of $\Gamma$ into local fields. 

\begin{thm}[Superrigidity]\label{superrig}
Let $G$ be a semisimple group of rank at least $2$, and $\Gamma \subset  G$ an irreducible lattice. Let $E$ be a local field of characteristic zero, and $\mathcal{G}'$ be an adjoint, absolutely simple $E$-group. Then every homomorphism $\Gamma \to \mathcal{G}'(E)$ with Zariski-dense and unbounded image extends uniquely to a continuous homomorphism $G\to \mathcal{G}'(E)$.
\end{thm}

Recall that two subgroups $\Gamma_1,\Gamma_2$ of $G$ are said to be \emph{commensurable} with each other if $\Gamma_1 \cap \Gamma_2$ has finite index in both $\Gamma_1$ and $\Gamma_2$. The commensurator $\Comm(\Gamma)$ of $\Gamma$ is a subgroup of $G$:
\begin{displaymath}
\Comm(\Gamma):=\{g\in G : \Gamma \text{  and  } g \Gamma g^{-1} \text{  are commensurable with each other}\}.
\end{displaymath} 

The following is known as the \emph{commensurability criterion for arithmeticity}, and it will be reproven, in the special case of complex hyperbolic lattices, in Section \ref{sectionmarguliscomm}, as an ``unlikely intersection phenomenon''.
\begin{thm}\label{commcriterion}
Assume $G$ is without compact factors and $\Gamma$ is an irreducible lattice. Then $\Gamma$ is arithmetic if and only if it has an infinite index in $\Comm (\Gamma)$. 
\end{thm}

\subsection{Non--arithmetic complex hyperbolic lattices}\label{known} Regarding commensurability classes of non-arithmetic lattices in $\PU(1,n)$, at the time of writing this paper, we have:
\begin{itemize}
\item[$n=2$.] By the work of Deligne, Mostow and Deraux, Parker, Paupert, there are 22 known commensurability classes in $\PU(1,2)$. See \cite{MR3461365, zbMATH07390741} and references therein;
\item[$n=3$.] By the work of Deligne, Mostow and Deraux, there are 2 commensurability classes of non-arithmetic lattices in $\PU(1,3)$. In both cases the trace field is $\Q(\sqrt{3})$ and the lattices are not cocompact. See \cite{zbMATH07195380} and references therein.
\end{itemize}
For $n>3$ non-arithmetic lattices are currently not known to exist. One of the biggest challenges in the study of complex hyperbolic lattices is to understand for each $n$ how many commensurability classes of non-arithmetic lattices exist in $\PU(1,n)$.

\section{Real Shimura data and general properties}\label{sectiongeneral}
We extend the (geometric) theory of Shimura varieties, more precisely as defined by Deligne in terms of Shimura datum \cite{delignetravaux, deligneshimura}, to include quotients of Hermitian symmetric spaces by arbitrary lattices. We then discuss a generalisation of the theory of toroidal compactifications to non-arithmetic lattices and various definability results, proving in the most general form all the results we need in the rest of the paper.
\subsection{Definitions and recollections}\label{31}
Let $G$ be a real reductive group. In this section we let $\Gamma$ be a discrete subgroup of $G$, whose image in the adjoint group of $G$, denoted by $G^{\text{ad}}$, is a lattice. Recall that a normal subgroup $N$ of $G$ induces a decomposition, up to isogeny, of $G$. In symbols we write $G \sim G'\times N$, for some subgroup $G'$ of $G$.

\begin{defi}
A $\Gamma$-\emph{factor} of $G$ is either a one dimensional split torus in the centre of $G$, or a normal subgroup $N$ of $G\sim G'\times N$ such that the image of $\Gamma$ along the projection from $G $ to $N$ is a lattice in $N$. We call a $\Gamma$-factor \emph{irreducible}, if the image of $\Gamma$ in $N$ is an irreducible lattice.
\end{defi}
In the above definition of $\Gamma$-factor, $N$ is not required to be irreducible, and may have compact semisimple factors. Moreover the centre of $G$ is a product of irreducible tori which are $\Gamma$-factors.

\begin{prop}\label{proponproducts}
The group $G$ can be written as a finite product of $\Gamma$-irreducible factors.
\end{prop} 

\begin{proof}
Given $\Gamma$ in $G$ there is a unique direct product decomposition $G=\prod _{i\in I} G_i$, where $G_i$ is normal in $G$, such that $\Gamma_ {G_i}$ is an irreducible lattice in $G_i$ and $\prod_i \Gamma_{G_i}$ has finite index in $\Gamma$. See for example \cite[Theorem 5.22, page 86]{MR0507234}.
\end{proof}
From now on, we write $\DT$ for the \emph{Deligne torus}. That is the real torus obtained as Weil restriction from $\C$ to $\R$ of the group $\C^*$. 
\begin{defi}\label{realshimuradef}
A \emph{real Shimura datum} is a triplet $(G,X, \Gamma)$ where $G$ is a real reductive algebraic group, $\Gamma \subset G$ a discrete subgroup such that its image in $G^{\ad}$ is a lattice, and $X$ is a $G$-orbit in the set of morphisms of algebraic groups $\Hom (\DT, G)$ such that for some (all) $x\in X$ the \emph{real Shimura--Deligne axioms} are satisfied:
\begin{itemize}
\item[RSD0.] The image of $\Gamma$ in each $\Gl_2$-irreducible factor is an arithmetic lattice;
\item[RSD1.] The adjoint representation $\Lie(G)$ is of type $\{(-1,1), (0,0), (1,-1)\}$;
\item[RSD2.] The involution $x(\sqrt{-1})$ of $G^{\text{ad}}$ is a Cartan involution;
\item[RSD3.] $G$ has no simple compact $\Gamma$-factors.
\end{itemize}
\end{defi}
Without the axiom RSD0 the above definition would include every Riemann surface.
\begin{rmk}\label{rmkfactorisa}
Let $X$ be a $G$-orbit in the set of morphisms of algebraic groups $\Hom (\DT, G)$ satisfying RSD1, RSD2 and RSD3. Let $x \in X$ be such that $x : \DT \to G$ factorises trough $H$, for some subgroup $H$ of $G$. Then the $H$ orbit of $x$ in $X$ satisfies RSD1 and RSD2. See for example \cite[Proposition 3.2]{MR2180407}.
\end{rmk}

We have a notion of morphism of real Shimura data. 
\begin{defi}
Let $(G_1,X_{1}, \Gamma_{1})$ and $(G_2,X_2, \Gamma_2)$ be real Shimura data. A \emph{morphism of real Shimura data}
\begin{displaymath}
(G_1,X_{1}, \Gamma_{1})\to(G_2,X_2, \Gamma_2)
\end{displaymath}
is a morphism of real algebraic groups $f: G_1 \to G_2$ such that for each $x\in X_1$, the composition $f\circ x : \DT \to G_2$ is in $X_2$ and $f ( \Gamma_{1}) \subset \Gamma_2$.
\end{defi}
We need to work with a more general definition of real sub-Shimura datum, allowing the axiom RSD0 to fail. Indeed, as long as the Hermitian space $X$ has dimension strictly bigger than one, we want to consider arbitrary sub-Shimura data. For example, in Theorem \ref{thmbader}, we want to consider every totally geodesic subvariety of dimension one.

\begin{defi}\label{defsubshimura}
Let $(G,X, \Gamma)$ a real Shimura datum, $H$ a subgroup of $G$ and $\Gamma_H$ a lattice in $H$. A \emph{real sub-Shimura datum} is a triplet $(H,X_H, \Gamma_H)$ where $X_H$ is a $H$-orbit in the set $\Hom (\DT, G)$ satisfying RSD1, RSD2 and RSD3 of Definition \ref{realshimuradef}. 
\end{defi}

\begin{defi}
Let $(G,X, \Gamma)$ be a real Shimura datum. We define its \emph{adjoint real Shimura datum}, simply denoted by
\begin{displaymath}
(G,X, \Gamma)^{\ad},
\end{displaymath}
as the triplet $(G^{\ad},X^{\ad},\Gamma^{\ad})$ where $X^{\ad}$ is the $G^{\ad}$-conjugacy class of morphisms $x : \DT \to G \to G^{\ad}$ and $\Gamma^{\ad}$ is the image of $\Gamma$ in $G^{\ad}$. It is again a real Shimura datum.
\end{defi}

\begin{defi}
Let $(G,X,\Gamma)$ be a real Shimura datum. We denote by $S_\Gamma $ the quotient $\Gamma \backslash X$ and we refer to it as a \emph{Shimura variety}.  
\end{defi}
Every sub-Shimura datum $(H,X_H, \Gamma_H) \subset (G,X, \Gamma)$, as in Definition \ref{defsubshimura}, induces a closed (for the analytic topology) subvariety $S_{\Gamma_H}=\Gamma_H \backslash X_H$ of $S_\Gamma$.

Proposition \ref{proponproducts} can be rephrased as follows.
\begin{prop}\label{irred}
Any real Shimura datum $(G,X,\Gamma)$ can be uniquely written, up to isogeny, as a product of a finite number of real sub-Shimura data $(G_i,X_i, \Gamma_i)$ such that the $\Gamma_i$ are irreducible lattices in $G_i$ and 
\begin{displaymath}
\prod_i S_{\Gamma _i }\to S_\Gamma
\end{displaymath}
is a finite covering (of complex manifolds).
\end{prop}

\begin{thm}[{\cite[Proposition 1.1.14 and Corollary 1.1.17]{delignetravaux}}]\label{vhsonsgamma}
Let $(G,X,\Gamma)$ be a real Shimura datum. Then $X$ has a unique structure of a complex manifold such that, for every real representation
$\rho : G \to \Gl (V_\R)$, $(V_\R, \rho \circ h)_{h \in X}$ is a holomorphic family of Hodge structures. For this complex structure, each family $(V_\R, \rho \circ h)_{h \in X}$ is a variation of Hodge structure, and $X$ is a finite disjoint union of Hermitian symmetric domains.
\end{thm}

Choosing a connected component $X^+$ of $X$, we can also define \emph{connected real Shimura data} and connected Shimura varieties. In what follows we often implicitly work over some connected component.

\begin{rmk}
Let $(\mathbf{G}, \mathcal{X})$ be a Shimura datum in the sense of Deligne (i.e. $\mathbf{G}$ is assumed to be a $\Q$-group). For any faithful representation $\mathbf{G}\to \Gl(V_\Z) $, the triplet $(\mathbf{G}_\R, \mathcal{X}, \mathbf{G}(\Q)\cap \Gl(V_\Z))$ defines a real Shimura datum. Given a Shimura datum $(\mathbf{G}, \mathcal{X})$ as before and $K_{\mathbb{A}_f}$ a compact open subgroup of the finite adelic points of $\mathbf{G}$, the triplet $(\mathbf{G}_\R, \mathcal{X}, \mathbf{G}(\Q)\cap K_{\mathbb{A}_f})$ is a real Shimura datum. We refer to the latter case as \emph{congruence Shimura datum}.
\end{rmk}

\begin{thm}[Baily-Borel, Siu-Yau, Mok]\label{algebraic}
Every Shimura variety $S_\Gamma$ has a unique structure of a quasi-projective complex algebraic variety.
\end{thm}
\begin{proof}
If the lattice $\Gamma$ is arithmetic, it was proven by Baily and Borel \cite{MR0216035}. To prove the result, we may and do replace $\Gamma$ by a finite index subgroup. In particular, thanks to Proposition \ref{irred}, we may assume that $\Gamma$ is irreducible. Thanks to Theorem \ref{margulisrank}, if $G$ has rank bigger or equal than two, then $\Gamma$ is arithmetic. Regarding rank one groups we have that $X$ is isomorphic to the complex unit ball $\mathbb{B}^n\subset \C^n$, for some $n \geq 1$. If $n=1$ the axiom RSD0 ensures that $\Gamma$ is arithmetic. If $n>1$, Siu-Yau \cite{MR645748} and Mok \cite{MR2884042} proved that $S_\Gamma$ is biholomorphic to a quasi-projective variety. To see that $S_\Gamma$ has a unique algebraic structure, one can apply \cite[Theorem A]{deng2020big} in place of the classical Borel extension theorem to the identity map ${S_\Gamma}^{\an} \to {S_\Gamma}^{\an}$, where the left hand side is understood with a different algebraic structure from the right hand side.
\end{proof}

\begin{defi}\label{gammaspeci}
Let $W\subset S_\Gamma$ be an irreducible algebraic subvariety. We say that $W$ is $\Gamma$-\emph{special} if it is induced by a real sub-Shimura datum $(H,X_H, \Gamma_H)$ of $(G,X, \Gamma)$. We say that $W$ is \emph{weakly-$\Gamma$-special} if there exists a real sub-Shimura datum $(H,X_H, \Gamma_H)$ of $(G,X, \Gamma)$ such that its adjoint splits as a product
\begin{displaymath}
(H,X_H, \Gamma_H)^{\ad}=(H_1,X_1, \Gamma_1) \times (H_2,X_2, \Gamma_2),
\end{displaymath}
and $W$ is the image of $X_1^+\times \{x_2\}$ for some $x_2\in X_2$.
\end{defi}
It is a standard fact that $\Gamma$-special subvarieties of $S_\Gamma$ are totally geodesic.

\begin{defi}
Given a real Shimura datum $(G,X,\Gamma)$, we say that $S_\Gamma$ is an \emph{arithmetic Shimura variety} if $\Gamma$ is arithmetic. If $\Gamma$ is irreducible and non-arithmetic, we say that $S_\Gamma$ is a \emph{non-arithmetic ball quotient}.
\end{defi}

\subsection{Real sub-Shimura data in rank one}\label{embeddings}
Let $(G=\PU(1,n),X,\Gamma)$ be a real Shimura datum, for some $n>1$. We remark here that the main results of the paper remain true, with the same proofs, if one starts with $\Gamma$ an irreducible lattice in $\PU(1,n)\times K$, where $K$ is some compact factor. Let $X_H$ be a Hermitian symmetric sub-domain of $X$ with automorphism group $H$. Such spaces will be often referred to as \emph{real sub-Shimura couples} of $(G,X)$. In this section we explicitly describe the map $H \to G$ inducing the inclusion of $X_H$ in $G$. For more details regarding totally geodesic subvarieties of the complex ball, we refer to \cite[Section 3.1.11]{zbMATH01270600} and also \cite[Section 2.2]{MR3461365}. We will see that every $H$ appearing in this way can be assumed to be a semisimple group. 

Let $V_\C$ be a $n+1$-dimensional complex vector space with $\varphi : V_\C \times V_\C\to \C$ a hermitian form of signature $(1,n)$. Fix a $\C$-basis $(e_1, \dots, e_{n+1})$ such that the quadratic form associated to $\varphi$ becomes
\begin{displaymath}
z_1 \overline{z}_1 - \sum_{i=2}^{n+1} z_i \overline{z}_i.
\end{displaymath}
The group $G=\PU(1,n)$ can be defined in $\Gl (V_\C)$ as equivalence classes of matrices $M$ satisfying
\begin{displaymath}
\overline{M}^t g M =g,
\end{displaymath}
where $g=\operatorname{diag}(1, -1 , \dots , -1)$. Since $g=\overline{g}$, $G$ can be seen in $\Gl (V_\R)$. Let $W_\C$ be a sub-vector space of $V_\C$ such that $\varphi$ restricts to a hermitian form $W_\C \times W_\C\to \C$ and write $V_\C$ as the direct sum of $W_\C$ and its orthogonal complement. We can arrange a basis of $W_\C$ in such a way that the quadratic form associated to $\varphi_{| W_\C}$ corresponds to a matrix of the form $\operatorname{diag}(1, -1 , \dots , -1)$.

Let $H \subset G \subset \Gl(V_\R)$ be the subgroup stabilizing $W_\C$. Since $H$ has to stabilise also the orthogonal of $W_\C$, it is isomorphic to $\Pu(1,m)\times C$, where $m$ is the rank of $W_\C$, from some compact subgroup $C$ in $G$. Moreover every $H$ associated to an $X_H$ arises in this way. Indeed, up to conjugation by an element in $G$, any $\alpha : \DT \to G$ can be written as
\begin{displaymath}
z \mapsto  \operatorname{diag}(z,\overline{z},1,\dots,1).
\end{displaymath}
Given a subgroup $H\subset G$ as before, we may therefore assume that $\alpha(\DT) \subset H $, for some choice of a basis for $W_\C$.

Notice also that a similar description holds when $G=\operatorname{GU}(1,n)$.

\subsection{Toroidal compactification of non-arithmetic ball quotients}\label{compac}
For the length of this section, we assume that $G$ has real rank one. A compactification of $S_\Gamma=\Gamma \backslash X$ as a complex spaces with isolated normal singularities was obtained by Siu and Yau \cite{MR645748}. Hummel and Schroeder \cite{MR1374201} and then Mok \cite[Theorem 1]{MR2884042} showed that $S_\Gamma$ admits a unique smooth toroidal compactification, which we denote by $\overline{S_\Gamma}$. If the parabolic subgroups of $\Gamma$ are unipotent, the boundary $\overline{S_\Gamma}-S_\Gamma$ is a disjoint union of abelian varieties. The minimal compactification $S_{\Gamma}^{BB}$, which is proven to be projective-algebraic in \cite{MR2884042}, can be recovered by blowing down the boundary.

Given a parabolic subgroup $P\subset G$ we write:
\begin{itemize}
\item $N_P$ for the unipotent radical of $P$;
\item $U_P$ for the centre of $N_P$. We may identify $U_P$ with its Lie algebra $\Lie(U_P)\cong \R$.
\end{itemize}
\begin{defi}\label{defparabolci}
A parabolic subgroup $P\subset G$ is called $\Gamma$-\emph{rational} if its unipotent radical $N_P$ intersects $\Gamma$ as a lattice.
\end{defi}

\subsubsection{First properties of the toroidal compactification}\label{toroidalsubse}
From now on, when speaking of toroidal compactifications, we replace $\Gamma$ by a finite index subgroup and always assume that $\Gamma$ is torsion free and that the parabolic subgroups of $\Gamma$ are unipotent. The following is a consequence of \cite{MR2884042}.
\begin{thm}[{\cite[Corollary 7.6, Chapter III]{MR2884042}}]
The toroidal compactification $\overline{S_\Gamma}$ of $S_\Gamma$ is a smooth compactification with strict normal crossing divisor at infinity $B:= \overline{S_\Gamma} - S_\Gamma$.
\end{thm}

Denote by $\Delta\subset \C$ the complex disk and by $\Delta^*$ the punctured disk. As in the arithmetic case, there exists an open cover $\{U_\alpha\}_\alpha$ of $\overline{S_\Gamma}$ such that $U_\alpha=\Delta^n$ and $U_\alpha \cap S_\Gamma= \Delta^{n-1}\times \Delta^*$. We have the following, see also \cite[Theorem 7.2, Chapter III]{MR2884042}.
\begin{prop}\label{quasiunipot}
For any $\Gamma$-rational parabolic $P$, the set $U_P \cap \Gamma$ is isomorphic to $\Z$. The image of the fundamental group of $U_\alpha \cap S_\Gamma= \Delta^{n-1}\times \Delta^*$ in the fundamental group of $\overline{S_\Gamma}$ lies in $\Gamma_{U_P}\cong \Z$.
\end{prop}

\begin{proof}
The first part can be found in \cite[Section 1.3]{MR2884042}. For the second part, we work with the local coordinates as in \cite[pages 255-256]{MR471627}. Let $X^\vee$ be the compact dual of $X$ and $\mathcal{X}\subset X^\vee$ be the Borel embedding. Assume that we are working with a $\Gamma$-rational boundary component $F=\{b\}$ corresponding to a $\Gamma$-rational parabolic $P$, and let $V$ be the quotient of the unipotent radical of $P$ by it centre. It is a real vector space of rank $n-1$ (where $n=\dim S_\Gamma$). Set 
\begin{displaymath}
\mathcal{X}_b := \bigcup _{g\in U_P  \otimes \C}  g \cdot \mathcal{X} \subset X^\vee.
\end{displaymath} 
There exists a canonical holomorphic isomorphism
\begin{displaymath}
j: \mathcal{X} \cong  \C^{n-1}\times \C \times \{b\},
\end{displaymath}
where $\C^{n-1} =V_\C$ and the latter copy of $\C$ is $U_P \otimes \C$. We can naturally identify the universal cover of $U_P \cap S_\Gamma= \Delta^{n-1}\times \Delta^*$ with
\begin{equation}
\D \cong \{(z_1,\dots, z_{n-1},z_n , b) \in \C^{n-1}\times \C  \times F : \im(z_n) \geq 0\}.
\end{equation}
The group $U_P \otimes \C$ acts on $\D$, in these coordinates, by $(z_1,\dots, z_{n-1},z_n,b) \mapsto (z_1,\dots, z_{n-1},z_n+a,b)$. Observe that we can factorise the map $\pi: \mathcal{X}\to S_\Gamma$ as
\begin{displaymath}
\mathcal{X} \xrightarrow{\exp_F}\exp_F ( \mathcal{X}) \to S,
\end{displaymath}
where $\exp_F : \C^{n-1}\times \C \times F\to  \C^{n-1} \times \C^* \times F$ is simply $\exp(2\pi i -)$ on the $\C$-component and the identity on $\C^{n-1}$. Moreover $\exp_F ( \mathcal{X})$ is $\Gamma_{U_P} \backslash \mathcal{X}$. To conclude, we observe that the diagram
\begin{center}
\begin{tikzpicture}[scale=2]
\node (A) at (-1,1) {$\exp_F ( \mathcal{X})$};
\node (B) at (1,1) {${\exp_F ( \mathcal{X})}^\vee$};
\node (C) at (-1,0) {${S_\Gamma}$};
\node (D) at (1,0) {$\overline{S_{\Gamma}}$};
\path[right hook->,font=\scriptsize]
(A) edge node[above]{} (B)
(C) edge node[above]{} (D);
\path[->,font=\scriptsize,>=angle 90]
(A) edge node[right]{} (C)
(B) edge node[right]{} (D);
\end{tikzpicture}
\end{center}
is commutative, since the boundary of ${\exp_F ( \mathcal{X})}^\vee$ is mapped onto the boundary of $\overline{S_\Gamma}$.
\end{proof}

Finally we describe in a more explicit way what we discussed in the proof of Proposition \ref{quasiunipot}. We follow \cite[Section 1]{MR1238407} and \cite[Equation 8]{MR2884042}. Identify $X$ with the complex ball $\mathbb{B}^n\subset \C^n$. For any $b$ in the boundary of $\mathbb{B}^n$, let $\Gamma_b\subset \Gamma$ be the set of elements fixing $b$. The Siegel domain $D$ representation of $X$ with $b$ corresponding to infinity is 
\begin{displaymath}
D:=\{(z_1,\dots , z_n): \C^n \in \im (z_n) > \sum_{i=1}^{n-1}|z_i|^2\}.
\end{displaymath}
For any $N>0$ set
\begin{equation}\label{dn}
D^{(N)}:=\{(z_1,\dots , z_n)\in \C^n : \im (z_n) > \sum_{i=1}^{n-1}|z_i|^2 +N\}.
\end{equation}
Let $W_b$ be the set of parabolic isometries fixing $b$. If $N$ is large enough, two points $x,y \in D^{(N)}$ are equivalent mod $\Gamma$ if and only if they are equivalent modulo the lattice $\Gamma_{W_b}$. Moreover, thanks to the work of Margulis and Gromov (see \cite[Section 2]{MR645748}, and references therein), for $N$ sufficiently large one can take $D^{(N)}$ to be in the set 
\begin{displaymath}
A_b := \{x \in \mathbb{B}^n : \min_{\gamma \in \Gamma_b} d(x, \gamma x) < \epsilon\},
\end{displaymath}
for some $\epsilon >0$ (which depends on $\Gamma$). Consider the set
\begin{displaymath}
E:=\{ x \in \mathbb{B}^n : \min_{\gamma \in \Gamma} d(x, \gamma x) \geq  \epsilon \}.
\end{displaymath}
Notice that the image of $E$ along $\pi : \mathbb{B}^n\to S_\Gamma$ is compact. Finally we can write $S_\Gamma$ as the union of a finite number of $\pi(A_b)$s and $\pi(E)$.

\subsubsection{Metric at infinity}\label{metriinf}
Let $\omega$ be the $(1,1)$ form on $S_\Gamma$ inducing the natural Hermitian metric. In this section we explain the behaviour of $\omega$ near the boundary, as it appears in To's paper \cite[Section 2]{MR1238407}.

Fix an open cover $\{U_\alpha\}_\alpha$ of $\overline{S_\Gamma}$ such that $U_\alpha=\Delta^n$ and $U_\alpha \cap S_\Gamma= \Delta^{n-1}\times \Delta^*$.
\begin{defi}\label{poincaregrowthdef}
We say that $\omega$ is of \emph{Poincaré-growth} with respect to the toroidal compactification $\overline{S_\Gamma}$ if, for any $\alpha$, $\omega$ restricted to $U_\alpha$ is bounded by
\begin{displaymath}
C_\alpha \left( \sum_{k=1}^{n-1} |dz_k|^2 + \frac{|dz_n|^2}{|z_n|^2|\log z_n|^2} \right).
\end{displaymath}
\end{defi}
The above definition does not depend on the choice of coordinates. Whenever we say that $\omega$ is of Poincaré-growth, it will be understood that we refer to the growth of $\omega$ with respect to the toroidal compactification $S_\Gamma$.

\begin{thm}[{\cite[Section 1.2]{MR2884042}} and {\cite[Section 2]{MR1238407}}]\label{poincaregrowth}
The K\"{a}hler form $\omega$ on $S_\Gamma$ is of Poincaré-growth.
\end{thm}
Finally we notice that the holomorphic tangent bundle $TS_{\Gamma}$ extends to a holomorphic vector bundle $\overline{TS_{\Gamma}}$ on $\overline{S_{\Gamma}}$ as follows: in an open neighbourhood $U_\alpha=\Delta^n$ of $\overline{S_{\Gamma}}$ where $U_\alpha\cap S_\Gamma=\Delta^{n-1}\times \Delta^*$, a local holomorphic basis of $\overline{TS_{\Gamma}}$ on $U_\alpha$ is given by
\begin{displaymath}
\frac{\partial}{\partial z_1}, \dots ,\frac{\partial}{\partial z_{n-1}}, z_n\frac{\partial}{\partial z_n} .
\end{displaymath}

\subsection{Some algebraicity results}
We first describe Siegel sets in $G$ and $X$, for the action of $\Gamma$. This is needed to define semialgebraic fundamental sets for the action of $\Gamma$ on $X$ and for proving that every real sub-Shimura datum of $(G,X,\Gamma)$ gives an algebraic subvariety of $S_\Gamma$. As usual, the only case we need to describe is when $G$ has rank one.

\subsubsection{Siegel sets}\label{siegelsetssec}
Garland and Raghunathan \cite{MR0267041} extended Borel's reduction theory \cite{MR0244260} to arbitrary lattices in rank one groups. Notice that Borel follows a different convention (in this paper $G$ always acts on $X$ on the left). Since we use the reverse order of multiplication, compared to \cite{MR0244260, MR0267041}, the inequalities defining the set $A_t$ are reversed.

Write $X\cong G/K_{0}$ for some fixed compact maximal subgroup $K_{0}\subset G$ and denote by $x_0$ the base point $\id K_0$ of $X$. Let $P$ be a minimal $\Gamma$-rational parabolic subgroup and let $L$ be the unique Levi subgroup $L$ of $P$ which is stable under the Cartan involution associated to $K_0$. Consider the Langlands decomposition
\begin{displaymath}
P=N_PAM,
\end{displaymath}
where $A$ is the split component of $P$ with respect to the basepoint $x_0$. That is $L=AM$, where $A$ is a split torus, $L\cap K$ is a maximal compact in $L$, and $M$ is the maximal anisotropic subgroup of the connected centralizer $Z(A)$ of $A$ in $P$. Since $Z(A)$ is compact, $M$ lies in $K$. For any real number $t>0$, let
\begin{displaymath}
A_t := \{a\in A : a^\alpha \geq t \},
\end{displaymath}
where $\alpha$ is the unique positive simple root of $G$ with respect to $A$ and $P$.

A \emph{Siegel set} for $G$ for the data $(K_0, P, A)$ is a product
\begin{displaymath}
\Sigma ' _{t, \Omega}= \Omega \cdot A_t \cdot K_0 \subset G,
\end{displaymath}
where $\Omega$ is a compact neighbourhood of the identity in $N_P$. Finally we consider Siegel sets in $X$
\begin{displaymath}
\Sigma  _{t, \Omega}= \Omega \cdot A_t \cdot x_0 \subset X,
\end{displaymath}
where $\Sigma  _{t, \Omega}$ denotes the image of $\Sigma ' _{t, \Omega}$ in $X$.

It is interesting to compare Siegel sets with the $D^{(N)}$s described in equation \eqref{dn} in the previous section. The arguments of \cite[Proposition 3.2]{MR3502100} give the following.
\begin{prop}\label{siegelanddn}
Let $\Sigma =\Sigma _{t, \Omega}$ be a Siegel set in $X$ for the action of $\Gamma$. Then $ \Sigma$ is covered by a finite union of open subsets $\Theta$ with the following properties. For each $\Theta$ there exists $\Gamma$-rational boundary component $F=\{b\}$ corresponding to a $\Gamma$-rational parabolic $P$, a positive integer $N$ large enough, relatively compact subsets $U' \subset U(P)$, $Y'\subset \C^{n-1}$ such that $\Theta$, up to a bi-isometry, can be described as
\begin{displaymath}
 \{(z_1,\dots , z_n)\in \C^{n-1}\times U(P)_\C, \operatorname{Re}(z_n)\in U', (z_1,\dots z_n)\in Y'   : \im (z_n) > \sum_{i=1}^{n-1}|z_i|^2 +N\} \subset D^{(N)}.
\end{displaymath}
\end{prop}

\begin{defi}
A \emph{fundamental set} for the action of $\Gamma$ on $X$ is a connected open subset $\F$ of $X$ such that $\Gamma \cdot \overline{\F}=X$ and the set
\begin{displaymath}
\{\gamma \in \Gamma : \gamma \F \cap \F \neq \emptyset\}
\end{displaymath}
is finite.
\end{defi}

\begin{thm}[Garland-Raghunathan]\label{fundamentalset}
For any Siegel set $\Sigma_{t, \Omega}$ the set
\begin{displaymath}
\{\gamma \in \Gamma : \gamma \Sigma_{t,\Omega}  \cap \Sigma_{t,\Omega}\neq \emptyset \}
\end{displaymath}
is finite. There exists a Siegel set $\Sigma_{t_0, \Omega}$ and a finite subset $\Xi$ of $G$ such that:
\begin{itemize}
\item For all $b \in \Xi$, $\Gamma \cap b N_P b^{-1}$ is a lattice in $b N_P b^{-1}$;
\item The set $\Xi \cdot \Sigma_{t_0, \Omega}$ is a fundamental set for the action of $\Gamma$ on $X$.
\end{itemize}
Furthermore, when $\Omega$ is chosen to be semialgebraic, the associated Siegel sets $\Sigma_{t,\Omega}$ are semialgebraic. In particular the fundamental set $\F:= \Xi \cdot \Sigma_{t_0, \Omega}$ is semialgebraic.
\end{thm}

\begin{proof}
The first part of the statement is \cite[Theorem 4.10]{MR0267041}. It is formulated for \emph{admissible} discrete subgroups of $G$ (see Definition 0.4 in \emph{loc. cit.}), a class that includes any lattice $\Gamma \subset G$, as proved in Theorem 0.7 of \emph{loc. cit.}. When $\Omega$ is chosen to be semialgebraic, the Siegel set for $G$
\begin{displaymath}
 \Omega \cdot A_t \cdot K_0
\end{displaymath}
is semialgebraic since it is defined as product in $G$ of semialgebraic sets and therefore its image $\Sigma_{t,\Omega}$ in $X$ is again semialgebraic. Finally, since $\Xi$ is finite, $\Xi \cdot \Sigma_{t_0, \Omega}$ is semialgebraic.
\end{proof}

We discuss some functoriality properties of Siegel sets. Let $H$ be a semisimple group of $G$ intersecting $\Gamma$ as a lattice. As described in Section \ref{embeddings}, there is a real sub-Shimura datum $(H,X_H, \Gamma_H)$ associated to $H$. The next proposition describes the intersection of Siegel sets in $X$ with $X_H$ (assuming that the dimension of $X_H$ is non-zero).
\begin{prop}\label{funct}
Let $(H,X_H, \Gamma_H)$ be a real sub-Shimura datum of $(G=\PU(1,n), X, \Gamma)$. Any Siegel sets of $H$ is contained in a Siegel set for $G$.
\end{prop}

\begin{proof}
Since we are interested only in Siegel set in $X_H$ and $X$, as recalled in Section \ref{embeddings}, we may assume that $H=\PU(1,m)$ for some $m<n$. Let $P$ be a minimal $\Gamma$-rational parabolic subgroup of $P$ and $P_H$ be $P\cap H$. As above, consider the decompositions relative to $K_0 $ and $K_0 \cap H$
\begin{displaymath}
P=N_PAM, \ \ P_H= N_H A_H M_H.
\end{displaymath}
As explained in Section \ref{embeddings} we have that $N_P\cap P_H = N_H$, $A \cap P_H = A_H$ and $M\cap P_H=M_H$. The result then follows by the construction of Siegel sets explained above.
\end{proof}

\subsubsection{Fundamental sets and o-minimality}\label{ominintro}
We have all the ingredients to prove that the uniformising map, opportunely restricted, is definable, extending \cite[Theorem 4.1]{MR3502100} to cover arbitrary lattices. For a general introduction to \emph{o-minimal structures}, we refer to \cite{MR1633348}. Thanks to Theorem \ref{algebraic}, $S_\Gamma$ has a canonical structure of a complex algebraic manifold, and therefore of an $\R_{\operatorname{an, exp}}$ manifold, see also \cite[Appendix A]{MR3502100}. The following can be proven as \cite[Theorem 4.1]{MR3502100} replacing all references to \cite{MR2590897} with the arguments of Section \ref{compac}.

\begin{thm}\label{defini}
Let $(G,X,\Gamma)$ be a real Shimura datum. There exists a semi--algebraic fundamental set $\F$ for the action of $\Gamma$ in $X$ such that the restriction $\pi: \F \to S_\Gamma$ of the uniformising map $\pi : X \to S_\Gamma$ is definable in $\R_{\operatorname{an, exp}}$.
\end{thm}

The following will be implicitly used whenever speaking of $\Gamma$-special subvarieties.
\begin{cor}\label{algmorphism}
Let $(H,X_H, \Gamma_H)$ be a real sub-Shimura datum of $(G=\PU(1,n), X, \Gamma)$. The induced map $S_{\Gamma_{H}}\to S_{\Gamma}$ is algebraic.
\end{cor}
\begin{proof}
By definition $(H,X_H, \Gamma_H) \to (G, X, \Gamma)$ induces a commutative diagram
\begin{center}
\begin{tikzpicture}[scale=2]
\node (A) at (-1,1) {$X_H$};
\node (B) at (1,1) {$X$};
\node (C) at (-1,0) {$S_{\Gamma_H}$};
\node (D) at (1,0) {$S_{\Gamma}$};
\path[->,font=\scriptsize,>=angle 90]
(A) edge node[above]{$\tilde{i}$} (B)
(C) edge node[above]{$i$} (D)
(A) edge node[right]{} (C)
(B) edge node[right]{} (D);
\end{tikzpicture}.
\end{center}
For some $x \in X_H$, we can write the graph of $\tilde{i}$ as 
\begin{displaymath}
\{h.(x, \tilde{i}(x)), \ h \in H\}.
\end{displaymath}
Let $\F_H$ be a fundamental set for the action of $\Gamma_H$ in $X_H$. By Proposition \ref{funct}, we have that the graph of $\tilde{i}$ restricted to $\F_H$ is a definable set. Let $\F$ be a fundamental set for $\Gamma$ in $X$, containing $\F_H$, and consider the restriction of the uniformising map $X_H \times X \to S_{\Gamma_H}\times S_{\Gamma}$ to $\F_H \times \F$:
\begin{displaymath}
\pi _{ | \F_H \times \F}: \F_H \times \F \to  S_{\Gamma_H}\times S_{\Gamma}.
\end{displaymath}
By Theorem \ref{defini}, it is a definable map and therefore $\pi _{ | \F_H \times \F} (\operatorname{graph}(\tilde{i}_{| \F_H}))$ is definable. The result then follows from Peterzil--Starchenko's \emph{o-minimal GAGA}, which is explained below. We have eventually proven that $i: S_{\Gamma_H}\hookrightarrow S_\Gamma$ is algebraic.
\end{proof}

\begin{thm}[{\cite[Theorem 4.4 and Corollary 4.5]{MR2492989}}]\label{ogaga}Let $S$ be a smooth complex algebraic variety. Let $Y \subset S$ be a closed complex analytic subset which is also a definable subset (for some o-minimal structure expanding $\R_{\an}$). Then $Y$ is an algebraic subvariety of $S$.
\end{thm}

\section{\texorpdfstring{$\Z$}{z}-Hodge theory on non-arithmetic quotients and rigidity}\label{zhodgesection}
Let $(G,X, \Gamma)$ be a real Shimura datum and $S_\Gamma= \Gamma \backslash X$ the associated complex algebraic variety. Thanks to Theorem \ref{vhsonsgamma}, faithful linear real representations of $G$, induce real variations of Hodge structure on both $X$ and $S_\Gamma$. In this section we investigate how to construct a \emph{natural} $\Z$-variation of Hodge structure on $S_\Gamma$. The first step is proving that the traces of the image of any $\gamma \in \Gamma$ along the adjoint representation lie in the ring of integers of a totally real number field $K$.

Given a complex algebraic variety $S$, we denote by $S^{\an}$ the complex points $S(\C)$ with its natural structure of a complex analytic variety and by $\pi_1(S)$ the topological fundamental group of $S^{\an}$ (unless it is necessary, we omit the base point in the notation). Regarding variations of Hodge structure (VHS form now on) we consider only polarised and pure VHS. More precisely we consider only $\Q$, $K$, $\R$, $\C$, $\Z$ and $\Oo_K$-VHS, as defined in \cite{MR2918237}, where $K$ is a number field.

Let $(G=\PU(1,n),X,\Gamma)$ be a real Shimura datum ($n>1$). Recall that there is an isogeny from $\lambda : \SU(1,n)\to \PU(1,n)$, and we may assume that, up to replacing $\Gamma$ by a finite index subgroup, $\Gamma$ is the image along $\lambda$ of a lattice in $\SU(1,n)$. For this section we assume this is indeed the case and we let $V_\C$ be a complex vector space of dimension $n+1$, and $\SU(1,n)\to \Gl(V_\C)$ be the standard representation of $\SU(1,n)$. Let $\mathbb{V}$ be the local system associated to 
\begin{equation}\label{locsystem}
\rho : \pi_1(S_\Gamma) \to \Gl(V_\C).
\end{equation}
If $\Gamma$ is an arithmetic subgroup of $G$, it is a well-known fact that $\mathbb{V}$ is induced by a $\Z$-VHS corresponding to a family of principally polarised abelian varieties.

We use a recent work of Esnault and Groechenig to prove that lattices in $G$, not necessarily cocompact, have integral traces. As noticed in \cite[Remark 7.12]{MR3826464}, if $\Gamma$ is cocompact, such integrality follows directly from the work of Esnault and Groechenig \cite{MR3874695}. This is the first step to construct a $\Z$-VHS on $S_\Gamma$ closely related to the original one (Theorem \ref{thmzvhs}). Eventually we construct \emph{generalised modular embeddings} of $S_\Gamma$ in some period domain, see Section \ref{diagram}. 

\subsection{Cohomological rigidity and integral traces}
The first step in constructing a $\Z$-VHS is the following, which may be of independent interest (in Theorem \ref{totreal} we will prove that the trace field $K$ is in fact totally real).
\begin{thm}\label{integrality}
Let $\Gamma$ be a lattice in $G=\PU(1,n)$ for some $n>1$, $K$ its trace field and $\mathbf{G}$ the $K$-form of $G$ determined by $\Gamma$. There exists a finite index subgroup $\Gamma ' \subset \Gamma $ with integral traces. Equivalently, up to conjugation by $G$, $\Gamma'$ lies in $\mathbf{G}(\Oo_K)$, where $\Oo_K$ is the ring of integers $K$.
\end{thm}

In Section \ref{localrig} we discussed locally rigid lattices. To construct a $\Z$-VHS we need a stronger rigidity, namely \emph{cohomological rigidity}. Building on a study initiated by Weil \cite{MR0169956} in the cocompact case, Garland and Raghunathan proved the following. 
\begin{thm}[{\cite[Theorem 1.10]{MR0267041}}]\label{garalndragvanishing}
Let $G$ be a semisimple Lie group, not locally isomorphic to $\Sl_2$, nor to $\Sl_2(\C)$. For any lattice $\Gamma$ in $G$, the first Eilenberg--MacLane cohomology group of $\Gamma$ with respect to the adjoint representation is zero. In symbols:
\begin{displaymath}
H^1(\Gamma, \Ad)=0.
\end{displaymath}
\end{thm}

To see why such a vanishing is related to a rigidity result, observe that the the space of first-order deformations of $ \rho: \Gamma \hookrightarrow G$ is naturally identified with $H^1(\Gamma, \Ad)$, where 
\begin{equation}\label{adjointrep}
\Ad: \Gamma \xrightarrow{\rho } G \xrightarrow{\Ad} \Aut(\mathfrak{g})
\end{equation}
is the adjoint representation.

The following is proven by Esnault and Groechenig \cite{MR3874695} and its proof relies on Drinfeld’s theorem on the existence of $\ell$-adic companions over a finite field. Theorem \ref{integrality} is a consequence of such result.
\begin{thm}[{\cite[Theorem 1.1]{MR3874695}}]\label{esgro}
Let $S$ be a smooth connected quasi-projective complex variety. Then a complex local system $\mathbb{V}$ on $S$ is integral, i.e. it comes as extension of scalars from a local system of projective $\Oo_L$-modules of finite type (for some number field $L\subset \C$), whenever it is:
\begin{enumerate}
\item Irreducible;
\item Quasi-unipotent local monodromies around the components at infinity of a compactification with normal crossings divisor $i:S \hookrightarrow \overline{S}$;
\item Cohomologically rigid, that is $\Hh^1(\overline{S}, i_{!* } \End^0(\mathbb{V}))$ vanishes;
\item Of finite determinant.
\end{enumerate}
\end{thm}
Here $ i_{!* } \End^0(\mathbb{V})$ denotes the intermediate extension seen as a perverse sheaf. 
See \cite[Remark 2.4]{MR3874695} for more details. Moreover $\Hh^1(\overline{S}, i_{!* } \End^0(\mathbb{V}))$ is the Zariski tangent space at the moduli point of $\mathbb{V}$ of the Betti moduli stack of complex local systems of given rank with prescribed determinant and prescribed local monodromies along the components of the normal crossing divisor $\overline{S}- i(S)$.

\begin{proof}[Proof of Theorem \ref{integrality}]
In the proof we are free to replace $\Gamma$ by a finite index subgroup, and so may and do assume that, as in Section \ref{toroidalsubse}, $\Gamma$ is torsion free and that the parabolic subgroups of $\Gamma$ are unipotent. Moreover, as explained before, we assume also that $\Gamma$ lifts to $\SU(1,n)$.

Let $S_\Gamma$ be the ball quotient $\Gamma \backslash X$. Let $\mathbb{V}$ be the local system on $S_\Gamma $ associated to the standard representation of $\SU(1,n)$, as in (\ref{locsystem}). Thanks to Bass--Serre theory \cite{MR586867} (or \cite[Theorem 2]{MR0279206}), $\mathbb{V}$ is integral in the sense of Theorem \ref{esgro} if and only if the image of every $\gamma \in \Gamma$ in $\Gl(V_\C)$ has traces in the ring of integers of some number field $E \subset \C$ (see also \cite[Lemma 7.1]{MR2457528}). This implies that the traces of $\Ad (\gamma)$ are in $\Oo_K$, where $K$ is the trace field of $\Gamma$ (which will be proven to be a totally real number field in the next section).

As recalled in Section \ref{compac}, $S_\Gamma$ is a smooth quasi-projective variety which admits a smooth toroidal compactification 
\begin{displaymath}
i: S_\Gamma \hookrightarrow \overline{S_\Gamma},
\end{displaymath}
with smooth boundary\footnote{As recalled in the first paragraph of Section \ref{compac}, the boundary is actually a disjoint union of $N$ abelian varieties. Moreover, since $G$ has rank one, the toroidal compactification of $S_\Gamma$ does not depend on any choices.}. By construction the local system $\End^0(\mathbb{V})$ corresponds to the adjoint representation described in equation \eqref{adjointrep}.

Notice that, since $\Gamma$ is irreducible and Zariski dense in $G$, also $\mathbb{V}$ is irreducible (this of course depends on our choice of the faithful representation of $\SU(1,n)$ in $\Gl (V_\C)$). To prove the integrality of $\mathbb{V}$ we are left to check conditions (2),  (3) and (4) of Theorem \ref{esgro}. The fourth condition is satisfied, since every element of $\SU(1,n)$ has finite determinant.

\begin{proof}[Proof of (2)]
Let $B$ be $\overline{S_\Gamma} -i(S_\Gamma)$. Write $B$ as union of its $N$ disjoint irreducible components: $B=\bigcup _i B_i$. We notice here that the singular locus of $B$ is empty, and therefore $\overline{S_\Gamma}$ is equal to what is denoted by $U$ in \cite[Section 2]{MR3874695}. Let $P_i$ be $\Gamma$-rational parabolic corresponding to $B_i$. As proven in Proposition \ref{quasiunipot} the local monodromy at every $B_i$ corresponds to an element $T_i\in \Gamma\cap U_{P_i} $, which is certainly unipotent.
\end{proof}

\begin{proof}[Proof of (3)]
Here we show that $\mathbb{V}$ is cohomologically rigid (without boundary conditions). Recall that
\begin{displaymath}
H^1(S_\Gamma , \End^0(\mathbb{V}) )\cong H^1(\Gamma, \Ad)=0,
\end{displaymath}
where the last equality follows from Theorem \ref{garalndragvanishing}. To show that $\Hh^1(\overline{S_\Gamma}, i_{!* } \End^0(\mathbb{V}))$ vanishes, it is enough to observe that it injects in $H^1(S_\Gamma , \End^0(\mathbb{V}))$. This follows from the description of the group 
\begin{displaymath}
\Hh^1(\overline{S_\Gamma}, i_{!* } \End^0(\mathbb{V}))
\end{displaymath}
 appearing in \cite[Proposition 2.3]{MR3874695}, more precisely see page 4284 line 8, and Remark 2.4 in \emph{op. cit.}.
\end{proof}
Eventually we have checked all the conditions of Theorem \ref{esgro}, concluding the proof of Theorem \ref{integrality}.
\end{proof}

\subsection{Weil restrictions after Deligne and Simpson}\label{weilres}
Let $(G,X,\Gamma)$ be an arbitrary real Shimura datum. Up to replacing $\Gamma$ by a finite index subgroup, we can write $(G,X,\Gamma)$ as a product of $(G_i,X_i, \Gamma_i)$ such that each $ \Gamma_i$ is an irreducible lattice in $G_i$ (as in Proposition \ref{irred}). As recalled in Theorem \ref{margulisrank}, if $G_i$ has rank strictly bigger than one, then $\Gamma_i$ is an arithmetic lattice, which implies that $\Gamma_i$ has integral traces. If $G_i$ is of rank one and isomorphic to $\Gl_2$, RSD0 of Definition \ref{realshimuradef} ensures that $\Gamma_i$ is also an arithmetic lattice. Therefore we may apply Theorem \ref{integrality} to the remaining $G_i$s of rank one, and conclude that each $\Gamma_i$ has integral traces. It follows that we can write, for some real number field $K$, every real Shimura datum as follows:
\begin{displaymath}
(G,X,\Gamma \subset \mathbf{G}(\Oo_K)).
\end{displaymath}
Possibly up to a finite index subgroup, and a conjugation by some $g \in G$. Thanks to the argument explained above, from now on, we assume that $G=\PU(1,n)$ and that $\Gamma$ lifts to $\SU(1,n)$. In particular $\Gamma$ comes with a standard representation on $V_\C$, a $n+1$ dimensional vector space.

After having established Theorem \ref{integrality}, to complete the proof of Theorem \ref{strategy1} we are left to show two things. Namely that the trace field $K$ is totally real (rather than just real), and that we can construct a natural $\Z$-VHS on $S_\Gamma$. For both proofs the main observation is that for each embedding $\sigma: K \to \C$ (or possibly a larger real number field), the local system $\mathbb{V}^\sigma$ associated to the representation
\begin{displaymath}
\rho_\sigma : \pi_1(S_\Gamma)\to \Gl(V_\C)
\end{displaymath}
underlies a polarised complex VHS. The arguments are fairly standard, and essentially due to Simpson, but we recall here the main steps. They appeared for example in \cite[Theorem 5]{MR1179076} and, as stated in \cite{{MR3826464}}, \emph{the basic idea goes back at least to Deligne} \cite{delignetravaux}. For completeness and related discussions we refer also to \cite[Section 10]{MR2457528} and \cite[Proposition 7.1]{MR3826464}. We point out here that, among the ones cited so far, the only paper dealing with quasi-projective varieties, rather than projective ones, is \cite{MR2457528}. For a complete and more general discussion on this issue, we refer to the monograph \cite{zbMATH05119702}. Before stating the results we need, we recall an important theorem of Corlette.
\begin{thm}[{\cite[Theorem 8.1]{MR2457528}}]\label{deformation}
Suppose $\mathbb{V}$ is local system with quasi-unipotent monodromy at infinity. If $\mathbb{V}$ is rigid, then it underlines a $\C$-VHS.
\end{thm}

We are now ready to prove that the field $K$ appearing in Theorem \ref{integrality} is totally real. In the proof we actually work with some real number field $K'$ containing $K$.
\begin{thm}\label{totreal}
Let $\Gamma$ be a lattice in $G=\PU(1,n)$, $n>1$. Then the trace field $K$ of $\Gamma$ is totally real.
\end{thm}

\begin{proof}
 Possibly after conjugating, the local system \eqref{locsystem} $\rho: \pi_1(S_\Gamma)\to \Gl(V_\C)$ takes values in some field $K'$. For each $\sigma : K' \to \C$, consider
\begin{equation}\label{eqsigam}
\rho_\sigma = \sigma \rho : \pi_1(S_\Gamma)\to \Gl(V_\C).
\end{equation}
The first step is to prove that the local system $\mathbb{V}^\sigma$ corresponding to $\rho_\sigma$ is in fact a VHS (either a $K'$ or a $\C$-VHS). Since
\begin{displaymath}
H^1(\Gamma , \Ad \circ \rho_\sigma)=H^1(\Gamma, \Ad) \otimes_{K',\sigma} \C=0,
\end{displaymath}
the twisted representation $\rho_\sigma$ is again cohomologically rigid (see also \cite[Lemma 6.6]{MR2457528}). Moreover $\mathbb{V}^\sigma$ has also quasi-unipotent monodromy at infinity by construction. Indeed, let $P_i$ be a $\Gamma$-rational parabolic corresponding to an irreducible divisor $B_i$ of the boundary, and denote by $T_i$ the local monodromy for $\mathbb{V}$ and by ${T_i}^\sigma$ the local monodromy for $\mathbb{V}^\sigma$. By construction of $\mathbb{V}^\sigma$, the element ${T_i}^\sigma$ is obtained by $T_i \in \Gamma_{U_{P_i}} \subset \Gamma \subset \mathbf{G}(K')$ by applying $\sigma$ to its entries. Since being unipotent is a geometric condition, it is enough to check that ${T_i}^\sigma \in G_\sigma \otimes \C \cong G \otimes \C$ is unipotent, which holds true because $T_i$ is unipotent (here we let $G_\sigma (\C)$ be complex group $\mathbf{G} \times_{K', \sigma} \operatorname{Spec}(\C)$). Eventually we can apply Theorem \ref{deformation} to conclude that each $\mathbb{V}^\sigma$ is a $\C$-VHS, as claimed above.

The second step is as follows, and it is just a matter of reproducing the argument appearing in \cite[End of page 56 and beginning of page 57]{MR1179076}. For each embedding $\sigma$ of $K'$, let $\overline{\sigma}$ denote the complex conjugated embedding. Since $\rho$ is induced by a polarized $\C$-VHS, we have
\begin{displaymath}
\overline{\sigma}\rho \cong \sigma \rho^*,
\end{displaymath}
where $\rho^*$ denotes the dual representation. As $\rho$ is a unitary representation, and therefore self dual, we observe that 
\begin{displaymath}
\overline{\sigma}\rho \cong \sigma \rho.
\end{displaymath}
By taking the traces, the above equality shows that the subfield of $K'$ generated by the traces of $\rho (\gamma)$, varying $\gamma \in \Gamma$, is totally real. Moreover, as recalled in Section \ref{localrig}, from the work of Vinberg \cite[Theorem 2]{MR0279206} we know that $K'$ contains the adjoint trace field $K_{\rho}$, for the representation $\Ad(\rho)$. Recall that $K$, as in Theorem \ref{integrality}, denotes the trace field of $\Gamma$, where traces are computed after applying the adjoint representation \eqref{adjointrep}. Since $\Gamma$ comes from $\SU(1,n)$, we see that $K_{\rho}=K$, therefore $K$ is naturally a subfield of a totally real number field. In particular it is a totally real number field, concluding the proof of Theorem \ref{totreal}.
\end{proof}

\begin{rmk}
 For each embedding $\sigma : K' \to \R$, let $G_\sigma$ be the group group $\mathbf{G} \times_{K', \sigma}\R$. Since the groups $G_\sigma$ are all isomorphic over $\C$, if $G=\PU(1,n)$, the proof of Theorem \ref{totreal} shows that we can write $G_\sigma$ as $\PU(p_\sigma, q_\sigma)$ for some positive non zero integers $p_\sigma, q_\sigma$ such that $p_\sigma+q_\sigma =n+1$. See also \cite[Lemma 4.5]{MR1179076}.
\end{rmk}
The following theorem is essentially due to Simpson \cite{MR1179076}.
\begin{thm}\label{thmzvhs}
Let $(G,X,\Gamma)$ be a real Shimura datum. There exists a totally real number field $K'$ such that:
\begin{itemize}
\item $K'$ contains $K$ the (adjoint) trace field of $\Gamma$;
\item $\Gamma \subset \mathbf{G}(\Oo_{K'})$;
\item For each embedding $\sigma : K' \to \R$, if we let $\mathbb{V}^\sigma$ be the $K'$-VHS constructed above, then
\begin{equation}\label{constru}
\widehat{\mathbb{V}} := \bigoplus_{\sigma : K'\to \R} \mathbb{V}^\sigma
\end{equation}
has a natural structure of (polarized) $\Q$-VHS. Moreover, by Theorem \ref{integrality} and Bass--Serre theory, it has a natural structure of a (polarised) $\Z$-VHS.
\end{itemize}
\end{thm}
 After having established Theorem \ref{totreal}, the proof of Theorem \ref{thmzvhs} is a simple application of \cite[Theorem 5]{MR1179076} (and the more detailed \cite[Proposition 7.1]{MR3826464}). The proof of Theorem \ref{strategy1} is eventually completed, and in Section \ref{monodromymtsection} we will explain how to get rid of $K'$ and work only with $K$.

\subsection{Integral variations of Hodge structure and period maps}\label{section46}
In this section we discuss period maps associated to $\Z$-VHS, Mumford--Tate groups and the algebraicity theorem of Cattani, Deligne and Kaplan.
\subsubsection{Recap of Mumford--Tate domains}\label{recapperioddomains}
For more details about period domains and Mumford--Tate domains we refer to \cite{MR2918237, klingler2017hodge, MR0498551}. Let $(H_\Z,q_\Z)$ be a polarised $\Z$-Hodge structure. We start by recalling the definition of \emph{period domains}. Let $\mathbf{R}$ be the $\Q$-algebraic group $\Aut(H_\Q,q_\Q)$. We have:
\begin{itemize}
\item The space $D$ of $q_\Z$-polarised Hodge structures on $H_\Z$ with specified Hodge numbers is an homogeneous space for $R$;
\item Choosing a reference Hodge structure we have $D=R/M$ where $M$ is a subgroup of the compact unitary subgroup $R\cap U(h)$ with respect to the Hodge form $h$ of the reference Hodge structure;
\item $D$ is naturally an open sub-manifold of the flag variety $D^{\vee}$.
\end{itemize}
Let $\mathbf{M}$ be the Mumford--Tate group, as recalled below, of $(H_\Z,q_\Z)$ and $h\in D$ the point corresponding to $(H_\Z,q_\Z)$. The \emph{Mumford--Tate domain} associated to $(H_\Z,q_\Z)$ is the $\mathbf{M}(\R)$-orbit of $h$ in the full period domain of polarized Hodge structures constructed above.

Let $S$ be a smooth connected quasi-projective complex variety. By \emph{period map}
\begin{displaymath}
S^{\an}\to \mathbf{R}(\Z) \backslash D
\end{displaymath} 
we mean a holomorphic locally liftable Griffiths transverse map. It is equivalent to the datum of a (polarised) $\Z$-VHS on $S$.

For future reference and to complete our brief introduction of variational Hodge theory, we recall here the definition of Mumford--Tate group. Let $\mathbb{V}=(\mathcal{V},\F,\mathcal{Q})$ a polarized variation of $\Q$-Hodge structure on $S$. Let $\lambda: \widetilde{S}\to S$ be the universal cover of $S$ and fix a trivialisation $\lambda^* \mathcal{V}\cong \widetilde{S}\times V$. Let $s\in S$, the \emph{Mumford--Tate group at s}, denoted by $\MT(s) \subset \Gl (\mathcal{V}_s)$, is the smallest $\Q$-algebraic subgroup $\mathbf{M} \subset \Gl (\mathcal{V}_s)$ such that the map 
\begin{displaymath}
h_s : \DT \to \Gl (V_{s,\R})
\end{displaymath}
describing the Hodge-structure on $V_{s}$ factors trough $\mathbf{M}_\R$. Choosing a point $\tilde{s}\in \lambda^{-1}(s)\subset \widetilde{S}$ we obtain an injective homomorphism $\MT(s) \subset \Gl(V)$.

It is well known that there exists a countable union $\Sigma \subsetneq S$ of proper analytic subspaces of $S$ such that:
\begin{itemize}
\item For $s\in S - \Sigma $, $\MT(s) \subset \Gl(V)$ does not depend on $s$, nor on the choice of $\tilde{s}$. We call this group \emph{the generic Mumford-Tate group of} $\mathbb{V}$;
\item For all $s$ and $\tilde{s}$ as above, with $s\in \Sigma$, $\MT(s)$ is a proper subgroup of $G$, the generic Mumford-Tate group of $\mathbb{V}$.
\end{itemize}

\subsubsection{Generalised modular embeddings}\label{diagram}
From the real Shimura datum $(G,X,\Gamma \subset \mathbf{G}(\Oo_K))$, in Theorem \ref{thmzvhs} we produced a $\Z$-VHS $\widehat{\mathbb{V}}$ on $S_\Gamma$. By construction, the generic Mumford--Tate group of $\widehat{\mathbb{V}}$ is a reductive $\Q$-subgroup of the Weil restriction from $K'$ to $\Q$ of $\mathbf{G}_{K'}=  \mathbf{G} \otimes_{K}K'$, which we denote by $\widehat{\mathbf{G}}'$. Regarding Weil restrictions, we refer to Section \ref{weilresbasic}.

 We now specialise the definitions of Section \ref{recapperioddomains} to our setting. As usual, we set $\widehat{G}'=\widehat{\mathbf{G}'}(\R)^+$. Associated to the pair $(S_\Gamma, \widehat{\mathbb{V}})$, we have a Mumford--Tate domain $(\widehat{\mathbf{G}}',D'=D_{\widehat{G}'}=\widehat{G}'/M')$, where $M'$ a compact subgroup of $\widehat{G}'$. As in Section \ref{weilresbasic}, we write:
\begin{itemize}
\item $r'$ for the degree of $K'$ over $\Q$;
\item $\sigma_i : K' \to \R$ the real embeddings of $K'$, ordered in such a way that $\sigma_1$ is simply the identity on $K'$;
\item $\pr_i:\widehat{\mathbf{G}}' \to \mathbf{G}'$ for the map of $K'$-group schemes obtained by projecting onto the $i$th factor of $\widehat{\mathbf{G}}'\otimes K'$.
\end{itemize}
Let $\psi : {S_\Gamma}^{\an} \to \widehat{\mathbf{G}}'(\Z)\backslash D'$ be the period map associated to the $\Z$-VHS $\widehat{\mathbb{V}}$ we constructed in \eqref{constru}. We have a commutative diagram
\begin{center}
\begin{tikzpicture}[scale=2]
\node (A) at (-1,1) {$X$};
\node (B) at (1,1) {$D'=D_{\widehat{G}'}$};
\node (C) at (-1,0) {${S_\Gamma}^{\an}$};
\node (D) at (1,0) {$\widehat{\mathbf{G}}'(\Z)\backslash D'$};
\path[right hook->,font=\scriptsize]
(A) edge node[above]{$\tilde{\psi}$} (B);
\path[->,font=\scriptsize,>=angle 90]
(C) edge node[above]{$\psi$} (D)
(A) edge node[right]{$\pi_{\Gamma}$} (C)
(B) edge node[right]{$\pi_{\Z}'$} (D);
\end{tikzpicture}.
\end{center}
By construction $D_{\widehat{G}'}$ is the product of $X$ and a homogeneous space under the group $\prod_{i=2,\dots, r'}G_{\sigma_i} $. 
Given such a decomposition of $ D_{\widehat{G}'}$ we can write:
\begin{displaymath}
\tilde{\psi}(x) = (x, x_{\sigma_2}, \dots, x_{\sigma_{r'}}),
\end{displaymath}
where $x_{\sigma_i}$ is the Hodge structure obtained by the fibre of $\mathbb{V}^{\sigma_i}$, or more precisely of its lift to $X$, at $x$. Finally it is important to observe that $\tilde{\psi}$ is holomorphic and $\Gamma$-equivariant, in the sense that for each $\gamma \in \Gamma$
\begin{displaymath}
\tilde{\psi}(\gamma x) = (\gamma x, \rho_{\sigma_2}(\gamma) x_{\sigma_2}, \dots,\rho_{\sigma_{r'}}(\gamma) x_{\sigma_{r'}}),
\end{displaymath}
where $\rho_{\sigma_i}$, as introduced in \eqref{eqsigam}, is obtained by applying $\sigma_i: K' \to \R$ to the coefficients of $\Gamma$.

\subsubsection{Definition of \texorpdfstring{$\Z$}{z}-special subvarieties}\label{defzsp}
We specialise the definitions of Klingler \cite{klingler2017hodge} for the period map $\psi :{S_\Gamma}^{\an} \to \widehat{\mathbf{G}}'(\Z)\backslash D'$ associated to the $\Z$-VHS $\widehat{\mathbb{V}}$ when $G=\PU(1,n)$ for some $n>1$. In \cite[Section 1.2]{MR3958791} the following spaces are called \emph{weak Mumford--Tate domain}. We omit the word \emph{weak} because we are ultimately interested in subvarieties of $S_\Gamma$, where there are no possibilities of moving $\Z$-special subvarieties in families.

\begin{defi}
 A \emph{Mumford--Tate sub-domain} $D_0$ of $D'$ is an orbit $M.x$ where $x\in D'$ and $M$ is the real group associated to $\mathbf{M}$ a normal algebraic $\Q$-subgroup of $\MT(x)$. In fact $D_0$ is a complex sub-manifold of $D'$ and $\pi(D_0)=\mathbf{M}(\Z)\backslash M.x$ is a complex analytic subvariety of $\widehat{\mathbf{G}}'(\Z)\backslash D'$, which we call a \emph{Mumford--Tate subvariety}. 
\end{defi}

The following is a celebrated result of Cattani, Deligne and Kaplan, see also the recent work of Bakker, Klingler and Tsimerman \cite[Theorem 1.6]{bakker2018tame}. It will be crucial in proving Corollary \ref{cor1}.
\begin{thm}[{\cite[Theorem 1.1]{MR1273413}}]\label{cdk}
Let $D_0 \subset D'$ be a Mumford--Tate sub-domain. Then the set $\psi ^{-1}(\pi (D_0))$ is an algebraic subvariety of $S_\Gamma$.
\end{thm}

\begin{defi}
Subvarieties of $S_\Gamma$ of the form $\psi ^{-1}(\pi (D_0))^0$, for some Mumford--Tate sub-domain $D_0\subset D'$ are called $\Z$-\emph{special} (here $(-)^0$ denotes some irreducible component). An irreducible subvariety $W$ of $S_\Gamma$ is $\Z$-\emph{Hodge generic} if it is not contained in any strict $\Z$-special subvariety.
\end{defi}

By definition, $S_\Gamma$ is always $\Z$-special. In the next section we will see that $\Gamma$-special subvarieties (see Definition \ref{gammaspeci}) give examples of (strict) $\Z$-special subvarieties (in the sense of Definition \ref{defzsp}), and vice versa. Moreover we explicitly describe the $\Z$-special subvarieties associated to the $\Gamma$-special ones.

\section{Comparison between the \texorpdfstring{$\Gamma$}{gamma}-special and \texorpdfstring{$\Z$}{z}-special structures}\label{comparisonsection}
Let $(G=\PU(1,n),X, \Gamma \subset \mathbf{G}(\Oo_K))$ be a real Shimura datum. In this section we introduce and compare two notions of \emph{monodromy groups} for subvarieties of $S_\Gamma= \Gamma \backslash X$: on the one hand with respect to the natural $\Oo_{K}$-VHS, and, on the other, to the $\Z$-VHS $\widehat{\mathbb{V}}$ of Theorem \ref{thmzvhs}. Thanks to the Deligne--André Theorem, such comparison translates to a computation of the Mumford--Tate group of $\widehat{\mathbb{V}}$. The main result of the section is Theorem \ref{mainthm}, which will be used to prove Theorem \ref{thmbader} in Section \ref{sectionfinit}. We also prove the non-arithmetic Ax-Schanuel conjecture in Section \ref{proofofas}. Finally several corollaries, which may be of independent interest, are discussed.

\subsection{Computing the monodromy for the \texorpdfstring{$\Z$}{z}-VHS}\label{section51}
Recall that $\Gamma$ is assumed to be torsion free, and that it defines a $K$-form $\mathbf{G}$ of $G$, where $K$ denotes the (adjoint) trace field of $\Gamma$. Let $K'$ be the totally real field of Theorem \ref{thmzvhs} (which contains $K$), and denote the Weil restriction from $K'$ to $\Q$ of $\mathbf{G}_{K'}=  \mathbf{G} \otimes_{K}K'$ by $\widehat{\mathbf{G}}'$. It naturally contains $\widehat{\mathbf{G}}$, the Weil restriction from $K$ to $\Q$ of $\mathbf{G}$.

We consider the real representation of $\Gamma \cong \pi_1(S_\Gamma)$ associated to $\widehat{\mathbb{V}}$, which, by construction, is given by 
\begin{equation}\label{equationong}
\rho_{\widehat{\mathbb{V}}}: \Gamma \to  \widehat{G}'=\prod _{\sigma \in \Omega_\infty} G_\sigma, \ \ \gamma \mapsto (\sigma_1(\gamma)=\gamma, \dots , \sigma_{r'} (\gamma)),
\end{equation}
where, as in Section \ref{diagram}, $\sigma_i$ are the $r'$ real embeddings of $K'$. In this section we compute the algebraic monodromy of subvarieties of $S_\Gamma$, for the $\Z$-VHS $\widehat{\mathbb{V}}$. Among other things, we prove that $\rho_{\widehat{\mathbb{V}}}$ has Zariski dense image in $\widehat{\mathbf{G}} \subset \widehat{\mathbf{G}}'$.

Let $W$ be an irreducible (closed) algebraic subvariety of $S_\Gamma$ of positive dimension, and let $\widetilde{W}$ be an analytic component of $\pi^{-1}(W)$, where $\pi : X \to S_\Gamma$ denotes the quotient map. As in Section \ref{zhodgesection}, we omit the base point in the notation of the fundamental group and we replace $W$ by its smooth locus. 

We denote as follows the $\Gamma$-\emph{monodromy} of $W$, that is the monodromy of $W$ with respect to the $K$-local system $\mathbb{V}$:
\begin{displaymath}
\leftidx{^\Gamma} {\mathbf{M}}_W:= \overline{\im (\pi_1(W)\to \pi_1(S_\Gamma)=\Gamma)}^{\Zar,0}\subset \mathbf{G},
\end{displaymath}
that is the identity component of the Zariski closure in $\mathbf{G}$ of the image of $\pi_1(W)$.
\begin{lemma}\label{lemmanoncomp}
The $K$-algebraic group $\leftidx{^\Gamma} {\mathbf{M}}_W$ is non-trivial and its real points are not compact.
\end{lemma}
\begin{proof}
Heading for a contradiction suppose that the intersection between $\Gamma$ and $\leftidx{^\Gamma} {M}_W$ is finite. By Riemann existence theorem, this implies that $\widetilde{W}$ is algebraic and therefore does not admit bounded holomorphic functions. This, however, contradicts the fact that $X=\mathbb{B}^n$ admits a bounded realisation. Finally, since $\Gamma \cap \leftidx{^\Gamma} {M}_W $ is a discrete and infinite subgroup of $\leftidx{^\Gamma} {M}_W$, $\leftidx{^\Gamma} {M}_W$ can not be compact.
\end{proof}
The group $\leftidx{^\Gamma} {\mathbf{M}}_W$ is defined over $K$, the adjoint trace field of $\Gamma$, but it could happen that it is defined over a smaller field. In the next definition we look at the field of definition of the adjoint of $\leftidx{^\Gamma} {\mathbf{M}}_W$, in the sense of Section \ref{localrig}.
\begin{defi}\label{deftracefield}
The field generated by the traces of the elements in $\Gamma \cap \leftidx{^\Gamma} {M}_W $, under the adjoint representation (of $M_W$), is called the \emph{trace field of} $W$ and denoted by $K_W$.
\end{defi}
\begin{rmk}\label{rmknew}
The field $K_W$ is naturally a sub-field of $K$. Indeed, by construction, the field
\begin{displaymath}
L:= \Q\{\tr \ad_{|{M_W}}(\gamma): \gamma \in \Gamma \cap \leftidx{^\Gamma} {M}_W \}
\end{displaymath}
is a sub-field of $K$, where $\ad_{|{M_W}}$ denotes the restriction of the adjoint representation of $G$ to $M_W$. By Vinberg's theorem \cite{MR0279206}, as explained in Section \ref{localrig}, we have that the adjoint trace field of $\Gamma \cap \leftidx{^\Gamma} {M}_W$ is a sub-field of $L$, and therefore of $K$.
\end{rmk}

We define the $\Z$-\emph{monodromy} of $W$, denoted by $\leftidx{^\Z} {\mathbf{M}}_W$, as the monodromy of $W$ with respect to the $\Z$-VHS $\widehat{\mathbb{V}}$ we constructed in Theorem \ref{thmzvhs}. That is
\begin{displaymath}
\leftidx{^\Z} {\mathbf{M}}_W:=\overline{\im (\pi_1(W)\to \widehat{\mathbf{G}}'(\Z))}^{\Zar,0}\subset \widehat{\mathbf{G}}'.
\end{displaymath}
It is a connected $\Q$-subgroup of $\widehat{\mathbf{G}}$, the Weil restriction from $K$ to $\Q$ of $\mathbf{G}$, which is seen diagonally as a subgroup of $\widehat{\mathbf{G}}'$.

We record here two simple observations relating the two ``special structures'', relying only on the functoriality of the fundamental group:
\begin{itemize}
\item If $W$ is contained in a $\Gamma$-special subvariety (as introduced in Definition \ref{gammaspeci}) associated to $(H,X_H,\Gamma_H\subset \mathbf{H}(\Oo_K))$, then the $\Gamma$-monodromy of $W$ is contained in $\mathbf{H}$;
\item If $W$ is contained in a $\Z$-special subvariety (as introduced in Definition \ref{defzsp}) associated to $(R,D_R, \mathbf{R}(\Z))$, then the $\Z$-monodromy of $W$ is contained in $\mathbf{R}$.
\end{itemize}

We define $\widehat{\leftidx{^\Gamma} {\mathbf{M}}_W}$ as the Weil restriction from $K_W $ to $\Q$ of the adjoint of $\leftidx{^\Gamma} {\mathbf{M}}_W$. It is a $\Q$-algebraic group whose real points can be identified as the product of the $ \left( \leftidx{^\Gamma} {\mathbf{M}}^{\ad}_W\right) _\sigma$, varying $\sigma: K_W \to \R$. Each factor has a structure of $K_W$-algebraic group. Since $K_W$ is contained in $K$, as explained in Remark \ref{rmknew}, the factors have also a structure of $K$-group scheme.

The following will be crucial for the future arguments and it is the first step towards the proof of Theorem \ref{mainthm}.
\begin{thm}\label{thmmonodromy}
The $\Z$-monodromy of $W$ is the Weil restriction from $K_W$ to $\Q$ of the $\Gamma$-monodromy of $W$. In symbols we write:
\begin{displaymath}
\widehat{\leftidx{^\Gamma} {\mathbf{M}}_W}=\leftidx{^\Z} {\mathbf{M}}^{\ad}_W.
\end{displaymath}
\end{thm}

The representation
\begin{displaymath}
\pi_1(W)\to \leftidx{^\Gamma} {\mathbf{M}}^{\ad}_W,
\end{displaymath}
associated to the restriction of the $\Z$-VHS $\widehat{\mathbb{V}}$ from $S_\Gamma$ to (the smooth locus of) $W$ is obtained, by equation \eqref{equationong}, using each embedding $\sigma : K \to \R$:
\begin{displaymath}
\pi_1(W)\to \leftidx{^\Gamma} {\mathbf{M}}^{\ad}_W \to \widehat{\leftidx{^\Gamma} {\mathbf{M}}_W}.
\end{displaymath}
In particular the adjoint of $\leftidx{^\Z} {\mathbf{M}}_W$ is contained in $\widehat{\leftidx{^\Gamma} {\mathbf{M}}_W}$. 

\begin{proof}
If $K_W=\Q$, there is nothing to prove, so we may assume that $[K_W:\Q]>1$. For each $\sigma : K_W \to \R$, denote by $\mathbf{M}_\sigma$ the image of the adjoint of $\leftidx{^\Z} {\mathbf{M}}_W$ along the projection 
\begin{displaymath}
\pr_\sigma :  \widehat{\mathbf{G}}\to \mathbf{G}_\sigma.
\end{displaymath}
We first prove that $\mathbf{M}_\sigma$ is the adjoint of $\left( \leftidx{^\Gamma} {\mathbf{M}}_W\right) _\sigma$. By equation \eqref{equationong}, we have that $\mathbf{M}_\sigma$ contains the image of the monodromy representation associated to the VHS $\mathbb{V}^\sigma_{| W}$, which is 
\begin{displaymath}
g_W: \pi_1 (W) \to \left( \leftidx{^\Gamma} {\mathbf{M}}_W\right) _\sigma.
\end{displaymath} 
Since the image of $\pi_1(W)$ is Zariski dense in $ \leftidx{^\Gamma}{\mathbf{M}}_W$, we have that the image of $g_W$ is contained in $\mathbf{M}_\sigma$ and it is Zariski dense in $\left( \leftidx{^\Gamma} {\mathbf{M}}_W\right) _\sigma$. It follows that
\begin{equation}\label{mteq}
\left( \leftidx{^\Gamma}{\mathbf{M}}_W\right)^{\ad} _\sigma = \mathbf{M}_\sigma.
\end{equation}

For any two of the $r_W$ archimedean places $\sigma_i, \sigma_j$ of $K_W$ we prove that $\leftidx{^\Z} {\mathbf{M}}_W$ surjects onto
\begin{displaymath}
\mathbf{M} _{\sigma_i} \times\mathbf{M} _{\sigma_j}.
\end{displaymath}
This is enough to conclude when $r_W=2$, and then we argue by induction. Let $\mathbf{H}$ be the intersection of $\leftidx{^\Z} {\mathbf{M}}_W$ with $\mathbf{M} _{\sigma_i} \times\mathbf{M} _{\sigma_j}$. It is an algebraic group whose minimal field of definition is $K_W$.

 For $\ell \in \{i,j\}$, we have a short exact sequences (in the category of algebraic $K_W$-groups)
\begin{displaymath}
1 \to K_\ell \to \mathbf{H} \to  \mathbf{M} _{\sigma_\ell}\to 1.
\end{displaymath}
Since we are working with semisimple group, the only possibilities for $K_\ell $ is to be either finite or equal to $ \mathbf{M} _{\sigma_{\ell'}} $, for $\ell '$ such that $\{\ell, \ell '\}=\{i,j\}$. If $K_\ell = \mathbf{M} _{\sigma_{\ell'}}$ then 
\begin{displaymath}
\mathbf{H}=\mathbf{M} _{\sigma_i} \times\mathbf{M} _{\sigma_j},
\end{displaymath}
proving the claim. If both kernels are finite we argue using Goursat's Lemma, as recalled below.
\begin{lemma}[Goursat's Lemma]\label{goursat}
Let $G_1, G_2$ be algebraic groups and $H$ be a closed subgroup of $G_1 \times G_2$ such that the two projections from $H$ to $G_1$ and $G_2$ are surjective. For $i=1,2$, let $K_i$ be the kernel of $H\to G_i$. The image of $H$ in $G_1/K_2 \times G_2/K_1$ is the graph of an isomorphism $G_1/K_2 \to G_2/K_1$.
\end{lemma}
Assume that both $K_i$ and $K_j$ are finite groups. By applying Lemma \ref{goursat} we can write $\mathbf{H}$ as the graph of an isomorphism of $K_W$ group schemes
\begin{displaymath}
f: \mathbf{M} _{\sigma_i}/K_{j}\to \mathbf{M} _{\sigma_j}/K_{i}.
\end{displaymath}
In particular, up to a quotient by a finite group, we obtain a commutative square
\begin{center}
\begin{tikzpicture}[scale=2]
\node (A) at (-1,1) {$\mathbf{M} _{\sigma_i} (K_W)$};
\node (B) at (1,1) {$\mathbf{M} _{\sigma_j} (K_W)$};
\node (C) at (-1,0) {$ \mathbf{M} _{\sigma_i} (\R)$};
\node (D) at (1,0) {$\mathbf{M} _{\sigma_j} (\R)$};
\path[->,font=\scriptsize,>=angle 90]
(A) edge node[above]{$f_{K_W}$} (B)
(C) edge node[above]{$f_\R$} (D)
(A) edge node[right]{$\sigma_i$} (C)
(B) edge node[right]{$\sigma_j$} (D);
\end{tikzpicture}.
\end{center}

The diagram induces an equivalence between the two places $\sigma_i, \sigma_j$. Indeed, since $K_W$ is the field generated by the traces of the $\gamma \in \Gamma \cap \leftidx{^\Gamma} {M}_W$, and 
\begin{displaymath}
\Q\{ \tr (\ad (\gamma)) : \gamma \in \sigma_i(\Gamma \cap \leftidx{^\Gamma} {M}_W) \} = \Q\{ \tr (\ad (\gamma)) : \gamma \in \sigma_j(\Gamma \cap \leftidx{^\Gamma} {M}_W) \}
\end{displaymath}
we obtain that $\sigma_i(K_W)=\sigma_j(K_W)$. This is impossible because $\sigma_i, \sigma_j$ are places of $K_W$, the smallest field of definition of $\leftidx{^\Gamma} {\mathbf{M}}_W$.

To finish the proof we argue by induction. Assume that $\leftidx{^\Z} {\mathbf{M}}_W$ surjects onto any products of $m$ factors, for some $m<r_W$. We have to prove that it surjects onto any products of $m+1$ factors. Let $S$ be a set of $m+1$ places of $K_W$ and write $S=S'\cup \{\sigma_{\ell}\}$ for some $\sigma_\ell \in S$. Set
\begin{displaymath}
 \mathbf{M} _{S'} := \prod_{\sigma\in S'}  \mathbf{M} _{\sigma},\ \ \  \mathbf{M} _{S} := \prod_{\sigma\in S}  \mathbf{M} _{\sigma},
\end{displaymath}
and consider the exact sequence
\begin{displaymath}
1 \to K_{S'} \to \mathbf{H} \to  \mathbf{M} _{S'}\to 1.
\end{displaymath}
The kernel $K_{S'}$ is either finite, or equal to $\mathbf{M} _{\sigma_\ell}$. In the latter case, we conclude that $\mathbf{H}=\mathbf{M} _{S} $. In the former case, consider the short exact sequence
\begin{displaymath}
1 \to K_\ell \to \mathbf{H} \to  \mathbf{M} _{\sigma_\ell}\to 1.
\end{displaymath}
By applying Lemma \ref{goursat}, we can have an isomorphism of $K$-group schemes
\begin{displaymath}
\mathbf{M} _{\sigma_\ell} \cong \mathbf{M}_{S'}/ K_\ell.
\end{displaymath}
Since $M_{S'}$ is a product of simple groups, the right hand side has to be isomorphic to (a finite quotient of) $\mathbf{M} _{\sigma_{\ell'}}$, for some $\ell ' \in S'$. We therefore obtain the same commutative square as in the case where $S$ had cardinality two, concluding the proof of Theorem \ref{thmmonodromy}.
\end{proof}

\subsection{Monodromy and Mumford--Tate groups}\label{monodromymtsection}
We are ready to compute the generic Mumford--Tate groups for $\widehat{\mathbb{V}}$ restricted to subvarieties of $S_\Gamma$, in the sense of Section \ref{recapperioddomains}. We first recall the following important theorem which relates the monodromy and the Mumford--Tate groups. It can be found in \cite[Theorem 1]{MR1154159} and \cite{MR0498551} (see also \cite[Theorem 4.10]{MR3821177}).
\begin{thm}[Deligne, André]\label{deligneandre}
Let $W$ a smooth quasi-projective variety supporting a $\Z$-VHS and $s \in W$ a Hodge generic point. Let $\mathbf{M}_s$ be the algebraic monodromy of $W$ (at $s$). Then $\mathbf{M}_s$ is a normal subgroup of the derived subgroup $\MT(s)^{\der}$ of the Mumford--Tate group at $s$. In symbols we simply write $\mathbf{M}_s \triangleleft \MT(s)^{\der}$.
\end{thm}

Combining Theorem \ref{deligneandre} and Theorem \ref{thmmonodromy}, we obtain two important corollaries. The former will be used to prove that $\Gamma$-special subvarieties are $\Z$-special, and the latter to establish the converse implication.

\begin{cor}\label{gammadense}
The $\Z$-monodromy group of $S_\Gamma$ is $\widehat{\mathbf{G}}$, i.e. $\Gamma$ has Zariski dense image in $\widehat{G}$. Moreover the derived subgroup of the generic Mumford--Tate group of $\widehat{\mathbb{V}}$ is $\widehat{\mathbf{G}}$.
\end{cor}
\begin{proof}
By definition, the $\Gamma$-monodromy of $S_\Gamma$ is given by $\mathbf{G}$. Theorem \ref{thmmonodromy}, applied to $W=S_\Gamma$, shows that the $\Z$-monodromy of $S_\Gamma$ is $\widehat{\mathbf{G}}$, since $K_W$ is just $K$, the trace field of $\Gamma$. By Theorem \ref{deligneandre}, we conclude that generic Mumford--Tate group of $\widehat{\mathbb{V}}$ is $\widehat{\mathbf{G}}$.
\end{proof}

\begin{rmk}\label{remaknew}
As a direct consequence of Corollary \ref{gammadense}, we have the following. Let $S_{\Gamma_H}$ be a $\Gamma$-special subvariety of $S_\Gamma$ associated to a triplet $(H,X_H,\Gamma_H\subset \mathbf{H}(\Oo_K))\subset (G,X,\Gamma\subset \mathbf{G}(\Oo_K))$. The Mumford--Tate group of $\widehat{\mathbb{V}}$ restricted to the (smooth locus of) $S_{\Gamma_H}$ is the Weil restriction from $K_H$ to $\Q$ of $\mathbf{H}$, where $K_H$ is the subfield of $K$ given by the trace field of the lattice $\Gamma_H$.
\end{rmk}

Now that we have finally computed the generic Mumford--Tate group of $\widehat{\mathbb{V}}$, we can upgrade the diagram presented in Section \ref{diagram}, to a one targeted in $(\widehat{\mathbf{G}},D_{\widehat{G}})$ (rather than its decorated version). After having explained this, $K'$ and $\widehat{\mathbf{G}}'$ will no longer appear in the paper. First notice that
\begin{displaymath}
\widehat{G}'= \prod _{\sigma : K \to \R}{ G_\sigma}^{[K':K]}.
\end{displaymath}
Write
\begin{displaymath}
\widehat{\mathbb{V}} = \bigoplus_{\sigma : K \to \R} \left(  \bigoplus_{ \ell : K' \to \R , \ \ell | \sigma} \mathbb{V}^\ell \right),
\end{displaymath}
and notice that the original $K$-VHS $ \mathbb{V}$ is one of the factors appearing above. We work at the places $\ell$ above ``the identity'' $\sigma_1 : K \to \R$. The lift of the period map $\tilde{\psi}$ induces a holomorphic map $\tilde{\psi}_{\id}$ from $X$ to a homogeneous space under $G^{[K':K]}$, that we denote by $D_{\id}$. The image of $\tilde{\psi}_{\id}$ lies in a $\Delta(G)$-orbit, where 
\begin{displaymath}
\Delta : G \to G^{[K':K]}
\end{displaymath}
is the diagonal embedding. Since, in the coordinates of $D'$, we have a component of $\tilde{\psi}$ corresponding to the identity, we have that, for any $g\in G$
\begin{displaymath}
\tilde{\psi}_{\id}(gx)=(gx,gx_{\ell_2}, \dots , gx_{\ell_{[K':K]}}).
\end{displaymath}
As a consequence of the above $G$-equivariance, we observe that
\begin{displaymath}
\tilde{\psi}_{\id}: X \to D_{\id} \cong X\times \cdots \times X,
\end{displaymath}
where on the right hand side we have $[K':K]$-factors. In such coordinates we can indeed write
\begin{displaymath}
\tilde{\psi}_{\id}(x)=(x, g_{\ell_1}x, \dots, g_{\ell_{[K':K]}}x).
\end{displaymath}
By choosing the opportune biholomorphism $ D_{\id} \cong X^{[K':K]}$, $\tilde{\psi}_{\id}$ becomes just the diagonal embedding. This explains why, from now on, $K'$ and $\widehat{\mathbf{G}}'$ will no longer appear. We therefore refine the notations of Section \ref{diagram} by ``undecorating it''; we let:
\begin{itemize}
\item $r$ be the degree of $K$ over $\Q$;
\item $\sigma_i : K \to \R$ the real embeddings of $K$, ordered in such a way that $\sigma_1$ is simply the identity on $K$;
\end{itemize}
and we let $\psi : {S_\Gamma}^{\an} \to \widehat{\mathbf{G}}(\Z)\backslash D$ be the period map associated to the $\Z$-VHS $\widehat{\mathbb{V}}$ we constructed in \eqref{constru}. We have a commutative diagram
\begin{center}
\begin{tikzpicture}[scale=2]
\node (A) at (-1,1) {$X$};
\node (B) at (1,1) {$D=D_{\widehat{G}}$};
\node (C) at (-1,0) {${S_\Gamma}^{\an}$};
\node (D) at (1,0) {$\widehat{\mathbf{G}}(\Z)\backslash D$};
\path[right hook->,font=\scriptsize]
(A) edge node[above]{$\tilde{\psi}$} (B);
\path[->,font=\scriptsize,>=angle 90]
(C) edge node[above]{$\psi$} (D)
(A) edge node[right]{$\pi_{\Gamma}$} (C)
(B) edge node[right]{$\pi_{\Z}$} (D);
\end{tikzpicture},
\end{center}
in such a way that $D\cong X \times D_{>1}$, where $D_{>1}$ is a homogeneous space under $\prod_{i>1}G_{\sigma_i}$, and the first component of $\tilde{\psi}$ is simply the identity.

\begin{rmk}
We notice here that, by Theorem \ref{arithcrit}, the group $\prod_{i=2,\dots, r}G_{\sigma_i} $ is compact if and only if $\Gamma$ is an arithmetic lattice. This will be a fundamental step in Section \ref{sectionfinit}, whose starting point is the following:
\begin{center}
$\Gamma$ is non-arithmetic if and only if $\codim_{\widehat{\mathbf{G}}(\Z)\backslash D} ({\psi({S_\Gamma}^{\an})})>0$.
\end{center}
\end{rmk}

\begin{cor}\label{cor123}
Let $\mathbf{H}/\Q$ be the the generic Mumford--Tate group of a subvariety (of positive dimension) $W\subset S_\Gamma$. Then $\mathbf{H}/\Q$ is a Weil restriction from $K_W$ to $\Q$ of some subgroup $\mathbf{F}$ of $\mathbf{G}$.
\end{cor}
\begin{proof}
Let $W$ be a (smooth, irreducible, closed) subvariety of $S_\Gamma$. Applying Theorem \ref{deligneandre} to $W$, we see that
\begin{displaymath}
\leftidx{^\Z} {\mathbf{M}}_W \triangleleft \MT(w)^{\der}  \subseteq \widehat{\mathbf{G}},
\end{displaymath}
where $w\in W$ is a point whose Mumford--Tate group is $\mathbf{H}$, the generic Mumford--Tate of $W$. That is a Hodge generic point, in the sense of Section \ref{recapperioddomains}. Theorem \ref{thmmonodromy} shows that the left hand side is the Weil restriction, from $K_W$ to $\Q$, of some subgroup $\mathbf{F} \subset \mathbf{G}$. Since the centraliser of $\mathbf{F}_\R$ in $\mathbf{G}_\R$ is finite (being $G$ of rank one), we conclude that 
\begin{displaymath}
\widehat{\mathbf{F}}= \leftidx{^\Z} {\mathbf{M}}_W = \MT(w)^{\der},
\end{displaymath}
as desired.
\end{proof}

\subsection{Proof of Theorem \ref{mainthm}}\label{proofmainthm} We first prove one of the two implications of Theorem \ref{mainthm} (the fact that $\Gamma$-special subvarieties are the same as the totally geodesic ones was already discussed in Section \ref{sectiongeneral}).
 \begin{prop}[$\Gamma$-special $\Rightarrow$ $\Z$-special]\label{Zspecial}
Let $(H,X_H,\Gamma_H\subset \mathbf{H}(\Oo_K))$ be a real sub-Shimura datum of $(G,X,\Gamma\subset \mathbf{G}(\Oo_K))$. Denote by $\widehat{\mathbf{H}}$ the Weil restriction from $K$ to $\Q$ of $\mathbf{H}$. The subvariety $S_{\Gamma_H}$ of $S_\Gamma$ can be written an irreducible component of the preimage along the period map 
\begin{displaymath}
\psi : {S_\Gamma }^{\an}\to \widehat{\mathbf{G}}(\Z) \backslash D_{\widehat{G}}
\end{displaymath}
of the Mumford-Tate sub-domain $\widehat{\mathbf{H}}(\Z) \backslash D_{\widehat{H}} \subset \widehat{\mathbf{G}}(\Z) \backslash D_{\widehat{G}}$. 
\end{prop}
\begin{proof}
As explained in Theorem \ref{algebraic} and Corollary \ref{algmorphism}, $S_{\Gamma_H}$ is an irreducible algebraic subvariety of $S_\Gamma$. We can restrict the $\Z$-VHS $\widehat{\mathbb{V}}$ from $S_\Gamma$ to (the smooth locus of\footnote{The variety $\Gamma_H \backslash X_H$ is smooth (at least when $\Gamma$ is torsion free), but it may be that its image in $S_\Gamma$ is not smooth.}) $S_{\Gamma_H}$. From Corollary \ref{gammadense} (see indeed Remark \ref{remaknew}), we know that the
\begin{equation}\label{eqmtnew}
\MT(\widehat{\mathbb{V}}_{| S_{\Gamma_H}})^{\der}= \widehat{\mathbf{H}}^{\der}.
\end{equation} Here $\widehat{\mathbf{H}}$ denotes the Weil restriction from $K$ to $\Q$ of $\mathbf{H} \subset \mathbf{G}/K$. To be more precise we should restrict from the $K_H$, the trace field of $\Gamma_H$, but, as will be explained in Remark \ref{twofields}, this difference plays no role in the sequel.

From \eqref{eqmtnew}, we see that $\widehat{\mathbf{H}}(\Z) \backslash D_{\widehat{H}}$ is, as claimed in the statement of the proposition, indeed a Mumford-Tate sub-domain of $\widehat{\mathbf{G}}(\Z) \backslash D_{\widehat{G}}$. By construction and the definition of generic Mumford--Tate group recalled in Section \ref{section46}, $\widehat{\mathbf{H}}(\Z) \backslash D_{\widehat{H}}$ is indeed the smallest sub-domain of $\widehat{\mathbf{G}}(\Z) \backslash D_{\widehat{G}}$ containing $\psi(S_{\Gamma_H})$. We therefore have a commutative diagram
\begin{center}
\begin{tikzpicture}[scale=2]
\node (A) at (-1,1) {$X_H$};
\node (B) at (1,1) {$D_{\widehat{H}}\subset  D_{\widehat{G}}$};
\node (C) at (-1,0) {${S_{\Gamma_H}}^{\an}$};
\node (D) at (1,0) {$\widehat{\mathbf{H}}(\Z)\backslash D_{\widehat{H}} \subset \widehat{\mathbf{G}}(\Z) \backslash D_{\widehat{G}}$};
\path[right hook->,font=\scriptsize]
(A) edge node[above]{$\tilde{\psi}_{| X_H}$} (B);
\path[->,font=\scriptsize,>=angle 90]
(C) edge node[above]{$\psi_{| {S_{\Gamma_H}}^{\an}}$} (D)
(A) edge node[right]{} (C)
(B) edge node[right]{$\pi$} (D);
\end{tikzpicture}.
\end{center}
 Observe that 
\begin{displaymath}
\tilde{\psi}^{-1}(D_{\widehat{H}})=X_H.
\end{displaymath}
This is enough to see that 
\begin{displaymath}
S_{\Gamma_H} \subseteq \psi^{-1}(\widehat{\mathbf{H}}(\Z)\backslash D_{\widehat{H}})^0.
\end{displaymath}
The right hand side is a closed irreducible analytical subvariety of ${S_\Gamma}^{\an}$. Since both objects have the same dimension, being locally isomorphic to $X_H$, we conclude that $S_{\Gamma_H}$ is of the desired shape.

\end{proof}

\begin{rmk}\label{twofields}
It could happen that $\mathbf{H}$ is defined over a smaller field (but not over a bigger one, as explained in Remark \ref{rmknew}). The trace field of $\Gamma_H$, which we denote by $K_H=K_{S_{\Gamma_H}}$, may be strictly contained in $K$ (it could even happen that $K_H=\Q$). In this case we can consider the Weil restriction $\Res ^{K_H}_\Q \mathbf{H}$ and the associated period domain 
\begin{displaymath}
\tilde{D}:=  D_{\Res ^{K_H}_\Q \mathbf{H}(\R)} \subset D_{\widehat{G}}.
\end{displaymath}
We remark here that $S_{\Gamma_H}$ can also be written as $\psi^{-1}\pi(\tilde{D})^0$. Indeed $\tilde{D}$ is contained in $D_{\widehat{H}}$ and $\dim \psi (S_{\Gamma_H})\cap \pi(\tilde{D})= \dim \psi (S_{\Gamma_H})\cap \widehat{\mathbf{H}}(\Z)\backslash D_{\widehat{H}}$, since, as in the proof of Proposition \ref{Zspecial}, both are locally isomorphic to $X_H$.
\end{rmk}

The following is the converse of Proposition \ref{Zspecial} and their combination gives Theorem \ref{mainthm}. An important input in the proof is Corollary \ref{cor123}.
\begin{prop}[$\Z$-special $\Rightarrow$ $\Gamma$-special]\label{mainthmproof}
Let $(G=\PU(1,n),X,\Gamma)$ be a real Shimura datum and let $S_\Gamma$ be the associated Shimura variety. Any $\Z$-special subvariety $W$ of $S_\Gamma$ is $\Gamma$-special.
\end{prop}

\begin{proof}
Let $W \subset S_\Gamma$ be a $\Z$-special subvariety, as usual assumed to be of positive dimension. Let $w \in \widetilde{W}$ be such that $\pi(w)$ is smooth and $\MT(w)=\MT(W)$. By Corollary \ref{cor123}, we can assume that $\MT(w)$ is the Weil restriction of some $K$-subgroup of $\mathbf{G}$. To keep the notation more compact, we write $\mathbf{M}$ for the associated $K$-subgroup of $\mathbf{G}$, $M$ for its real points, and $\widehat{\mathbf{M} }$ for its Weil restriction to $\Q$. Being $W$ a $\Z$-special subvariety with generic Mumford--Tate group $\widehat{\mathbf{M} }$ means that
\begin{equation}\label{eqW}
W=\psi^{-1}(\widehat{\mathbf{M} }(\Z)\backslash D_{\widehat{M}})^0.
\end{equation}
Here $\widehat{M}$ denotes the associated real group and $(-)^0$ some analytic component. Notice also that $D_{\widehat{M}}$ is the $\widehat{M}$-orbit of $\tilde{\psi}(w)$ .

Our goal is to show that $W$ is associated to some sub-Shimura datum of $(G=\PU(1,n),X,\Gamma)$. The natural candidate appears to be $(M,X_M, \Gamma_M)$, where 
\begin{displaymath}
X_M= \tilde{\psi}^{-1}(D_{\widehat{M}})=M.w \subset X
\end{displaymath} is an Hermitian symmetric subspace of $X$ (since $w: \DT \to G$ factorises trough $M$). The only problem is that, a priori, we do not know that $\Gamma_M= \Gamma \cap M$ has finite covolume in $M$. Equation \eqref{eqW} is enough to show that the first of the following equalities holds true
\begin{displaymath}
W=\pi(M . w)=\Gamma_M \backslash X_M.
\end{displaymath}
As $W$ is an algebraic subvariety of $S_\Gamma$, by the work of Cattani, Deligne and Kaplan (recalled before as Theorem \ref{cdk}), we can apply Lemma \ref{finallemma} below to conclude that $\Gamma_M$ is a lattice in $M$. Therefore we proved that the $\Z$-special subvariety $W$ is $\Gamma$-special and associated to $(M,X_M, \Gamma_M)$.
\end{proof}
To conclude the proof of Proposition \ref{mainthmproof}, we need the following.
\begin{lemma}\label{finallemma}
Let $(G=\PU(1,n),X, \Gamma)$ be a real Shimura datum. Let $X_H$ be a Hermitian symmetric subspace of $X$, associated to a subgroup $H\subset G$, such that $\pi(X_H)$ is an algebraic subvariety of $S_\Gamma$. Then $(H,X_H, \Gamma_H)$ is a real sub-Shimura datum of $(G,X, \Gamma)$. That is $\Gamma_H$ is a lattice in $H$.
\end{lemma}
\begin{proof}
 Let $\Omega \subset \C^m$ be a bounded domain such that $\Gamma_\Omega\backslash \Omega$ is a quasi-projective variety form some discrete subgroup $\Gamma_\Omega$ of the automorphisms of $\Omega$ acting freely on $\Omega$. Recall that Griffiths \cite[Proposition 8.12]{MR310284} proved that $\Gamma_\Omega\backslash \Omega$ has finite Kobayashi-Eisenman volume \cite{MR0259165} (see also discussion after \cite[Question 8.13]{MR310284}). Applying such a result to $\Omega=X_H=\mathbb{B}^m$ (for some $m<n$) and $\Gamma_\Omega=\Gamma_H$ we see that $\Gamma_H\backslash X_H$ has finite Kobayashi-Eisenman volume. Thanks to \cite[Proposition 2.4 (page 57)]{MR0259165}, the Kobayashi-Eisenman volume form is equal to the volume associated to K\"{a}hler form on $\Bb^m$. Therefore $\Gamma_H \backslash \mathbb{B}^m$ has finite volume with respect to the volume form associated to the standard K\"{a}hler form on $\Bb^m$. That is, $\Gamma_H$ is a lattice in $H$.
\end{proof}
After having established Theorem \ref{mainthm}, which says that $\Gamma$-special subvarieties are the same thing as the $\Z$-special ones, we give the following.
\begin{defi}\label{defispecial}
A smooth irreducible closed algebraic subvariety $W\subset S_\Gamma$ (of dimension $>0$) is \emph{special} if it is $\Gamma$-special or equivalently $\Z$-special or equivalently a totally geodesic subvariety. We define the \emph{special closure of} $W$ as the smallest special subvariety of $S_\Gamma$ containing $W$. We say that $W$ is \emph{Hodge generic} if it is not contained in any strict special subvariety of $S_\Gamma$.
\end{defi}
The reader interested only in the proof of Theorem \ref{thmbader} may skip the next two subsections and go directly to Section \ref{sectionfinit}. Notice however that in Proposition \ref{mainpropwithas} we will use Theorem \ref{zasthm} which is recalled below.

\subsection{Proof of the non-arithmetic Ax-Schanuel conjecture}\label{proofofas} 
In this section we prove Theorem \ref{nonarithasconjfinal}, which, thanks to Theorem \ref{mainthm} (proven in Section \ref{proofmainthm}), becomes the following. 
\begin{thm}[Non-arithmetic Ax-Schanuel Conjecture]\label{nonarithas7}
Let $(G,X,\Gamma)$ be an irreducible Shimura datum. Let $W \subset X \times  S_\Gamma$ be an algebraic subvariety and $\Pi \subset X \times S_\Gamma$ be the graph of $\pi : X \to S_\Gamma$. Let $U$ be an irreducible component of $W \cap \Pi$ such that
\begin{displaymath}
\codim U < \codim W + \codim \Pi,
\end{displaymath}
the codimension being in $X \times S_\Gamma$ or, equivalently,
\begin{displaymath}
\dim W < \dim U + \dim S_\Gamma.
\end{displaymath}
If the projection of $U$ to $S_\Gamma$ is not zero dimensional, then it is contained in a strict special subvariety of $S_\Gamma$.
\end{thm}

Before starting the proof of Theorem \ref{nonarithas7}, we recall the Ax-Schanuel for the $\Z$-VHS $\widehat{\mathbb{V}}$ constructed in Theorem \ref{strategy1}. The next theorem follows from the work of Bakker and Tsimerman \cite[Theorem 1.1]{MR3958791}, Theorem \ref{mainthm} and the fact that $G$ has rank one. We refer to it as the $\Z$-Ax--Schanuel (or simply as $\Z$-AS).
\begin{thm}\label{zasthm}
Let $\widehat{W} \subset {D_{\widehat{G}}} \times  S_\Gamma$ be an algebraic subvariety. Let $\widehat{U}$ be an irreducible component of $\left(\widehat{W} \cap D_{\widehat{G}} \times_{ \widehat{\mathbf{G}}(\Z)\backslash D}  S_\Gamma \right)\subset   D_{\widehat{G}} \times S_\Gamma $ such that
\begin{displaymath}
\codim \widehat{U} < \codim \widehat{W} + \codim D_{\widehat{G}} \times_{ \widehat{\mathbf{G}}(\Z)\backslash D}  S_\Gamma,
\end{displaymath}
the codimension being in $ D_{\widehat{G}} \times S_\Gamma $. Then the projection of $\widehat{U}$ to $S_\Gamma$ is contained in a strict special subvariety of $S_\Gamma$, unless it is zero dimensional.
\end{thm}

\begin{proof}[Proof of Theorem \ref{nonarithas7}]
If $\Gamma$ is arithmetic, the result is a special case of \cite[Theorem 1.1]{as}, so we may and do assume that $\Gamma$ is non-arithmetic. Thanks to Theorem \ref{strategy1} and the discussion after Corollary \ref{gammadense}, we have the following commutative diagram
\begin{center}
\begin{tikzpicture}[scale=2]
\node (A) at (-1,1) {$X$};
\node (B) at (1,1) {$D=D_{\widehat{G}}$};
\node (C) at (-1,0) {${S_\Gamma}^{\an}$};
\node (D) at (1,0) {$\widehat{\mathbf{G}}(\Z)\backslash D$};
\path[right hook->,font=\scriptsize]
(A) edge node[above]{$\tilde{\psi}$} (B);
\path[->,font=\scriptsize,>=angle 90]
(C) edge node[above]{$\psi$} (D)
(A) edge node[right]{$\pi_\Gamma$} (C)
(B) edge node[right]{$\pi_\Z$} (D);
\end{tikzpicture}.
\end{center}
In the proof we actually use that the above diagram is cartesian. Finally we can write $D\cong X \times D_{>1}$, for some homogeneous space $D_{>1}$ under $\prod_{i>1}^rG_{\sigma_i}$, where $r=[K:\Q]$. Let $W \subset   X\times S_\Gamma$ be an algebraic subvariety, and consider the algebraic subvariety of $D\times S_\Gamma$
\begin{displaymath}
\widehat{W}:= W \times D_{>1}
\end{displaymath}
(by applying the natural isomorphism $X\times S_\Gamma \cong S_\Gamma \times X$). We claim that $W \cap \Pi$ can be identified with $\widehat{W} \cap (D \times _{\widehat{\mathbf{G}}(\Z)\backslash D} S_\Gamma)\subset D \times S_\Gamma$. To see this, notice that
\begin{displaymath}
W \cap \Pi=\{w=(w_0,w_1)\in W\subset X \times S_\Gamma  : \pi_\Gamma(w_1)=w_0\},
\end{displaymath}
and
\begin{displaymath}
\widehat{W} \cap (D \times _{\widehat{\mathbf{G}}(\Z)\backslash D} S_\Gamma)\cong \widehat{W} \cap \{ (\pi_\Gamma(x), \tilde{\psi}(x)): x \in X\}.
\end{displaymath}
The natural map identifying the two complex analytic varieties is given by
\begin{equation}\label{identification}
 W \cap \Pi \ni w \mapsto (w_0, \tilde{\psi}(w_1))\in \widehat{W} \cap (D_{\widehat{G}} \times_{ \widehat{\mathbf{G}}(\Z)\backslash D}  S_\Gamma).
\end{equation}

Let $U$ be an irreducible component of $W \cap \Pi$ such that
\begin{equation}\label{atypeq}
\codim_{X\times S_\Gamma} U < \codim_{X\times S_\Gamma} W + \codim_{X\times S_\Gamma} \Pi.
\end{equation}
Thanks to the identification appearing in \eqref{identification}, we see that $U$ corresponds to $\widehat{U}$, an irreducible component of $\widehat{W} \cap (D \times_{ \widehat{\mathbf{G}}(\Z)\backslash D}  S_\Gamma)$. From \eqref{atypeq}, we obtain that
\begin{displaymath}
\codim \widehat{U} < \codim \widehat{W} + \codim D \times_{ \widehat{\mathbf{G}}(\Z)\backslash D}  S_\Gamma,
\end{displaymath}
the codimension being in $ D \times S_\Gamma\cong D_{>1} \times X \times S_\Gamma$. Eventually we may apply Theorem \ref{zasthm} to conclude.
\end{proof}

\subsection{Some corollaries}\label{sectioncor}
Let $(G=\PU(1,n),X,\Gamma)$ be a real Shimura datum and
\begin{displaymath}
\pi = \pi_\Gamma : X \to S_\Gamma= \Gamma \backslash X
\end{displaymath}
 be the projection map. Let $Y\subset X$ be a positive dimensional algebraic subvariety. Applying the non-arithmetic Ax--Schanuel, Theorem \ref{nonarithas7}, to
\begin{displaymath}
W:=Y \times \overline{\pi(Y)}^{\Zar}
\end{displaymath} 
we obtain the following, which recovers Mok's result \cite[Main Theorem]{mokalw}.
\begin{cor}[Non-arithmetic Ax-Lindemann-Weierstrass]\label{gammaalw}
Let $(G=\PU(1,n),X,\Gamma)$ be a real Shimura datum and $Y\subset X$ be an algebraic subvariety. Then any irreducible component of the Zariski closure of $\pi(Y)$ in $S_\Gamma$ is special.
\end{cor}
Another direct consequence of the Ax-Lindemann-Weierstrass theorem is the following characterisation of special subvarieties of $S_\Gamma$.
\begin{cor}\label{maincor}
Let $W\subset S_\Gamma$ be an irreducible algebraic subvariety (of dimension $>0$). The following are equivalent:
\begin{enumerate}
\item $W$ is totally geodesic;
\item $W$ is bi-algebraic, i.e. some (equivalently any) analytic component of the preimage of $W$ along $\pi : X \to S_\Gamma$ is algebraic; 
\item $W$ is a special subvariety (in the sense of Definition \ref{defispecial});
\item $W=\psi^{-1}(\pi_{\Z} (Y))^0$ for some algebraic subvariety $Y$ of $D_{\widehat{G}}$.
\end{enumerate}
\end{cor}

The fact that a non-arithmetic $\Gamma$ gives rise to a thin subgroup of $\widehat{\mathbf{G}}(\Z)$ follows already from Corollary \ref{gammadense}. As announced in the introduction, we can prove Corollary \ref{cor2} (at least up to the finiteness of maximal special subvarieties, Theorem \ref{thmbader}, which is proven in the next section).
\begin{cor}
Let $(G=\PU(1,n),X,\Gamma\subset \mathbf{G}(\Oo_K))$ be a real Shimura datum. Let $W$ be a Hodge generic subvariety of $S_\Gamma$ (in the sense of Definition \ref{defispecial}). If $\Gamma$ is non-arithmetic, the image of 
\begin{displaymath}
\pi_1(W)\to \Gamma
\end{displaymath}
gives rise to a thin subgroup of $\widehat{\mathbf{G}}(\Z)$.
\end{cor}
\begin{proof}
As in Remark \ref{infiniteindex}, since $\Gamma$ is non-arithmetic, $\Gamma$ has infinite index in $\mathbf{G}(\Oo_K)$ and so the same holds for its image in $\widehat{\mathbf{G}}(\Z)$. Because of the Hodge genericity assumption, the André--Deligne monodromy theorem (Theorem \ref{deligneandre}) implies that the image of $\pi_1(W)\to \Gamma \to \widehat{\mathbf{G}}(\Z)$ is Zariski dense in $\widehat{\mathbf{G}}$ and therefore a thin subgroup of $\widehat{\mathbf{G}}(\Z)$.
\end{proof}

Finally the following was announced in the introduction as Corollary \ref{cor1}. The proof is similar to the one of Proposition \ref{mainthmproof}.
\begin{cor}
Let $(F,X_F)$ be a real sub-Shimura couple of $(G=\PU(1,n),X)$ (in the sense of Section \ref{embeddings}). If $\Gamma_F$ is Zariski dense in $F$, then $\Gamma_{F}$ is a lattice in $F$.
\end{cor}
\begin{proof}
The fact that $(F,X_F)$ is a real sub-Shimura couple of $(G=\PU(1,n),X)$ implies that there exists some $x\in X$ such that the map
\begin{equation}\label{eq1234}
x: \DT \to G
\end{equation}
factorises trough $F$. Since $\Gamma_F \subset F $ is Zariski dense in $F$, it determines a the $K$-form of $F$, which we denote by $\mathbf{F}$. Let $\widehat{\mathbf{F}}$ be the Weil restriction from $K$ to $\Q$ of $\mathbf{F}$. Consider the sub-pair 
\begin{displaymath}
(\widehat{\mathbf{F}}, D_{\widehat{F}}=\widehat{\mathbf{F}}(\R). \tilde{\psi}(x)) \subset (\widehat{\mathbf{G}}, D=D_{\widehat{G}}).
\end{displaymath}
 Because \eqref{eq1234} factorises trough $F$, $(\widehat{\mathbf{F}}, D_{\widehat{F}})$ is a sub-Mumford-Tate domain of $(\widehat{\mathbf{G}}, D=D_{\widehat{G}})$.

As usual we argue using the fundamental cartesian diagram (coming from Theorem \ref{strategy1} and Corollary \ref{gammadense})
\begin{center}
\begin{tikzpicture}[scale=2]
\node (A) at (-1,1) {$X$};
\node (B) at (1,1) {$D=D_{\widehat{G}}$};
\node (C) at (-1,0) {${S_\Gamma}^{\an}$};
\node (D) at (1,0) {$\widehat{\mathbf{G}}(\Z)\backslash D$};
\path[right hook->,font=\scriptsize]
(A) edge node[above]{$\tilde{\psi}$} (B);
\path[->,font=\scriptsize,>=angle 90]
(C) edge node[above]{$\psi$} (D)
(A) edge node[right]{$\pi_\Gamma$} (C)
(B) edge node[right]{$\pi_\Z$} (D);
\end{tikzpicture}.
\end{center}
We observe that
\begin{displaymath}
\tilde{\psi}^{-1}(D_{\widehat{F}})=X_F,
\end{displaymath}
and so
\begin{displaymath}
\Gamma_F \backslash X_F = \pi_{\Gamma}(X_F)=\psi^{-1}(\widehat{\mathbf{F}}(\Z)\backslash D_{\widehat{F}} )^0,
\end{displaymath}
for the component $\psi^{-1}(\widehat{\mathbf{F}}(\Z)\backslash D_{\widehat{F}} )^0$ of $\psi^{-1}(\widehat{\mathbf{F}}(\Z)\backslash D_{\widehat{F}} )$ containing $\Gamma_F \backslash X_F$. Theorem \ref{cdk} implies that the right hand side is algebraic. Eventually Lemma \ref{finallemma} applies and shows that $\Gamma_F$ is a lattice in $F$, as desired.
\end{proof}

\section{Finiteness of special subvarieties}\label{sectionfinit}
Let $n$ be an integer $>1$. Let $(G=\PU(1,n),X,\Gamma )$ be a real Shimura datum, as in Definition \ref{realshimuradef}, and $S_\Gamma=\Gamma \backslash X$ the associated Shimura variety (not necessarily of arithmetic type). Recall that, thanks to Theorem \ref{strategy1}, the lattice $\Gamma \subset G$ determines a totally real number field $K$, and a $K$-form $\mathbf{G}$ of the real group $G$ for which, up to $G$-conjugation, $\Gamma $ lies in $\mathbf{G}(\Oo_K)$.

Here we prove Theorem \ref{thmbader}, which, thanks to Definition \ref{defispecial} and Theorem \ref{mainthm}, reads as follows. Moreover, in Section \ref{sectionmarguliscomm}, we explain how a similar strategy can be used to reprove the Margulis commensurability criterion for arithmeticity for lattices in $\PU(1,n)$, which was recalled as Theorem \ref{commcriterion}.

\begin{thm}\label{thm123}
If $S_\Gamma$ is non-arithmetic, then it contains only finitely many special subvarieties. 
\end{thm}
A special geodesic subvariety $S'\subset S_\Gamma$ is called \emph{maximal} if the only special subvariety of $S_\Gamma$ strictly containing $S'$ is $S_\Gamma$ itself. Recall that, from Definition \ref{defispecial}, (positive dimensional) special subvarieties of $S_\Gamma$ are the same as totally geodesic subvarieties, $\Gamma$-special and $\Z$-special subvarieties. The maximality condition can not be dropped in the statement, unless $n=2$, where it is always satisfied. For example $S_\Gamma$ could contain a special subvariety $S'\subsetneq S_\Gamma$ associated to a real Shimura datum $(H,X_H,\Gamma_H)$ where $\Gamma_H$ is arithmetic and $S'$ could contain countably many special subvarieties. An equivalent way of phrasing the above theorem is the following.

\begin{thm}\label{thm1234}
Let $(Y_i)_{i\in \N}$ be a family of special subvarieties of $S_\Gamma$. Then every irreducible component of the Zariski closure of the union
\begin{displaymath}
\bigcup _{i \in \N} Y_i
\end{displaymath}
is either one of the $Y_i$s, or it is a special and arithmetic subvariety of $S_\Gamma$.
\end{thm}
To establish the equivalence of the two statements we only need to show that Theorem \ref{thm123} implies Theorem \ref{thm1234}. Let $S'$ be an irreducible component of $\overline{\bigcup _{i \in \N} Y_i}^{\text{Zar}}$. If $S'$ is not one of the $Y_i$, we can consider the smallest special subvariety $S''$ of $S_\Gamma$ containing $S'$. By construction, $S''$ contains infinitely many maximal totally geodesic subvarieties and therefore, by Theorem \ref{thm123}, it is arithmetic. To see that $S'=S''$, just notice that $S'$ is a subvariety of an arithmetic Shimura variety containing a Zariski dense set of sub-Shimura varieties. It follows that $S'$ is special \cite[Theorem 4.1]{MR3177266} and, by minimality of $S''$, that $S'=S''$.

We summarise here the strategy of the proof of Theorem \ref{thm123}, explaining what we are left to do.
\begin{itemize}
\item In Section \ref{zhodgesection}, we constructed a polarised integral variation of Hodge structure ($\Z$-VHS) $\widehat{\mathbb{V}}$ on $S_\Gamma$, corresponding to a period map
\begin{displaymath}
\psi : {S_{\Gamma}}^{\an}\to \widehat{\mathbf{G}}(\Z)\backslash D_{\widehat{G}},
\end{displaymath}
where $\widehat{\mathbf{G}}$ denotes the Weil restriction from $K$ to $\Q$ of $\mathbf{G}$, and $\widehat{G}=\widehat{\mathbf{G}} \times_\Q \R$. See Section \ref{section46} for the Hodge theoretic preliminaries needed, and Corollary \ref{gammadense} for the proof of the fact that $\widehat{\mathbf{G}}$ is the generic Mumford--Tate group of $\widehat{\mathbb{V}}$;
\item In Section \ref{comparisonsection}, we proved that the totally geodesic subvarieties of $S_\Gamma$ are $\Z$-special, that is they can be described as intersections between $S_{\Gamma}$ and some explicit sub-period domains of $\widehat{\mathbf{G}}(\Z)\backslash D_{\widehat{G}}$ (see Section \ref{proofmainthm});
\item Here we show that such $\Z$-special subvarieties are actually \emph{unlikely intersection}, and prove, using tools from functional transcendence for $\Z$-VHS, that the maximal ones appear in a finite number, unless $\psi$ is an isomorphism (which happens if and only if $\Gamma$ is arithmetic). See also the sketch of the proof presented in Section \ref{sketch}.
\end{itemize}

The last item, which is inspired by the arguments of \cite{MR3177266, MR3552014, MR3867286}, is proven in Section \ref{sectionwithproof}. We will show that maximal $\Z$-special subvarieties are parametrised by a countable and definable set in some o-minimal structure (as introduced in Section \ref{ominintro}). This is enough to conclude that they form a finite set, by the very definition of o-minimal structure. A key input to show the countability is the Ax-Schanuel conjecture for the pair $(S_\Gamma, \widehat{\mathbb{V}})$, as recently proven by Bakker and Tsimerman \cite{MR3958791} (recalled before as Theorem \ref{zasthm}). From now on we refer simply to special subvarieties of $S_\Gamma$, bearing in mind that we want to conclude something about the $\Gamma$-special subvarieties, by arguing on the $\Z$-special ones. This is where we employ Theorem \ref{mainthm}, which was proven in Section \ref{proofmainthm}.

\subsection{Proof of Theorem \ref{thmbader}}\label{sectionwithproof}

If $S_\Gamma$ contains no special subvariety, there is nothing to prove. Assume there is a special subvariety $W$ of $S_\Gamma$ of maximal dimension among special subvarieties. We will use several times Theorem \ref{mainthm} to describe $W$. First we see that $W$ is associated to $(H,X_H, \Gamma_H \subset \mathbf{H}(\Oo_K))$, a real sub-Shimura datum of $(G,X,\Gamma)$, where $\mathbf{H}$ is a strict semisimple $K$-subgroup of $\mathbf{G}$, since $W$ is in particular $\Gamma$-special. Denote by $\widehat{\mathbf{H}}$ the $\Q$-subgroup of $\widehat{\mathbf{G}}$ obtained as Weil restriction, from $K$ to $\Q$, of $\mathbf{H}$, and by $\widehat{H}$ the associated real group\footnote{Here we could consider the Weil restriction from the trace field $K_H$ of $\Gamma_H$, rather than from $K$, of $\mathbf{H}$, but, as explained in Remark \ref{twofields}, both groups describe the same $\Gamma$-special subvariety. The argument of Section \ref{mainpropsec} will moreover show that the smaller $K_H$ is, the more atypical $W$ is.}. 
\subsubsection{Preliminaries}
Before moving on with the argument, we need the following.
\begin{lemma}
Let $(G=\PU(1,n),X,\Gamma)$ be a real Shimura datum. There exists a compact subset $C$ of $S_\Gamma$ such that every special subvariety $S' \subset S_\Gamma$ has non-empty intersection with $C$.
\end{lemma}
\begin{proof}
If $\Gamma$ is arithmetic this appears as \cite[Lemma 4.5]{MR2180407}, for the general case see \cite[Proposition 3.2]{MR3936536}. Here is important that $G$ has rank one, and therefore that there is no difference between special (or totally geodesic) subvarieties and what is called in \cite{MR2180407} \emph{strongly special} subvarieties (see Section 4.1 in \emph{loc. cit.}).
\end{proof}
Let $\F\subset X$ be a fundamental set for the action of $\Gamma$ as in Section \ref{siegelsetssec} and set
\begin{displaymath}
\mathcal{C}:= \pi^{-1}(C) \cap \F,
\end{displaymath}
where $\pi : X \to S_\Gamma$ is the canonical projection map.
\begin{lemma}\label{623}
The set
\begin{equation}\label{eqdefsigm}
\Pi_0(\mathbf{H}):=\{(x,\hat{g}) \in \mathcal{C} \times \widehat{G}: \im (\tilde{\psi}(x): \DT \to \widehat{G})\subset \hat{g} \widehat{H} \hat{g}^{-1} \}
\end{equation}
is definable in an o-minimal structure.
\end{lemma}
To unpack the definition of $\Pi_0(\mathbf{H})$, recall that $\tilde{\psi}: X \to D$ denotes the lift of the period map 
\begin{displaymath}
\psi : {S_{\Gamma}}^{\an}\to \widehat{\mathbf{G}}(\Z)\backslash D_{\widehat{G}}.
\end{displaymath}
By the definition of period domains and Mumford--Tate domains, recalled in Section \ref{recapperioddomains}, every element in $D$ corresponds to a Hodge structure (whose associated Mumford--Tate group is a reductive $\Q$-subgroup of $\widehat{\mathbf{G}}$). In particular, given $x\in X$, $\tilde{\psi}(x)$ corresponds to a map of real algebraic groups
\begin{displaymath}
\tilde{\psi}(x): \DT \to \widehat{G},
\end{displaymath}
where $\DT$, as in Section \ref{31}, is the Deligne torus $\Res^{\C}_{\R} \Gm$. From now on, by \emph{definable}, we always mean definable in the o-minimal structure $\R_{\an, \exp}$ (which was introduced in Section \ref{ominintro}).
\begin{proof}[Proof of Lemma \ref{623}]
Thanks to the description of Siegel sets of Section \ref{fundamentalset}, we can assume that $\mathcal{C}$ is the closure of an open semialgebraic subset of $\F$. Note that the restriction of $\tilde{\psi}: X \to D_{\widehat{G}}$ to $\mathcal{C}$ is holomorphic and therefore definable in $\R_{\operatorname{an}}$. Eventually the condition on the image of $\tilde{\psi}(x)$ is a definable condition.
\end{proof}
\begin{rmk}
The reason why we work with the compact $\mathcal{C}$, rather than with the fundamental set $\F$, is to avoid discussing the definability of (opportune restrictions of) $\tilde{\psi}$. However it should be possible to apply the arguments of \cite[Theorem 1.5]{bakker2018tame} to prove that $\tilde{\psi}_{| \F}$ is definable. It would be a consequence of the fact that $\tilde{\psi}$ restricted to ${D^{(N)}}$, where $D^{(N)}\subset X$ was defined in \ref{dn}, is definable, at least for $N$ large enough, and that the complement of a compact subset of $\F$ is covered by a finite union of $D^{(N)}$s.
\end{rmk}
The condition on the image of $\tilde{\psi} (x)$ appearing in \eqref{eqdefsigm} ensures that we are only considering conjugated of $\widehat{H}$ by elements $\hat{g}\in \widehat{G}$ such that 
\begin{displaymath}
\hat{g} \widehat{H} \hat{g}^{-1}. \tilde{\psi}(x)
\end{displaymath}
has a complex structure. In particular $\hat{g} \widehat{H} \hat{g}^{-1}(\C). \tilde{\psi}(x)$ is an algebraic subvariety of the flag variety of $D_{\widehat{G}}$ whose intersection with $D_{\widehat{G}}$ is $\hat{g} \widehat{H} \hat{g}^{-1}. \tilde{\psi}(x)$. That is $\hat{g} \widehat{H} \hat{g}^{-1}. \tilde{\psi}(x)$ is an algebraic subvariety of $D_{\widehat{G}}$ in the sense of \cite[Appendix B]{MR3502100}. Associated to $(x,\hat{g}) \in \Pi_0(\mathbf{H})$ there is
\begin{displaymath}
S_{x,\hat{g}}:= \psi^{-1}\pi (\hat{g} \widehat{H} \hat{g}^{-1}. \tilde{\psi}(x))^0 \subset {S_\Gamma}^{\an},
\end{displaymath}
that could happen to be a special subvariety of $S_\Gamma$, or not. Indeed $S_{x,\hat{g}}$ is a special subvariety of $S_\Gamma$ only when the real group $\hat{g} \widehat{H} \hat{g}^{-1}$ has a structure of a $\Q$-subgroup of $\widehat{\mathbf{G}}$.

In what follows, by dimension we always mean the complex dimension. For $(x,\hat{g}) \in \Pi_0(\mathbf{H}) $ consider the function computing the local dimension at $\tilde{\psi}(x)$:
\begin{displaymath}
d_X(x,\hat{g}):= \dim_{\tilde{\psi}(x)} \left(\hat{g} \widehat{H} \hat{g}^{-1}. \tilde{\psi}(x)  \cap \tilde{\psi}(X)\right).
\end{displaymath}
Finally, given $0< j < n= \dim S_\Gamma$, we consider the sets
\begin{displaymath}
\Pi^j_1(\mathbf{H}):=\{(x,\hat{g}) \in \Pi_0(\mathbf{H}) : d_X(x,\hat{g}) \geq j\},
\end{displaymath}
\begin{equation}\label{sigmadef}
\Sigma^j=\Sigma(\mathbf{H})^j := \{\hat{g} \widehat{H}\hat{g}^{-1}: (x,\widehat{g})	\in \Pi^j_1(\mathbf{H}) \text{  for some  } x \in \mathcal{C} \}.
\end{equation}
\subsubsection{Main proposition and the \texorpdfstring{$\Z$}{z}-Ax--Schanuel}\label{mainpropsec} Let $W$ be a special subvariety of $S_\Gamma$ of maximal dimension among special subvarieties, say associated to $(H,X_H, \Gamma_H \subset \mathbf{H}(\Oo_K))$.
\begin{prop}\label{mainpropwithas}
Let $m= \dim W$ be the maximum of the dimensions of the strict special subvarieties of $S_\Gamma$. The set $\Sigma^m=\Sigma(\mathbf{H})^m$ is finite.
\end{prop}

\begin{proof}[Proof of Proposition \ref{mainpropwithas}]
It is enough to prove that $\Sigma^m$ is definable and countable. First observe that, since the local dimension of a definable set at a point is a definable function, $\Pi^m_1(\mathbf{H})$ is a definable set. Consider the map
\begin{displaymath}
\Pi^m_1(\mathbf{H}) \to \widehat{G}/N_{\widehat{G}}(\widehat{H}), \ \ \ (x,\widehat{g}) \mapsto \widehat{g} \cdot N_{\widehat{G}}(\widehat{H}),
\end{displaymath}
where $N_{\widehat{G}}(\widehat{H})$ denotes the normaliser of $\widehat{H}$ in $\widehat{G}$. Such a map is definable, and therefore its image, which is in bijection with $\Sigma^m$, is definable.

To prove that $\Sigma^m$ is countable it is enough to show that each $\hat{g} \widehat{H}\hat{g}^{-1}$ is defined over the rational numbers. We show that such a $(x,\hat{g})\in \Sigma^m$ gives rise to a special subvariety of $S_\Gamma$, that is 
\begin{displaymath}
S_{x,\hat{g}}= \hat{g} \widehat{H}\hat{g}^{-1}. \tilde{\psi}(x)\cap \tilde{\psi}(X)
\end{displaymath}
is a special subvariety of $S_\Gamma$. For simplicity simply write $D= D_{\widehat{G}}$. Consider the algebraic subset of $S_\Gamma \times {D}^\vee$, where ${D}^\vee$ denotes the flag variety associated to $D$
\begin{displaymath}
\widehat{W}:= S_\Gamma \times \left( \hat{g} \widehat{H} \hat{g}^{-1}. \tilde{\psi}(x)\right).
\end{displaymath}
Let $\widehat{U}$ be a component at $\tilde{\psi}(x)$ of the intersection
\begin{displaymath}
\widehat{W} \cap S_\Gamma \times_{ \widehat{\mathbf{G}}(\Z)\backslash D} D,
\end{displaymath}
such that the projection of $\widehat{U}$ to $S_\Gamma$ contains $S_{x,\hat{g}}$. We claim that, if $\Gamma$ is non-arithmetic, $\widehat{U}$ is an unlikely/atypical intersection. That is
\begin{equation}\label{equ}
\codim_{S_\Gamma \times D}\widehat{U} <  \codim_{S_\Gamma \times D}\widehat{W}  + \codim_{S_\Gamma \times D}  \left(  S_\Gamma \times_{ \widehat{\mathbf{G}}(\Z)\backslash D} D \right).
\end{equation}
Assuming the claim, the $\Z$-Ax-Schanuel, which was introduced before as Theorem \ref{zasthm}, asserts that the projection of $\widehat{U}$ to $S_\Gamma$ is contained in a strict special subvariety $S'$ of $S_\Gamma$. We have
\begin{displaymath}
S_{x,\hat{g}}\subset S' \subsetneq S_\Gamma.
\end{displaymath}
By definition of $\Sigma^m$, $\dim S_{x,\hat{g}} \geq m$ and, since $m$ is the maximum of the dimensions of the strict special subvarieties of $S_\Gamma$, $\dim S' \leq m$. It follows that $S_{x,\hat{g}}= S' $. That is $S_{x,\hat{g}}$ is a special subvariety of $S_\Gamma$ and so $\hat{g} \widehat{H}\hat{g}^{-1}$ is defined over $\Q$. This concludes the proof, assuming equation \eqref{equ}.

To prove the validity of equation \eqref{equ}, for each embedding $\sigma_1, \sigma_2,\dots, \sigma_r : K \to \R$ (ordered in such a way that $\sigma_1$ is just the identity on $K$) we set 
\begin{displaymath}
d_i := \dim D_{G_{\sigma_i}},
\end{displaymath}
\begin{displaymath}
d_{H_i} := \dim D_{H_{\sigma_i}},
\end{displaymath}
for any $i=1, \dots, r$. Observe that each $d_{H_i}$ is smaller or equal than $d_i$. We have
\begin{displaymath}
\dim \widehat{W} = d_1 + \sum_{i=1}^r d_{H_i},
\end{displaymath}
and, thanks to the definition of the set $\Pi^m_1(\mathbf{H})$, we also have
\begin{displaymath}
\dim \widehat{U} = d_X(x,\hat{g}) \geq d_{H_1}.
\end{displaymath}
Putting everything together, we obtain that
\begin{equation}
\codim_{S_\Gamma \times D}\widehat{U} =  \codim_{S_\Gamma \times D}\widehat{W}  + \codim_{S_\Gamma \times D}  \left(  S_\Gamma \times_{ \widehat{\mathbf{G}}(\Z)\backslash D} D \right)
\end{equation}
if and only if
\begin{displaymath}
\codim \widehat{U} = \dim D - \dim D_{\widehat{H}} +\dim D \leq d_{1}+\dim D - d_{H_1}.
\end{displaymath}
Or simply if and only if, for all $i\neq 1$, $d_{i}=d_{H_i}=0$. This happens if and only if each $H_{\sigma_i}$ and $G_{\sigma_i}$ are compact groups. By the Mostow–Vinberg arithmeticity criterion, Theorem \ref{arithcrit}, it is equivalent to $\Gamma$ being an arithmetic lattice in $G$. This concludes the proof of Proposition \ref{mainpropwithas}.
\end{proof}

\subsubsection{End of the proof and induction}
We have all the tools to finally prove Theorem \ref{thmbader}, arguing by induction as in the proof of \cite[Theorem 4.1]{MR3177266}.
\begin{proof}[Proof of Theorem \ref{thmbader}]
Let $m$ be the maximum of the dimensions of the strict special subvarieties of $S_\Gamma$. For any $j\leq m$ consider the set
\begin{center}
$\mathcal{E}_j := \{$maximal special subvarieties of $S_\Gamma$ of dimension $j\}$.
\end{center}Up to conjugation by elements of $\widehat{G}$, there are only a finite number of semisimple subgroups of $\widehat{G}$. Therefore, Proposition \ref{mainpropwithas} implies that there are only a finite number of special subvarieties of $S_\Gamma$ of maximal dimension $m$. That is, the set $\mathcal{E}_m$ is finite. We argue by induction (downward) on the $j$s for which $\mathcal{E}_j\neq \emptyset$. Let $\mathbf{F}$ be a semisimple subgroup of $\mathbf{G}$ for which there exists a special subvariety associated to $(F,X_F,\Gamma_F \subset \mathbf{F}(K))$ lying in $\mathcal{E}_j$. Consider the set
\begin{displaymath}
\Sigma_2(\mathbf{F})^j:= \{\hat{g} \widehat{F}\hat{g}^{-1}: (x,\widehat{g})\in	\Sigma(\mathbf{F})^j, S_{x,\widehat{g}}\notin \mathcal{E}_{j'}, \forall j'>j\},
\end{displaymath}
where $\Sigma(\mathbf{F})^j$ is defined as in \eqref{sigmadef}. The inductive assumption implies that the second condition is definable (being each $\mathcal{E}_{j'}$ finite). The proof of Proposition \ref{mainpropwithas} shows that $\Sigma_2(\mathbf{F})^j$ is countable. Since, up to conjugation by elements of $\widehat{G}$, there are only a finite number of possibilities for $\mathbf{F}$, we proved that $\mathcal{E}_j$ is finite. This concludes the proof of Theorem \ref{thmbader}.
\end{proof}

\begin{rmk}\label{klingrmk}
Recently Klingler \cite{klingler2017hodge} has proposed a generalisation of the Zilber--Pink Conjecture for arbitrary irreducible smooth quasi-projective complex varieties supporting a $\Z$-VHS. It is interesting to notice that his conjecture for the pair $(S_{\Gamma}, \widehat{\mathbb{V}})$ implies Theorem \ref{thmbader}. Given an irreducible subvariety $W\subset S_\Gamma$, Klingler introduced a notion of \emph{Hodge codimension}\footnote{Namely the rank of the horizontal part of the Lie algebra of the generic Mumford--Tate group of $W$ minus the dimension of $\psi(W)$.} of $W$ and a corresponding notion of \emph{Hodge optimality}. Conjecture 1.9 in \emph{loc. cit.} asserts that there are only finitely many Hodge optimal subvarieties of $S_\Gamma$. Now \eqref{equ}, appearing in Section \ref{mainpropsec}, shows indeed that, if $\Gamma$ is non-arithmetic, then the maximal special subvarieties (of positive dimension) are Hodge optimal, and therefore Klingler's conjecture predicts Theorem \ref{thmbader}.
\end{rmk}

\subsection{Commensurability criterion for arithmeticity of complex hyperbolic lattice}\label{sectionmarguliscomm}
We conclude this chapter interpreting the commensurability criterion for arithmeticity of Margulis as an unlikely intersection phenomenon. We offer a new proof Theorem \ref{commcriterion}, in the complex hyperbolic case, following the strategy used in the proof of Theorem \ref{thmbader}, therefore using no equidistribution/superrigidity techniques.

Let $(G=\PU(1,n),X,\Gamma\subset \mathbf{G}(\Oo_K))$ be a real Shimura datum and $S_\Gamma$ be the associated Shimura variety.
\begin{defi}
A $\Gamma$-\emph{special correspondence} for $S_\Gamma$ is a $\Gamma$-special subvariety of $S_\Gamma \times S_{\Gamma}$ whose projections $\pi_1, \pi_2 : S_\Gamma \times S_\Gamma\to S_\Gamma$ are finite and surjective.
\end{defi}

\begin{thm}\label{marguliscomm}
Let $(G=\PU(1,n),X,\Gamma)$ be a real Shimura datum. The lattice $\Gamma$ is arithmetic if and only if $S_\Gamma$ admits infinitely many $\Gamma$-special correspondences.
\end{thm}
\begin{rmk}
If $n=1$, the axiom RSD0 of real Shimura datum implies that $\Gamma$ is arithmetic and therefore that $S_\Gamma$ admits infinitely many $\Gamma$-special correspondences. The proof of Theorem \ref{marguliscomm} carries on assuming only that $\Gamma$ is (cohomologically\footnote{Over a one dimensional base, all rigid systems are known to be cohomologically rigid \cite[Theorem 1.1.2, Corollary 1.2.4]{MR1366651}.}) rigid. However we are not able to recover Theorem \ref{commcriterion} for \emph{all} lattices in $\PU(1,1)$.
\end{rmk}

The diagonal embedding $\Delta: S_\Gamma\to S_{\Gamma}\times S_{\Gamma}$ realises $S_\Gamma$ as a $\Gamma$-special correspondence associated to the real sub-Shimura datum
\begin{displaymath}
(\Delta(G),\Delta(X),\Delta(\Gamma))\subset (G\times G, X\times X, \Gamma \times \Gamma).
\end{displaymath}
Notice also that $\Delta( S_\Gamma)$ is a strict maximal $\Gamma$-special subvariety of $S_{\Gamma}\times S_{\Gamma}$. Given $g\in G$, we consider the maps
\begin{displaymath}
i_g: G \to G\times G, \ \ f \mapsto (f, gfg^{-1}),
\end{displaymath}
\begin{displaymath}
i_{g}: X \to X\times X, \ \ x \mapsto(x, g.x).
\end{displaymath}
Let $\pi: X \to S_\Gamma$ be the quotient map. The subvariety $\pi\times \pi(i_g(X))\subset S_{\Gamma}\times S_{\Gamma}$ is a $\Gamma$-special subvariety, in the sense of Definition \ref{defsubshimura}, if and only if $g$ lies in $\Comm(\Gamma)$. In this case $\pi\times \pi(i_g(X))$ is indeed the image of
\begin{displaymath}
i_g : S_{\Gamma \cap i_g (\Gamma )}\to S_{\Gamma}\times S_{\Gamma}.
\end{displaymath}
If $\Gamma$ is arithmetic then it has infinite index in its commensurator, and therefore, in the proof of Theorem \ref{marguliscomm} we may assume that $\Gamma$ is non-arithmetic.

Using the map
\begin{displaymath}
\psi \times \psi : {S_\Gamma}^{\an}\times{S_\Gamma}^{\an} \to \widehat{\mathbf{G}}(\Z) \backslash D_{\widehat{G}}\times  \widehat{\mathbf{G}}(\Z) \backslash D_{\widehat{G}},
\end{displaymath}
we have also a definition of $\Z$-special correspondences and a similar description.

\begin{defi}
A $\Z$-\emph{special correspondence} for $S_\Gamma$ is a $\Z$-special subvariety of $S_\Gamma \times S_{\Gamma}$ whose projections $\pi_i : S_\Gamma \times S_\Gamma\to S_\Gamma$ are finite and surjective.
\end{defi}

The arguments appearing in the proof of Theorem \ref{mainthm}, namely in Section \ref{proofmainthm}, show also the following.
\begin{prop}
The $\Gamma$-special correspondences are precisely the $\Z$-special ones.
\end{prop}

From now on we just speak of \emph{special} correspondences for $S_\Gamma$. We are ready to prove Theorem \ref{marguliscomm}.
\subsubsection{Proof of Theorem \ref{marguliscomm}}
As in Section \ref{sectionwithproof}, let $\mathcal{C}\subset \F \subset X$ be a definable compact set such that $X_H\cap \mathcal{C}\neq \emptyset$, for any real sub-Shimura datum $(H,X_H, \Gamma_H)\subset (G,X,\Gamma)$. Given $\hat{g} \in \widehat{G}$, we write
\begin{displaymath}
\hat{g}=(g_1,\dots, g_r)\in \prod_{i=1,\dots, r} G_{\sigma_i},
\end{displaymath}
where $r=[K:\Q]$ and $\sigma_i: K \to \R$ are ordered in such a way that $\sigma_1$ induces the identity on $K$. Consider the map
\begin{displaymath}
i_{\hat{g}}:\widehat{G}\to \widehat{G}\times \widehat{G}, \ \ \ \hat{f}\mapsto (\hat{f},\hat{g} \hat{f} \hat{g}^{-1}), 
\end{displaymath}
and write $\widehat{\Delta}_{\hat{g}}$ for its image. Consider the set
\begin{equation}
\Pi_0:=\{(x,\hat{g}) \in \mathcal{C} \times \widehat{G}: g_1.x\in \mathcal{C},  \im ((\tilde{\psi}\times \tilde{\psi})(x,g_1.x): \DT \to \widehat{G}\times \widehat{G})\subset \widehat{\Delta}_{\hat{g}} \}.
\end{equation}
By Lemma \ref{623}, $\Pi_0$ is a definable set. Associated to $(x,\hat{g}) \in \Pi_0$ there is
\begin{displaymath}
C_{x,\hat{g}}:= \psi^{-1}\times \psi^{-1}\left(\pi \times \pi(\widehat{\Delta}_{\hat{g}}. (\tilde{\psi}(x),\tilde{\psi}(g_1 . x))\right)^0 \subset {S_\Gamma}^{\an}\times {S_\Gamma}^{\an},
\end{displaymath}
that could happen to be a special correspondence for $S_\Gamma$, or not. Moreover any special correspondence for $S_\Gamma$ arises as a $C_{x,\hat{g}}$, for some $(x,\hat{g}) \in \Pi_0$. Define the local dimension
\begin{displaymath}
c(x,\hat{g}):=\dim_{(\tilde{\psi}(x),\tilde{\psi}(g_1 .x))} \left( \widehat{\Delta}_{\hat{g}}. (\tilde{\psi}(x),\tilde{\psi}(g_1.x))\cap \tilde{\psi}\times\tilde{\psi} (X \times X)\right).
\end{displaymath}

The set 
\begin{displaymath}
\Sigma := \{\widehat{\Delta}_{\hat{g}}: (x,\widehat{g})	\in \Pi_0 \text{  for some  } x \in \mathcal{C}, c(x,\hat{g}) =n\}
\end{displaymath}
is definable. To show that $\Sigma$ is countable, we argue as in the proof of Proposition \ref{mainpropwithas}. More precisely we use $\Z$-AS, Theorem \ref{zasthm}, to prove that $\widehat{\Delta}_{\hat{g}}$ comes from a $\Q$-subgroup of $\widehat{\mathbf{G}}\times\widehat{\mathbf{G}}$.

Assume that $\Gamma$ is non-arithmetic and consider the algebraic subset of $S_\Gamma \times S_{\Gamma} \times  D_{\widehat{G}}\times D_{\widehat{G}}$
\begin{displaymath}
\widehat{W}:= S_\Gamma\times S_\Gamma\times \left( \widehat{\Delta}_{\hat{g}}.(\tilde{\psi}(x),\tilde{\psi}(g_1.x) )\right).
\end{displaymath}
Let $\widehat{U}$ be a component at $(\tilde{\psi}(x),\tilde{\psi}(g_1. x)) $ of the intersection
\begin{displaymath}
\widehat{W} \cap \left( S_\Gamma \times S_\Gamma \times_{ \widehat{\mathbf{G}}(\Z)\backslash D \times \widehat{\mathbf{G}}(\Z)\backslash D} D_{\widehat{G}} \times D_{\widehat{G}}\right),
\end{displaymath}
such that the projection of $\widehat{U}$ to $S_\Gamma \times S_\Gamma$ contains $C_{x,\hat{g}}$. A direct computation shows that, if $\Gamma$ is non-arithmetic, such an intersection is unlikely and the $\Z$-Ax-Schanuel, namely Theorem \ref{zasthm}, implies that $C_{x,\hat{g}}$, for $(x,\hat{g})\in \Sigma$, is a special correspondence of $S_\Gamma$. It follows that $\widehat{\Delta}_{\hat{g}}$ is defined over the rational numbers. We proved that $\Sigma$ is definable and countable, therefore finite. That is, if $\Gamma$ is non-arithmetic, then $S_\Gamma$ contains only finitely many special correspondences.

\section{Special points and their Zariski closure}\label{pointsection}
Let $(G=\PU(1,n),X,\Gamma \subset \mathbf{G}(\Oo_K))$ be a real Shimura datum and $S_\Gamma=\Gamma \backslash X$ be the associated Shimura variety. We end the paper presenting several notions of \emph{special point} of $S_\Gamma$, inspired by the distribution of special points of arithmetic Shimura varieties. We propose four different definitions of special points and we formulate four André--Oort like conjectures relating special points to special subvarieties in the sense of Definition \ref{defispecial} (that were studied in details in the previous sections). For an introduction to the André--Oort conjecture for (arithmetic) Shimura varieties, we refer for example to \cite{MR3821177}.

\subsection{\texorpdfstring{$\Gamma$}{gamma}-special points}
We first introduce the notion of zero dimensional $\Gamma$-\emph{special subvariety} of $S_\Gamma$. Here we denote by $K$ the totally real field associated to $\Gamma$, and $\mathbf{G}$ the $K$-form of $G$, given by Theorem \ref{strategy1}.
\begin{defi}\label{gammasppoint}
A point $x \in X$ is \emph{pre}-$\Gamma$-\emph{special} if the smallest $K$-subgroup of $\mathbf{G}$ whose extension to the real numbers contains the image of 
\begin{displaymath}
x: \DT \to \mathbf{G}_{\R}
\end{displaymath}
is commutative. A point $s\in S_\Gamma$ is $\Gamma$-\emph{special} if it is the image along $\pi: X \to S_\Gamma$ of a pre-$\Gamma$-special point $x\in X$.
\end{defi}
The following can be proven as in the classical case of arithmetic Shimura varieties. From now on, we fix an algebraic closure of the field of rational numbers, denoted by $\Qbar$.
\begin{prop}
Every $\Gamma$-special subvariety contains an analytically dense set of $\Gamma$-special points. Pre-$\Gamma$-special point are defined over $\Qbar$, with respect to the $\Qbar$-structure given by looking at $X$ inside its associated flag variety $X^\vee$.
\end{prop}
\begin{proof}
Let $S_{\Gamma_H}$ be a $\Gamma$-special subvariety associated to $(H,X_H, \Gamma_H)$. First notice that if one $\Gamma$-special point exists in $S_{\Gamma_H}$, then there is a dense set of $\Gamma$-special points. Indeed each point in $\pi(\mathbf{H}(K).x) \subset S_{\Gamma_H}$ is again $\Gamma$-special and the set $\mathbf{H}(K).x$ is dense by the usual approximation property.

For the existence of a $\Gamma$-special point we argue as in \cite[Section 5]{delignetravaux}. Let $x\in X_H$ and $T\subset G$ be a maximal torus containing the image of
\begin{displaymath}
x: \DT \to H.
\end{displaymath}
Write $T$ as the centraliser of some regular element $\lambda$ of the Lie algebra of $T$. Choose $\lambda' \in \mathbf{H}(K)$ sufficiently close to $\lambda$. It is still a regular element and its centralizer $\mathbf{T}'$ in $\mathbf{G}$ is a maximal torus in $\mathbf{G}$. Since there are only finitely many conjugacy classes of maximal real tori in $G$, we may assume that $T'$ and $T$ are conjugated by some $h\in H$. The point $\pi(hx)$ is a $\Gamma$-special point in $\pi(X_H)$.

Finally if $x\in X$ is pre-$\Gamma$-special, we can factorise its associated Hodge cocharacter as 
\begin{displaymath}
h_x: \C^* \to (\mathbf{T}_x )_{\C} ,
\end{displaymath}
where $\mathbf{T}_x$ is a maximal $K$-torus. The above factorisation has to hold even over $\Qbar$. Therefore the point associated to $x$ in the flag variety is fixed by $\mathbf{T}_x (\Qbar)$. This is enough to conclude that $x$ is defined over $\Qbar$. For a complete discussion, we refer to \cite[Proposition 3.7]{MR2825237}.
\end{proof}

As in the case of arithmetic Shimura varieties, the $\Gamma$-Ax-Schanuel proven in Section \ref{proofofas}, or more precisely Corollary \ref{gammaalw}, could be helpful in proving the following.
\begin{conj}[$\Gamma$-André--Oort]
An irreducible subvariety $W\subset S_\Gamma$ is $\Gamma$-special if it contains a Zariski dense set of $\Gamma$-special points.
\end{conj}
\subsection{\texorpdfstring{$\Z$}{z}-special points}
Another natural possibility is to look at zero dimensional intersections between $\psi ({S_\Gamma}^{\an})$ and sub-domains of the Mumford--Tate domain $\widehat{\mathbf{G}}(\Z) \backslash D_{\widehat{G}}$ constructed in Theorem \ref{strategy1}, see also Remark \ref{klingrmk}. That is the zero dimensional components of the Hodge locus of $(S,\widehat{\mathbb{V}})$, where $\widehat{\mathbb{V}}$ is the $\Z$-VHS introduced in Theorem \ref{strategy1}.
\begin{defi}
A point $s\in S_\Gamma$ is $\Z$-\emph{special} if it is a zero dimensional $\Z$-special subvariety. 
\end{defi}

A first relation between $\Z$-special point and $\Gamma$-special is given by the following.
\begin{prop}\label{propspepoints}
Let $s=\pi(x)\in S_\Gamma$ be a $\Z$-special point, then $s$ is $\Gamma$-special (in the sense of Definition \ref{gammasppoint}).
\end{prop}
\begin{proof}
Let $\mathbf{R}$ be a $\Q$-subgroup of $\widehat{\mathbf{G}}$ inducing a Mumford--Tate sub-domain $D_{R}$ of $D_{\widehat{G}}$. If $\mathbf{R}(\Z)\backslash D_R$ intersects $\psi ({S_\Gamma}^{\an})$ in a finite number of points, then $R \cap G_{\sigma_1}$ is a compact subgroup of $\widehat{\mathbf{G}}(\R)$. Let $x\in \tilde{\psi}^{-1}(D_{R})\subset X$. It corresponds to a map
\begin{displaymath}
x: \DT \to G,
\end{displaymath} and, by construction $x(\DT)$ is contained in $R \cap G_{\sigma_1}$. Since $G$ is of adjoint type, $x (\DT)$ is contained in the centre of $R \cap G_{\sigma_1}$. However $R \cap G_{\sigma_1}$ has a natural structure of $K$-subgroup of $\mathbf{G}$, which we denote by $\mathbf{R}_{\sigma}$. It follows that $x(\DT)$ is contained in a commutative $K$-subgroup of $\mathbf{G}$. That is $\MT(x)$ is commutative and $x$ is pre-$\Gamma$-special.
\end{proof}

In rank one, the $\Gamma$-AO Conjecture, thanks to Proposition \ref{propspepoints}, predicts that the irreducible components of the Zariski closure of a set of $\Z$-special points are $\Z$-special subvarieties of $S_\Gamma$.

\begin{prop}
If $\Gamma$ is non-arithmetic, $\Z$-special points of $S_\Gamma$ are atypical intersections in $\widehat{\mathbf{G}} (\Z) \backslash D_{\widehat{G}}$ between $\psi ({S_\Gamma}^{\an})$ and Mumford--Tate sub-domains of $D_{\widehat{G}}$.
\end{prop}
\begin{proof}
Let $\mathbf{R}\subset \widehat{\mathbf{G}}$ be as in the proof of Proposition \ref{propspepoints}. The intersection between $\mathbf{R}(\Z)\backslash D_R$ and $\psi ({S_\Gamma}^{\an})$ is typical when 
\begin{displaymath}
\codim_{\widehat{\mathbf{G}} (\Z) \backslash D_{\widehat{G}}}  \{s\} = \codim _{\widehat{\mathbf{G}} (\Z) \backslash D_{\widehat{G}}}S_{\Gamma} +  \codim_{\widehat{\mathbf{G}} (\Z) \backslash D_{\widehat{G}}} \mathbf{R}(\Z)\backslash D_R,
\end{displaymath}
where $s$ is one of the finite points lying in $\mathbf{R}(\Z)\backslash D_R \cap \psi ({S_\Gamma}^{\an})$. Let $r=[K:\Q]$. Equivalently 
\begin{displaymath}
\sum_{i=1}^r d_{i}=\sum_{i=2}^r d_{i}+ \left( d_1 + \sum_{i=2}^r d_{i} - d_{R_i}\right),
\end{displaymath}
where $d_{i}$ and $d_{R_i}$ were defined in the proof of Proposition \ref{mainpropwithas}. That is $s$ is typical if and only if $d_{i} =d_{R_i}$ for all $i\geq 2$, which is impossible unless $\mathbf{R}=\widehat{\mathbf{G}}$.
\end{proof}
As usual we expect that a subvariety of $S_\Gamma$ having a dense set of atypical points is atypical. 

\begin{conj}\label{zconj}
Let $(G,X,\Gamma)$ be a real Shimura datum and $W\subset S_\Gamma$ an irreducible algebraic subvariety. $W$ contains a Zariski dense set of $\Z$-special points if and only if it is special and arithmetic.
\end{conj}

We can also observe that the intersection of $\Gamma$-special subvarieties gives rise to $\Z$-special, and therefore $\Gamma$-special, points.
\begin{prop}
Let $S_{\Gamma_1}$, $S_{\Gamma_2}$ be $\Gamma$-special subvarieties of positive dimension in $S_\Gamma$ intersecting in a finite number of points. If $s \in S_{\Gamma_1}\cap S_{\Gamma_2}$, then $s$ is $\Z$-special.
\end{prop}
\begin{proof}Since $S_{\Gamma_1}$ and $S_{\Gamma_2}$ are also $\Z$-special subvariety, every point in the intersection can be written as the the preimage along the period map of the intersection of two Mumford--Tate sub-domains. The point $s$ is therefore $\Z$-special.
\end{proof}

\subsection{Complex multiplication points}
Recall that, as in Theorem \ref{strategy1}, we denote by $\widehat{\mathbb{V}}$ for the $\Z$-VHS on $S_\Gamma$ induced by $\mathbb{V}$. Another interesting class of points is given by the following.
\begin{defi}\label{defcm}
A point $s\in S_\Gamma$ is called a \emph{CM-point} if the Mumford--Tate group of $\widehat{\mathbb{V}}$ at $s$ is commutative.
\end{defi}
Such points are both $\Z$ and $\Gamma$-special, but they are even more special. The following is a special case of \cite[
Conjecture 5.6]{klingler2017hodge} and it is indeed predicted by, the more difficult, Conjecture \ref{zconj}.
\begin{conj}
Let $(G,X,\Gamma)$ be a real Shimura datum and $W\subset S_\Gamma$ an irreducible algebraic subvariety. Then $W$ contains a Zariski dense set of CM-points if and only if it is special and arithmetic.
\end{conj}
If $\Gamma$ is an arithmetic lattice, this is the classical André-Oort conjecture. If $\Gamma$ is non-arithmetic, but $\widehat{\mathbf{G}}(\Z)\backslash D_{\widehat{G}}$ is a Shimura variety, then the above conjecture follows from André--Oort conjecture for Shimura varieties of abelian type, which is now a theorem \cite{MR3744855}.

\subsection{Bi-arithmetic points}
In the classical case of arithmetic Shimura varieties, another option is to look at \emph{bi-arithmetic points} and \emph{bi-arithmetic subvarieties}. That is points $x\in X(\Qbar)$ such that $\pi(x)\in S_\Gamma(\Qbar)$ and subvarieties $W\subset S_\Gamma$ which are defined over $\Qbar$ and such that some component of $\pi^{-1}(W)$ is algebraic and defined over some number field. See also \cite[Section 4.2]{MR3821177}. Let $(G=\PU(1,n), X, \Gamma)$ be a real Shimura datum\footnote{From axiom RSD0 a real Shimura datum $(G=\PU(1,1), X, \Gamma)$ corresponds to a Shimura curve (for the standard definition).}. We briefly show how to prove that $S_\Gamma$, and its positive dimensional special subvarieties $S'\subset S_\Gamma$ have a model over a number field. This follows again from the rigidity of lattices. 

\subsubsection{Models}
Let $S$ be a smooth complex quasi-projective variety. We say that $S$ admits a $\Qbar$-\emph{model} if there exists $Y/\Qbar$ such that $Y \times _{\Qbar} \C \cong S$, with respect to some embedding $\Qbar \hookrightarrow \C$. Two $\Qbar$-models $Y,Y'$ are certainly isomorphic, both over $\C $ and $\Qbar$. We say that $S$ has a \emph{unique} $\Qbar$-structure if $S$ admits a $\Qbar$-model and any two $\Qbar$-models of $S$ are uniquely isomorphic.

For example if $\Gamma$ is cocompact, \cite[Theorem 1]{MR0223368} and \cite[Theorem 1]{MR0111058} imply the following\footnote{Actually Shimura does not discuss the uniqueness of the $\Qbar$-structure.}.
\begin{thm}[Shimura, Calabi-Vesentini]
Let $(G=\PU(1,n),X,\Gamma)$ be a real Shimura datum such that $S_\Gamma$ is projective. Let $\Theta$ be the sheaf of germs of holomorphic sections of the tangent bundle of $S_\Gamma$. If the dimension of $S_\Gamma$ is greater or equal to two, then, for $i=0,1$
\begin{displaymath}
H^i(S_\Gamma, \Theta)=0.
\end{displaymath}
It follows that $S_\Gamma$ has a unique $\Qbar$-structure.
\end{thm} 
If $\Gamma$ is arithmetic, the above proof can be generalised to cover the case when $S_\Gamma$ is not compact, by using Mumford's theory \cite{MR471627}. See indeed \cite{MR791585, MR3725498}. With the tools described in Section \ref{compac} it should be possible to generalise such arguments to the case of non-arithmetic lattices. For length reasons, we prefer to give a different argument which uses only the fact that all lattices in $\PU(1,n)$, for $n>1$, have entries in some number field, as explained in Section \ref{localrig}. Even if such a point of view does not address the uniqueness of the $\Qbar$-structure.

\subsubsection{The action of the automorphisms of \texorpdfstring{$\C$}{C} on a Shimura variety}
\begin{thm} \label{actionofc}
Let $(G,X,\Gamma)$ be a real Shimura datum and $\sigma \in \Aut(\C)$. There exists a real Shimura datum $(\leftidx{^\sigma}{G},\leftidx{^\sigma}{X},\leftidx{^\sigma}{\Gamma})$ such that
\begin{displaymath}
\leftidx{^\sigma}{S}{_\Gamma} \cong S_{^\sigma \Gamma}.
\end{displaymath}
Moreover:
\begin{itemize}
\item $G$ has rank one if and only if $\leftidx{^\sigma}{G}$ has rank one;
\item $\Gamma$ is arithmetic if and only if $\leftidx{^\sigma}{\Gamma}$ is arithmetic.
\end{itemize} 
\end{thm}

\begin{proof}
In the arithmetic case, the result follows from the work of Borovoi and Kazdhan \cite{MR693357}. To prove the result, we may and do replace $\Gamma$ by a finite index subgroup. In the non-arithmetic case, thanks to the work of Margulis (Theorem \ref{margulisrank}), the first part of the above statement reduces to the case of ball quotients (of dimension strictly bigger than one). The result follows from fact that the Chern numbers of $S_\Gamma$ satisfy the equality from Yau's solution of the Calabi conjecture (in the cocompact case, and by Tsuji \cite{zbMATH04097530} in the general case), which characterise ball quotients by their Chern numbers. See for example \cite[Proof of the reduction, page 257]{MR1211885} and references therein. Indeed let $\overline{S}$ be a smooth complex projective algebraic of dimension $n>1$, and let $D$ be a smooth divisor on $\overline{S}$. Assume that
\begin{itemize}
\item[1.] $ K_{\overline{S}} + (1-\epsilon)D$ is ample for every sufficiently small positive rational number $\epsilon$;
\item[2.] $K_{\overline{S}} +D $ is numerically trivial on $D$ and it is ample modulo $D$ and semi-ample (see \cite[Definition 1.2]{zbMATH04097530}).
\end{itemize}
Tsuji \cite[Theorem 1]{zbMATH04097530} proved that the inequality holds 
\begin{equation}\label{eqqqqq123}
c_1^n (\Omega^1_{\overline{S}}(\log D)) \leq \frac{2n+2}{n}c_1^{n-2}(\Omega^1_{\overline{S}}(\log D))c_2(\Omega^1_{\overline{S}}(\log D))
\end{equation}
and the equality holds if and only if $\overline{S}-D$ is an unramified quotient of the unit ball $\mathbb{B}^n$ in $\mathbf{\C}^n$. That is, if and only if $\overline{S}-D$ is a ball quotient of the form $S_{\Gamma '}$ for some lattice $\Gamma '$ in $\PU(1,n)$.

Let $\overline{S}$ be the toroidal compactification of $S_\Gamma$, as described in Section \ref{compac}, and let $\sigma \in \Aut (\C)$. We have that $\leftidx{^\sigma}{\overline{S}}$ satisfies 1. and 2. of above, since such properties were satisfied by $\overline{S}$ and are preserved by the action of $\Aut(\C)$. Since the Chern numbers of $\overline{S}$ are also preserved under the conjugate of a Galois action, also $\leftidx{^\sigma}{\overline{S}}$ satisfies the equality in \eqref{eqqqqq123}. Therefore it is the toroidal compactification of another ball quotient $S_{^\sigma \Gamma}$, associated to $(\leftidx{^\sigma}{G}=G,\leftidx{^\sigma}{X}=X=\mathbb{B}^n,\leftidx{^\sigma}{\Gamma})$ for some other lattice $\leftidx{^\sigma}{\Gamma}$ in $G$. This proves the first part of Theorem \ref{actionofc}, and the first bullet of the second part (since $G$ has rank one if and only if $X$ is a complex ball).

We are left to check that $\Gamma$ is arithmetic if and only if $\leftidx{^\sigma}{\Gamma}$ is arithmetic. Notice that elements in the commensurator of $\Gamma$, denoted by $\Comm (\Gamma)$, are in bijection with special algebraic correspondences of $S_\Gamma \times S_\Gamma$, in the sense of Section \ref{sectionmarguliscomm}. Special correspondences are preserved by the action of $\Aut(\C)$, as observed in the proof of \cite[Proposition 1.2]{MR3424477}. Therefore $\sigma^{-1}$ induces an isomorphism
\begin{displaymath}
\Comm (\Gamma) / \Gamma \cong \Comm (\leftidx{^\sigma}{\Gamma}) /\leftidx{^\sigma}{\Gamma}.
\end{displaymath}
By the Margulis commensurability criterion (as recalled in Sections \ref{commcriterion} and \ref{sectionmarguliscomm}), we conclude that $\Gamma$ is arithmetic if and only if $ \leftidx{^\sigma}{\Gamma}$ is arithmetic. Theorem \ref{actionofc} is eventually proven.
\end{proof}

\begin{rmk}\label{propmod}
To complete our discussion of the action of $\Aut(\C)$, we should explain what happens to special subvarieties of positive dimension of $S_\Gamma$ (i.e. in the sense of Definition \ref{defispecial}). Indeed it is a well known fact that, in the theory of arithmetic Shimura varieties, $\Aut(\C)$ maps special subvarieties of $S_\Gamma$ to special subvarieties of $S_{^\sigma \Gamma}$. Thanks to the description of special subvarieties of $S_\Gamma$ as the Hodge locus of a $\Z$-VHS, as we established in Corollary \ref{maincor}, we can see that $\sigma \in \Aut(\C)$ maps $\Z$-special subvarieties of dimension $>0$ of $S_\Gamma$ to $\Z$-special subvarieties of $S_{^\sigma \Gamma}$ as a special case of the main result of \cite{2020arXiv201003359K}.
\end{rmk}

\begin{thm}
Let $(G,X,\Gamma)$ be a real Shimura datum. Then $S_\Gamma$ admits a model over a number field.
\end{thm}

\begin{proof}
If $\Gamma$ is arithmetic, the result follows for example from Faltings \cite{MR791585}. Therefore we may and do assume that $G=\PU(1,n)$, for some $n>1$. Let ${S_\Gamma}^{BB}$ be Baily-Borel compactification of $S_\Gamma$, as discussed in Section \ref{compac}. From \cite[Criterion 1 (page 3)]{MR2253591}, the following are known to be equivalent:
\begin{itemize}
\item[(a)] ${S_\Gamma}^{BB}$ can be defined over $\Qbar$;
\item[(b)] The set $\{\leftidx{^\sigma}{{S_\Gamma}}^{BB}: \sigma \in \Aut(\mathbb{C} /  \Qbar)\}$ contains only finitely many isomorphism classes of complex projective varieties;
\item[(c)] The set $\{\leftidx{^\sigma}{{S_\Gamma}}^{BB}: \sigma \in \Aut(\mathbb{C} /  \Qbar)\}$ contains only countably many isomorphism classes of complex projective varieties.
\end{itemize}
Since the action of $\Aut(\C)$ on the Baily-Borel compactification of $S_\Gamma$ preserves the smooth locus, it is enough to check that $\{\leftidx{^\sigma}{S_\Gamma}: \sigma \in \Aut(\mathbb{C} /  \Qbar)\}$ contains countably many isomorphism classes. By Theorem \ref{actionofc}, it is enough to show that there are most countably many lattices in $G=\PU(1,n)$ (for some $n>1$ fixed), up to conjugation by elements in $G$. But this follows from the fact that complex hyperbolic lattices enjoy local rigidity, as discussed in Section \ref{localrig}.
\end{proof}

Finally, Remark \ref{propmod} shows that the positive dimensional special subvarieties of $S_\Gamma$ are also defined over some number field. 

\subsubsection{Zariski closure of bi-arithmetic points}
The last conjecture we propose is as follows.
\begin{conj}
An irreducible subvariety $W\subset S_\Gamma$ contains a dense set of bi-arithmetic points if and only if it is bi-arithmetic.
\end{conj}
This is known to be equivalent to the André--Oort conjecture only for arithmetic Shimura varieties of abelian type (thanks to W\"{u}stholz' analytic subgroup theorem). See \cite[Example 4.17]{MR3821177} and references therein. Finally it would be interesting to show that CM-points in $S_\Gamma$, in the sense of Definition \ref{defcm}, are bi-arithmetic.

\appendix
\section{Erratum on the published version}
The proof provided for \cite[Cor. 1.3.3]{BUann} (which appears as the proof of Cor. 5.5.4, on page 199, and Corollary \ref{cor1} on the ArXiv version) is not complete, and in fact we do not know whether the claimed statement holds true or not. The corollary in question is not applied anywhere else in the text and therefore the other results are left untouched. Below we explain the incriminated step of the proof.

In the second half of the proof on page 199, it is claimed that 
\begin{displaymath}
\tilde{\psi}^{-1}(D_{\widehat{F}})=X_F,
\end{displaymath}
and no justification is given. This is true, and used multiple times before, when $\Gamma_F$ is lattice (which is the desired conclusion). If so, $\Gamma_F \backslash X_F$ is a smooth quasi-projective variety and the claim essentially follows from the \emph{Fixed Part Theorem} (which requires an algebraic base). Assuming only that $\Gamma_F$ is Zariski dense in $F$, we believe new ideas are required to justify such claim, if true. Finally we thank N. Miller and M. Stover for several related discussions that helped us spotting the mistake.

\bibliographystyle{abbrv}
\bibliography{biblio.bib}

\Addresses

\end{document}